\documentclass[a4paper,12pt,reqno,oneside]{amsart}

\usepackage{setspace} 
\usepackage{typearea} 

\usepackage{amscd,amsmath,amssymb,verbatim}

\usepackage{ifthen}
\usepackage{xr-hyper}
\usepackage{graphicx}
\usepackage{paralist,cite}
\usepackage[colorlinks,backref,final,bookmarksnumbered,bookmarks]{hyperref}
\usepackage[backrefs]{amsrefs}
\usepackage{tikz-cd}

\newcommand{\arxiv}[2][]{\ifthenelse{\equal{#1}{}}
{\href{http://arxiv.org/abs/#2}{\tt arXiv:#2}}
{\href{http://arxiv.org/abs/math/#2}{\tt arXiv:math.#1/#2}}}

\makeatletter
\renewcommand\subsection{\@startsection
{subsection}{2}{0cm} 
{-\baselineskip}     
{0.5\baselineskip}   
{\sffamily}} 
\makeatother

\theoremstyle{plain}
\newtheorem{theorem}{Theorem}[section]
\newtheorem{lemma}[theorem]{Lemma}
\newtheorem{corollary}[theorem]{Corollary}
\newtheorem{proposition}[theorem]{Proposition}

\newtheorem{maintheorem}{Theorem}

\theoremstyle{definition}
\newtheorem{example}[theorem]{Example}

\newtheoremstyle{remark}
{}{}{}{}{\itshape}{}{ }{\thmname{#1}\thmnumber{ \itshape #2.}}
\theoremstyle{remark}
\newtheorem{remark}[theorem]{Remark}

\newtheoremstyle{concise}
{}{}{}{}{\bfseries}{}{ }{\thmnumber{#2.}\thmnote{ #3.}}
\theoremstyle{concise}
\newtheorem{definition}[theorem]{}

\newenvironment{roster}[1][0]
{

\begin{enumerate}\setcounter{enumi}{#1}}
{\end{enumerate}}

\def\incl{\subset} 
\def\N{\mathbb{N}} \def\R{\mathbb{R}} \def\Z{\mathbb{Z}} 
\def\x{\times}\def\but{\setminus} \def\eps{\varepsilon} \def\phi{\varphi}
\def\emb{\hookrightarrow} 
\def\invlim{\varprojlim} 
\def\xr#1{\xrightarrow{#1}} \def\xl#1{\xleftarrow{#1}} \renewcommand{\:}{\colon}
\def\imp{$\Rightarrow$}  
 
\DeclareMathOperator{\id}{id}

\DeclareMathOperator{\st}{st}

\def\ost{\mathop{\mathring{\text{st}}}}

\def\Cel{\lfloor} \def\Cer{\rfloor}
\def\Fll{\lceil} \def\Flr{\rceil}
\def\downscale#1{\mathchoice{\raisebox{1pt}{$\scriptstyle#1$}}
{\raisebox{1pt}{$\scriptstyle#1$}}{\raisebox{.5pt}{$\scriptscriptstyle#1$}}
{{\scriptscriptstyle#1}}}
\def\cel{\downscale\Cel} \def\cer{\downscale\Cer}
\def\fll{\downscale\Fll} \def\flr{\downscale\Flr}
 
\def\bigfl{\Fll} \def\bigfr{\Flr}
\def\bydef{\mathrel{\mathop:}=}
\def\cubvert{{\raisebox{-0.35ex}{$\hskip-1pt\scriptstyle\urcorner$}}}

\begin{document}
\title{Infinite-dimensional uniform polyhedra}
\author{Sergey A. Melikhov}
\address{Steklov Mathematical Institute of Russian Academy of Sciences,
ul.\ Gubkina 8, Moscow, 119991 Russia}
\email{melikhov@mi-ras.ru}

\begin{abstract}
Uniform covers with a finite-dimensional nerve are rare (i.e., do not form
a cofinal family) in many separable metric spaces of interest.
To get hold on uniform homotopy properties of these spaces, a reasonably
behaved notion of an infinite-dimensional metric polyhedron is needed;
a specific list of desired properties was sketched by J. R. Isbell in
a series of publications in 1959--64.
In this paper we construct what appears to be the desired theory of uniform
polyhedra; incidentally, considerable information about their metric and
Lipschitz properties is obtained.
\end{abstract}

\maketitle

\section{Introduction}

By a (non-uniform) {\it polyhedron} we mean a triangulable space, that is, a space homeomorphic to 
the traditional geometric realization $||K||$ of an (abstract) simplicial complex $K$ in 
the {\it metric} topology.%
\footnote{In more detail, $||K||$ is the union of the convex hulls $\left<\sigma\right>$ of all simplexes 
$\sigma\in K$ in the vector space $\R[V]$ of all formal $\R$-linear combinations of vertices of $K$.
The metric topology is the topology of the subspace of $\R[V]$ in either the $l_1$ or the $l_2$ or 
the $l_\infty$ metric or the topology of the subspace of the $V$-indexed product of copies of $\R$
(see \cite{M3}*{\ref{book:4topologies}} for a proof that all these topologies coincide on $||K||$).
Usually polyhedra are meant to be endowed with a fixed PL structure, that is, a family of compatible 
triangulations, but we do not need this for our purposes here.}
Although polyhedra are usually harder to deal with than CW complexes, a number of very basic facts about them are
known since the middle of the 20th century: polyhedra are ANRs (see \cite{M3}*{\ref{book:poly-ANR}}, 
\cite{Hu}, \cite{Sa2}) and every ANR is homotopy equivalent to a polyhedron (see \cite{M3}*{\ref{book:anr-he}}, \cite{Sa2});
in fact, a metrizable space is an ANR if and only if it is $\eps$-homotopy dominated by a polyhedron 
for each $\eps>0$ (see \cite{M3}*{\ref{book:hanner}}, \cite{Hu}, \cite{Sa2}).
Another classic result is that the barycentric subdivision of a simplicial complex $K$ is {\it admissible} in 
the sense that it determines the same topology on $||K||$ (see \cite{M3}*{\ref{book:barycentric-admissible}}, 
\cite{Le}, \cite{He}).%
\footnote{However, for $K$ that is not locally finite-dimensional, the metric topology on $||K||$ is strictly 
finer than the topology (in fact, also metrizable) with base consisting of the open stars of the simplexes of 
the iterated barycentric subdivisions $K,K',K'',\dots$ \cite{Le}.}

More recently, a number of basic results of PL topology have been generalized to arbitrary polyhedra (including 
locally infinite dimensional ones) by K. Sakai and his collaborators.
A simplicial subdivision of $K$ is admissible if and only if its set of vertices has no cluster points in $||K||$ 
\cite{MiSa} (see also \cite{M3}*{\ref{book:subdivision}}, \cite{Sa2}).
(Of course, it may possibly have cluster points only when $K$ is not locally finite.)
Every open cover of $||K||$ is refined by the cover consisting of the open stars of vertices of some 
admissible subdivision of $K$ \cite{Sa1} (see also \cite{M3}*{\ref{book:henderson}}, \cite{Sa2}).
Every continuous map $||K||\to||L||$ can be arbitrarily closely approximated by a simplicial map 
between some admissible subdivisions \cite{Sa2} (see also \cite{M3}*{\ref{book:simp-appr}}).
Metric topology respects the product of simplicial complexes with linearly ordered vertices \cite{Sa2}
(this last assertion also has a more elegant proof, as it follows from Theorem \ref{3.1}(a) and 
Proposition \ref{metric topology} of the present paper; see also \cite{M3}*{Theorem \ref{book:product-thm}}).

Still more recently, the theory of polyhedra was further developed by the author, who obtained, in
particular, some local characterizations of polyhedra \cite{M3}*{Chapter \ref{book:ch-polyhedra}}.
\bigskip

The purpose of the present paper is to develop a theory of polyhedra in the uniform setting, thereby 
fulfilling what seems to be the basic aims of a research program initiated by J.~ R.~ Isbell in the 1950s.
Specific terms of his program will be reviewed at the end of this introduction.
For now let us only note that one certainly wants {\it uniform polyhedra} to be metrizable uniform spaces
(i.e.\ equivalence classes of metric spaces up to uniform homeomorphism) whose underlying topological spaces 
are polyhedra.
And one certainly wants them to have reasonable properties with respect to the uniform category, and
not the topological category (so, uniformly continuous maps instead of continuous maps, uniform covers 
instead of open covers, uniformly discrete sets instead of discrete sets, uniform ANRs instead of ANRs,
and so on).

In the present paper we will deal only with countable polyhedra, or equivalently, separable metric spaces.
We refer to \cite{M2} for all details about the uniform category and also to \cite{M1}*{Chapter 2} for all 
combinatorial prerequisites.

\subsection{Basic ideas}
Since topological and uniform notions agree on compact spaces (recall that continuous maps with compact 
domain are uniformly continuous), the theory of uniform polyhedra is not supposed to say anything new 
about compact polyhedra.

Finite-dimensional uniform polyhedra were studied by Isbell, who used the $l_\infty$ metric on 
the traditional geometric realization $||K||$ and showed that the resulting uniform structure 
satisfies all basic reasonable properties that one could expect \cite{I0}, \cite{I1}, \cite{I3}.
In particular, Isbell's finite-dimensional uniform polyhedra are uniform ANRs \cite{I1}*{1.9}, \cite{I3}*{V.15}
(see \cite{M2} concerning uniform ANRs) and their uniform structure is combinatorially controlled.%
\footnote{Namely, if $C_n$ denotes the cover of $||K||$ by the open stars of vertices in the $n$th
barycentric subdivision and $D_n$ denotes the cover of $||K||$ by all open $2^{-n}$-balls in the $l_\infty$ 
metric, then every $C_n$ is refined by some $D_m$, and every $D_n$ is refined by some $C_m$,
as long as $K$ is finite dimensional \cite{I1}*{2.1}, \cite{I3}*{IV.6}.}
It is also easy to see that for an $n$-dimensional $K$, the $l_\infty$ metric of $\R[V]$ is uniformly 
equivalent, when restricted to $||K||$, to the $l_1$ and $l_2$ metrics, and when the vertex set $V$ is countable 
also to any metric inducing the uniform structure of the product.%
\footnote{Any two points of $||K||$ lie in a $(2n+1)$-simplex that is the convex hull of a finite set 
$F\subset V\subset\R[V]$.
But any two of the aforementioned metrics are well-known to be Lipschitz equivalent when restricted to
the finite dimensional subspace $\R[F]$, with both Lipschitz constants depending only on $n$ 
(see \cite{M2}*{\S\ref{metr:product}}).}

Unfortunately, Isbell's uniformity does not even remotely satisfy most of the basic reasonable properties in 
the infinite-dimensional case, as Isbell himself noted on a number of occasions (see the end of this introduction).

A key difficulty in the infinite-dimensional case can be seen from the following example.
Let $\Delta^n$ be the standard $n$-simplex in $\R^{n+1}$, that is the intersection of the positive octant 
$x_i\ge 0$, $i=1,\dots,n+1$, with the hyperplane $\sum x_i=1$.
Then $\Delta^n$ has a constant (i.e.\ independent of $n$) edge length in Euclidean, or $l_1$, or $l_\infty$ metric.
However, the distance from the barycenter of $\Delta^n$ (at $(\frac1{n+1},\dots,\frac1{n+1}$)) to the barycenter 
of a facet of $\Delta^n$ (at $(0,\frac1n,\dots,\frac1n)$) tends to zero as $n\to\infty$, in either metric.
Even if one replaces $\Delta^n$ with the standard spherical $n$-simplex, that is the intersection of 
the positive octant with the unit sphere $\sum x_i^2=1$, the great circle distance (in the usual sperical metric,
induced from the Euclidean metric in $\R^{n+1}$) from the barycenter of the spherical simplex 
(at $(\frac1{\sqrt{n+1}},\dots,\frac1{\sqrt{n+1}})$) to the barycenter of its facet 
(at $(0,\frac1{\sqrt n},\dots,\frac1{\sqrt n})$) still tends to zero as $n\to\infty$, although the edge length 
remains constant.

In order to endow a simplicial complex $K$ with a combinatorially controlled uniform structure, one needs 
a metric on $K$ and a sequence of successive subdivisions $K^{(i)}$ of $K$ such that the open stars of vertices 
of each $K^{(i)}$ have diameters $\le d_i$, where the $d_i$ tend to $0$ as $i\to\infty$, and form 
a uniform cover (i.e.\ a cover that is refined, for some $\lambda>0$, by the cover by all $\lambda$-balls).
However, if the complexes $K^{(i)}$ are to be simplicial, one can hardly construct such a metric on $K$; 
certainly, the simplices of any $K^{(i)}$ cannot be Euclidean, or be endowed with the $l_1$ or $l_\infty$ metric 
via an affine homemorphism with the standard simplex, as one can see from the above example.%
\footnote{Indeed, assuming the contrary, let $\sigma$ be an $n$-simplex of such a $K^{(i)}$.
Then the set $B_\sigma$ of the barycenters of the facets of $\sigma$ does not lie in the open star of 
any vertex of $K^{(i)}$. 
On the other hand, since the edges of $K^{(i)}$ are of lengths $\le d_i$, the given affine 
homeomorphism $\Delta^n\to\sigma$ is $d_i$-Lipschitz, and therefore
the diameters of $B_\sigma$ for $n$-dimensional simplices $\sigma$ tend to zero as $n\to\infty$.
So the cover of $K$ by the open stars of vertices of $K^{(i)}$ fails to be uniform, as long as $K$ 
is infinite-dimensional.}
Clearly, complexes of spherical simplexes endowed with the standard spherical metric fare no better.

Our point of departure is to use the ``canonical subdivision'' of 
\cite{M1}, which when applied to a simplicial complex produces 
a cubical complex (versions of this construction are well-known in 
geometric group theory and in topological combinatorics).
Each cube is then endowed with the $l_\infty$ metric, and by glueing
these cubes together we get the uniform geometric realization of 
a simplicial complex.
Canonical subdivision also applies to a cubical complex, and in this case
it is the obvious silly procedure producing another cubical complex by 
cutting every $n$-cube into $2^n$ of $n$-cubes.
Thus for the $i$th canonical subdivision $K^{\#i}$ of the simplicial 
complex $K$, the open stars of vertices of $K^{\#i}$ have diameters 
$\le 2^{2-i}$, and form a basic uniform cover of $K$.
It is not hard to check that this uniform structure induces 
the usual metric topology (see Proposition \ref{metric topology}).

While this simple idea already suffices for many practical purposes, it also
brings some deep combinatorial complications.
At least at a first glance, the ``cubical'' uniform structure appears to be in
an endless conflict with basic PL constructions such as cone, join and
mapping cylinder, which are manifestly ``non-cubical''.
It is this conflict that in a sense is the main theme  of the present paper.

\subsection{New combinatorics of combinatorial topology}
Our first step towards resolution of the conflict is to further subdivide
the cubes into simplices, without introducing new vertices.
(Of course, this subdivision is ``handicapped'' in that its open stars of 
vertices do not form a uniform cover.)
These simplices are now asymmetric, but each comes with a natural total order
on its vertices, for it is isometric to the ``standard skew $n$-simplex''
$\{(x_1,\dots,x_n)\mid 0\le x_1\le\dots\le x_n\le 1\}\subset\R^n$ for some $n$
with the $l_\infty$ metric.%
\footnote{This natural subdivision of an $n$-cube into $n!$ of $n$-simplices 
was described by H. Freudenthal \cite{Fr2}*{\S1}.
The case $n=3$ was also known to Liu Hui in the 3rd century AD, and the skew 
$n$-simplices (without the $l_\infty$ metric) were also known to L. Schl\"afli
in the 19th century, who called them ``orthoschemes''.}
One can now try to use these asymmetric metric simplices as separate 
building blocks, making sure their vertex orderings agree whenever they 
overlap.
Thus let us call a partial ordering of the vertices of a simplicial complex
{\it compatible} if it induces a total ordering on each simplex, or equivalently
if vertices connected by an edge are comparable.

If $K$ is a simplicial complex with a compatible partial ordering of vertices,
or more generally any $\Delta$-set (=semi-simplicial set in a modern terminology),
then we can canonically endow every simplex of $K$ with the standard skew
metric; as long as there are no loops in the $1$-skeleton (in the case of
a $\Delta$-set), this can be shown to extend to a path metric on the entire $K$.
Unfortunately, this metric is generally plagued by the very same problem that
we intended to avoid: it turns out that for each $\eps>0$ there exists an $n$
such that every point of the standard skew $n$-simplex is $\eps$-close to
some point in its boundary (see Example \ref{counterexample}).

For this reason we restrict our attention to flag complexes (a simplicial complex
$K$ is called ``flag'' if every subcomplex of $K$ isomorphic to the boundary
of a $d$-simplex, $d>1$, lies in an actual $d$-simplex in $K$).
A flag simplicial complex with a compatible partial ordering of vertices will
be called a {\it preposet}.
The $1$-skeleton of a preposet carries the structure of an acyclic digraph (=a directed graph 
with no directed cycles).
Let us note that the binary relation $E\subset V\x V$ of an acyclic digraph $G=(V,E)$ is a generalization
of strict partial order, capturing the notion of possibly non-transitive subordination (``the vassal of 
my vassal is not necessarily my vassal'').
Conversely, the flag complex spanned by an acyclic digraph is a preposet, with vertices partially ordered 
by the relation: $v\le w$ if and only if there exists a directed path from $v$ to $w$ in the $1$-skeleton.

A special case of a preposet is (the order complex of) a poset.
Thus we arrive at the uniform geometric realization of a poset $P$.
If $P$ happens to be the poset of all nonempty faces of a simplicial complex $K$, this simply brings us back 
to the uniform geometric realization of $K$, as discussed above.

\subsection{Main results}

All posets and other combinatorial objects will be assumed countable throughout.

We consider three notions of uniform geometric realization of a poset (in particular,
of the poset of nonempty faces of a simplicial or cubical complex) by a separable
metrizable uniform space:
\begin{itemize}
\item by constructing an explicit embedding into the unit cube of the functional
space $c_0$ (generalizing a construction of Shtan'ko--Shtogrin \cite{SS});

\item by gluing together the standard skew simplices via quotient uniformity
(akin to the traditional geometric realization of a simplicial set
via quotient topology);

\item by gluing together the standard skew simplices via path metric (like
in geometric polyhedral complexes used in geometric group theory).
\end{itemize}
The following follows from Theorems \ref{isometry} and \ref{CW}:

\begin{maintheorem} All three notions of geometric realization are equivalent.
\end{maintheorem}

We also show that for finite dimensional simplicial complexes our geometric realization 
is equivalent to Isbell's (Corollary \ref{Isbell's metric2}).
The geometric realization of a locally infinite dimensional poset (even
a simplicial complex) may fail to be complete; however it is ``homotopy complete''
in the sense that the completion can be instantaneously taken off the remainder
by a homotopy (Lemma \ref{3.9}).

Next, geometric realization is promoted to a functor from monotone maps between posets
to uniformly continuous maps between separable metrizable uniform spaces, which
is shown to preserve pullbacks and those pushouts that remain pushouts upon
barycentric subdivision (Theorem \ref{pullback-pushout}).
In particular, the functor respects joins, and mapping cylinders of simplicial maps.
Here the join of posets can refer to any of the two well-known distinct notions,
and the join and mapping cylinder of metrizable uniform spaces are defined
in \cite{M2}.

To include arbitrary pushouts, such as mapping cylinders of general monotone maps,
the geometric realization functor has to be extended to preposets.
Let us emphasize that the mapping cylinder fails to be a poset (and so is only
a preposet) for some monotone maps between posets that arise naturally in practice:
\begin{itemize}
\item the diagonal embedding $P\to P\x P$ (see
\cite{M1}*{Example \ref{comb:diagonal}});
\item approximations to uniformly continuous maps
(see Theorem \ref{monotone approximation});
\item bonding maps between nerves of covers
(see Example \ref{bad bonding map} and Lemma \ref{6.1}).
\end{itemize}
One can, however, emulate the mapping cylinder of a monotone map between posets
by a certain poset, without changing the relative uniform homotopy equivalence
class of the geometric realization (Corollary \ref{hmc}).

Unfortunately, preposets do not quite live up to our expectations.
It turns out that there exists a preposet $X$ whose geometric realization
is not uniformly locally contractible; worse yet, it contains essential
loops of arbitrarily small diameters (Example \ref{counterexample2}).
In fact, $|X|$ is not uniformly homotopy equivalent
to the geometric realization of any poset (Theorem \ref{CCP-homotopy}).

\begin{maintheorem} The geometric realization of every poset $P$ is uniformly locally
contractible.
\end{maintheorem}

This is saying that for each $\eps>0$ there exists a $\delta>0$ such that every two
$\delta$-close uniformly continuous maps from an arbitrary metric space into $|P|$
are uniformly $\eps$-homotopic with values in $|P|$ (Theorem \ref{LCU}).

On the other hand, there exists a poset $Y$ whose geometric realization is not
a uniform ANR; worse yet, for each $\eps>0$ it contains an embedded sphere of
some dimension that is essential in $|Y|$, but null-homotopic in
the $\eps$-neighborhood of $|Y|$ in some fixed uniform ANR (Example
\ref{counterexample3}).
In fact, $|Y|$ is not uniformly homotopy equivalent to
the geometric realization of any {\it conditionally complete} poset
(Theorem \ref{Hahn-homotopy}), that is, a poset where every nonempty set has either
the least upper bound or no upper bound whatsoever.

Conditionally complete posets include (the posets of nonempty faces of)
simplicial and cubical complexes but exclude for instance (those of) some
simplicial pseudo-complexes in the sense of Hilton--Wylie (where each simplex
is embedded, but different simplices may have more than one face in common).
Fortunately, the geometric realization of a conditionally complete poset
does turn out to be a uniform ANR (Corollary \ref{CCP-ANR}).
It is the geometric realizations of conditionally complete posets that we
call {\it uniform polyhedra}.

\begin{maintheorem}\label{main1} Uniform polyhedra are uniform ANRs.
\end{maintheorem}

This is saying that whenever a uniform polyhedron is uniformly embedded as
a closed subset of a metrizable uniform space $X$, then it is a uniform retract
of some its uniform neighborhood in $X$.%
\footnote{An examination of the proof reveals that uniform polyhedra endowed with
the path metric are in fact Lipschitz ANRs.
It seems unlikely that they can be remetrized as $1$-Lipschitz ANRs.}

\begin{example} Let $S^1$ be (the poset of nonempty faces of) 
$\partial\Delta^1$, the boundary of the $1$-simplex, with a
basepoint chosen at some vertex.

The countable product $\prod_\N |S^1|$ is not even locally 
$1$-connected, yet the weak product $\prod^w_\N S^1$ (see 
\cite{M1}*{\S\ref{comb:weak join}}) is a countable cubical complex, 
whence $|\prod^w_\N S^1|$ is a uniform ANR.
In fact, $\prod_\N|S^1|$ and $|\prod^w_\N S^1|$ are not even weakly
homotopy equivalent: $\pi_1(\prod_\N|S^1|)\simeq\prod_\N\Z$, whereas
$\pi_1(|\prod_\N S^1|)\simeq\bigoplus_\N\Z$.

Of course, if we didn't confine our attention to countable posets, 
a third option would be $|\prod_\N S^1|$, the full $l_\infty$ space 
of functions $\N\to|S_1|$.
This is also a uniform ANR (albeit an inseparable one), and
$\pi_1(|\prod_\N S^1|)$ is the group of all bounded functions 
$\N\to\Z$, an answer that also occurs in the contexts of simplicial 
sets (see \cite{Th}*{p.\ 307}) and Alexandroff spaces (see \cite{Ra}*{\S4}).
\end{example}

Theorem \ref{main1} is arguably the hardest result of the paper.
The case of simplicial complexes is somewhat easier
(Theorem \ref{simplicial ANR}), and actually reduces to the case of
``cubohedra'', which was already treated in
\cite[Theorem \ref{metr:cubohedron}]{M2}.
(A cubohedron is a cubical complex that is a subcomplex of the standard cubical
lattice in $c_0$.)
The finite-dimensional simplicial case is equivalent to a result of Isbell
\cite[Theorem 1.9]{I1}, taking into account that his geometric realization
is uniformly homeomorphic to ours in the case of a finite-dimensional
simplicial complex.

Simplicial and/or cubical complexes alone do not form a closed theory with
respect to a sufficient supply of operations that work uniformly.
The mapping cylinder of a simplicial map, while being a poset, is certainly
not a simplicial complex in general.
Worse yet, it need not even be a conditionally complete poset
(\cite[Example \ref{comb:non-CCP MC}]{M1}).
This is arguably the single most important deficiency of our theory; but up to
uniform homotopy, there is a remarkable workaround.
The {\it thickened mapping cylinder} $TMC(f)$ of a monotone map $f\:P\to Q$ between
posets is a natural subset of the join $P*Q$, related to the graph of $f$
(see \cite{M1}).
If $P$ and $Q$ are conditionally complete, then so is $TMC(f)$, and up to relative
uniform homotopy equivalence of geometric realizations it is the same as the usual
mapping cylinder (Theorem \ref{tmc2}).
Using this, we obtain the following result (Theorem \ref{Mather trick}):

\begin{maintheorem}
If $X$ is a uniform ANR, then $X\x\R$ is uniformly homotopy equivalent to
a uniform polyhedron.
\end{maintheorem}

We also establish what appears to be the ultimate inverse limit representation
theorem, and a uniform analogue of Hanner's characterization of ANRs
(Theorems \ref{inverse limit} and \ref{domination}):

\begin{maintheorem}\label{main2}
(a) Every separable metrizable complete uniform space is the limit of
an inverse sequence of uniformly continuous maps between geometric realizations
of simplicial complexes.

(b) A separable metrizable uniform space is a uniform ANR
if and only if it is uniformly $\eps$-homotopy dominated by the geometric
realization of a simplicial complex for each $\eps>0$.
\end{maintheorem}

Weaker forms of (a) and (b), with cubohedra in place of simplicial complexes,
are contained in the author's previous paper \cite[Corollary \ref{metr:Hanner2} and
Theorem \ref{metr:intersection of cubohedra2}]{M2}.
The present versions, being based on nerves of uniform covers rather than uniform
neighborhoods in $c_0$, have the advantage of greater flexibility, so they can
be straightforwardly adapted, for instance, to equivariant contexts.

The compact case of (a) is due to Lefschetz \cite{Le2}, who elaborated on Alexandroff's 
earlier work (see \cite{Al}).
The uniformly finitistic case of (a) is equivalent to a result of
Isbell \cite[Theorem V.34]{I3} (see also \cite[Lemma 1.6]{CI}), though he made 
a minor mistake in his proof, which
can be corrected as we show in the proof of Theorem \ref{inverse limit}.
The uniformly finitistic case of the ``only if'' direction in (b)
is equivalent to another result of Isbell \cite[7.3]{I2}.
A topological version of (a) was also proved by Isbell \cite{I6}*{Corollary 3.7}
(see also \cite[Theorem 3]{I4}, \cite[Lemma 1.6]{CI}); another proof (with possibly
infinite-dimensional polyhedra) is found in \cite{Sa2}*{4.10.11}.

\subsection{Isbell's problem}

Let us now discuss how our results address Isbell's research program.
One attempted formulation appears in his book ``Uniform spaces'':

\medskip
\begin{center}
\parbox{14.5cm}{\small
``{\bf Research Problem} $B_2$.
{\sc Infinite-dimensional polyhedra.} There is a large problem here, namely
the systematic investigation of topological and uniform realizations
of abstract simplicial complexes. One important paper in the literature
(Dowker [1952]) has examined this problem, not from a categorical viewpoint.
Dowker's work tends to confirm, what many successful applications suggest,
that for topology J. H. C. Whitehead's realization by CW-complexes has
strong claims to preference. Its definition is as simple as could be: [...]
But Dowker's work highlights the point that the suitability of CW-complexes
for homology and homotopy is not conclusive; many realizations are topologically
distinct but homotopy equivalent.

By now substantial experience in uniform spaces supports the pretensions
of uniform complexes, in the finite-dimensional case only. (In any case they
are homotopy equivalent (topologically) with CW-complexes; Dowker [1952].)
In general they are not satisfactory, e.g.\ because they lack subdivisions.
One can save the subdivisions, or any sufficiently narrow requirement, by
tailoring a definition to fit. (Kuzminov and \v Svedov [1960] define a realization
for which IV.6 [the covers by the stars of vertices in iterated barycentric
subdivisions form a basis of the uniformity] is always valid; but all their
applications are in the finite-dimensional case.) The real problem holding up
progress is, what applications can be made of infinite-dimensional polyhedra in
the general theory of uniform spaces? It would probably be beside the point to
carry out a formal investigation of realizations with no specific applications
in mind.'' \cite{I3} (1964)}
\end{center}

\medskip
Comments: (i) As stated, the problem is quite vague, but some clarification
on what kind of infinite-dimensional uniform complexes are sought here can be
inferred from Isbell's previous comments in his earlier papers, quoted below.

(ii) The covers by the stars of vertices in iterated {\it canonical}
subdivisions do form a basis of the uniformity of our uniform polyhedra
(see Theorem \ref{canonical}).
In fact, using the canonical subdivision we are able to extend Brouwer's 
simplicial approximation theorem for finite complexes to the infinite 
case (Theorem \ref{monotone approximation}).

(iii) One possible reason for the widespread preference for the CW topology could be Milnor's 1959 theorem 
that the space of maps from a pair of finite CW complexes to a pair of arbitrary CW-complexes (with 
the compact-open topology) is homotopy equivalent to a CW-complex \cite{Mil}.
(At least, this theorem was, in Milnor's own words in the very beginning of his paper, ``intended as 
propaganda for'' the class ``of all spaces which have the homotopy type of a CW-complex''.)
But the metrizable viewpoint actually fares even better in this respect: the space of maps from a pair
of compacta to a pair of ANRs (with the compact-open topology) is itself an ANR
(see \cite{Sa2}*{6.2.10(6) and 6.1.9(8)}).
In particular, the space of maps from a pair of finite polyhedra to a pair of arbitrary polyhedra
is homotopy equivalent to a polyhedron.

In the uniform setting, we have a still better result.
The space of uniformly continuous maps from a pair of metrizable uniform spaces (possibly non-compact!)\ to 
a pair of uniform ANRs (with the uniformity of the sup metric) is a uniform ANR 
\cite{M2}*{Theorem \ref{metr:A.3''}}.
In particular, the space of uniformly continuous maps between pairs of arbitrary uniform polyhedra 
is uniformly homotopy dominated by a uniform polyhedron, and becomes uniformly homotopy equivalent 
to a uniform polyhedron upon crossing with $\R$.

\medskip
\begin{center}
\parbox{14.5cm}{\small
``It should be noted that the theorem [that Isbell's finite-dimensional uniform
simplicial complexes are complete uniform ANRs] as stated is trivially false for
arbitrary uniform complexes, since some of them are incomplete.
It is false for many complete ones also. It seems likely that
strong results might be gotten by using some suitable uniformity for a complex,
different from the one defined by $\max|x_\alpha-y_\alpha|$, though not
necessarily different for finite-dimensional complexes.'' \cite{I1} (1959)
}
\end{center}

\medskip
Comments: indeed, with our adjusted uniformity, the theorem is now extended to
infinite-dimensional simplicial complexes (Theorem \ref{simplicial ANR}), apart
from the completeness.

Isbell only considered complete uniform ANRs, as well as their non-metrizable
generalization, called ANRUs.
Indeed he showed that completeness is forced by including certain non-metrizable
spaces in the setup.
However, if one works in the category of metrizable uniform spaces, then there
is no such restriction, as noticed independently by G. L. Garg and N. T. Nhu in
the 1970s (see \cite[\S\ref{metr:uniform-anrs}]{M2}).
The resulting theory of (possibly non-complete) uniform ANRs has been developed
only recently but turns out to be very flexible \cite{M2}.
These uniform ANRs are homotopy complete (i.e.\ the completion can be
instantaneously taken off the remainder)
\cite[Theorem \ref{metr:uniform ANR}]{M2}, which for all practical purposes
makes them just as easily manageable as if they were complete.

\medskip
\begin{center}
\parbox{14.5cm}{\small
``I should like to repeat the remark from [\cite{CI} and \cite{I1}]
that the uniform complexes are clearly not the right concept for the
infinite-dimensional case.
The finite-dimensionality in 7.2 [that every uniformly finitistic complete
uniform space $X$ is an inverse limit (i) of finite-dimensional uniform simplicial
complexes, (ii) of inverse limits of nerves of appropriate uniform covers of $X$]
and 7.3 [that every uniformly finitistic ANRU is uniformly $C$-homotopy dominated by 
a finite dimensional uniform simplicial complex for each uniform cover $C$] may
very likely appear for no better reason than that we do not have the right uniformity
for the complexes.'' \cite{I2} (1961)
}
\end{center}

\medskip
Comments: indeed, the said results extend to all separable uniform spaces,
using our adjusted uniformity on (countable) infinite-dimensional simplicial
complexes (Corollary \ref{inverse limit2} and Theorems \ref{inverse limit-general}
and \ref{domination-general}), except that cubical, rather than simplicial,
complexes are needed for the generalization of 7.2(i).

Neither of the results extends to inseparable uniform spaces by the negative solution
to Isbell's ``Research Problem $B_3$'', obtained independently by J. Pelant and
E. V. Shchepin in 1975, and clarified recently by A. Hohti (see references in
\cite[\S\ref{metr:finiteness}]{M2}).
Hohti's result is that the unit ball of the inseparable space $l_\infty$ is not
point-finite, i.e.\ does not have a basis of uniform covers that have a finite
multiplicity at each point.

In what follows we assume some degree of familiarity with \cite[Chapter 2]{M1}
and the entire paper \cite{M2}.

\section{Geometric realization via embedding}\label{geometric realization}

Rectilinear geometric realization of a {\it finite} preposet is described
in \cite[\S\ref{comb:examples of posets}]{M1}.
\cite[Lemma \ref{comb:2.6}]{M1} yields a realization of an arbitrary preposet
within some {\it combinatorial} simplex.
This however avoids the issue of sensible geometric realization of
the simplex itself.
Let us now address it.
In the case of finite simplicial complexes, the following construction
yields essentially the same result as in \cite{SS}.
(The author found it while being unaware of \cite{SS}.)

\begin{definition}[Geometric realization]
If $S$ is a set, the functional space $[0,1]^S$ of all maps $f\:S\to [0,1]$ is
endowed with the metric $d(f,g)=\sup_{s\in S}|f(s)-g(s)|$.
Note that the underlying uniform space of $[0,1]^S$ is just $U(S,[0,1])$,
where $S$ is endowed with the discrete uniformity.
The subset $\{0,1\}^S$ of $[0,1]^S$ may be identified with the power set
$\text{\it 2}^S$, by associating to every subset $T\incl S$ its {\it characteristic
function} $\chi_T\:S\to [0,1]$, defined by $\chi_T(T)=1$ and $\chi_T(S\but T)=0$.
Note that if $S$ is finite, $[0,1]^S$ is just the usual $|S|$-dimensional cube
with the $l_\infty$ metric, and $\{0,1\}^S$ is the set of its vertices.

Let us recall the embedding of a poset into a simplex given by
\cite[Lemma \ref{comb:2.6}(b)]{M1}.
Given a poset $P=(\mathcal P,\le)$, we identify every $p\in P$ with the cone
$\fll p\flr$, viewed as an element of
${\it 2}^{\mathcal P}=\{0,1\}^{\mathcal P}\incl [0,1]^{\mathcal P}$.
Then the {\it geometric realization} of $P$ is a subspace $|P|\incl [0,1]^{\mathcal P}$,
defined to be the union of the convex hulls of all nonempty finite chains of $P$.
Note that the cube vertex at the origin, $\{0,0,\dots\}$ is never in $|P|$
since $\fll p\flr$ is never empty.

More generally, given a preposet $P=(\mathcal P,\prec)$, following
\cite[Lemma \ref{comb:2.6}(a)]{M1}
we inject it into its transitive closure $\left<P\right>=(\mathcal P,\prec\!\prec)$,
and identify every $p\in P$ with the cone $\fll p\flr_{\prec\!\prec}$ in
the transitive closure, viewed again as an element of
${\it 2}^{\mathcal P}=\{0,1\}^{\mathcal P}\incl [0,1]^{\mathcal P}$.
Then the {\it geometric realization} of $P$ is a subspace $|P|\incl [0,1]^{\mathcal P}$,
defined to be the union of the convex hulls of all nonempty finite chains of $P$.
Every chain of $P$ is a chain of $\left<P\right>$, hence $|P|\incl |\left<P\right>|$.

Since $[0,1]^{\mathcal P}$ is complete, the closure $\overline{|P|}$ of $|P|$
in $[0,1]^{\mathcal P}$ is uniformly homeomorphic to the completion of $|P|$.
Note that each convex hull in the definition of $|P|$ is compact, and
therefore separable.
Hence if $P$ is countable, $|P|$ is separable; consequently $\overline{|P|}$ is
a Polish uniform space, that is, a separable metrizable complete uniform space.
\end{definition}

\begin{definition}[Generalized geometric realization]\label{generalized gr}
Let $P$ be a preposet, and fix an injection $j\:P\to 2^S$ for some $S$.
The underlying set $\text{\it 2}^S$ of the poset $2^S$ is identified, as
before, with $\{0,1\}^S\incl[0,1]^S$.
Then the {\it geometric $j$-realization} of $P$ is a subspace $|P|_j\incl [0,1]^S$,
defined to be the union of the convex hulls of the $j$-images of all nonempty
finite chains in $P$.
We note four basic examples:

$\bullet$ The injection $j_P$ of \cite[Lemma \ref{comb:2.6}]{M1} yields
the {\it standard} geometric realization $|P|_{j_P}=|P|$.

$\bullet$ If $P$ has a least element, then there is a more economical embedding
$j'_P\:P\emb 2^{\partial^*P}$, $p\mapsto\fll p\flr\cap\partial^*P$.
We call $|P|'\bydef |P|_{j'_P}$ the {\it reduced} geometric realization of $P$.

$\bullet$ On the other hand, if $P$ is an atomic poset, then
\cite[Lemma \ref{comb:2.9}]{M1} yields
a more economical embedding $a_P\:P\emb\Delta^{A(P)}$.
We call $|P|^\bullet\bydef |P|_{a_P}$ the {\it atomic} geometric realization of $P$.
Note that $a_P\:\Delta^\Lambda\to\Delta^\Lambda$ is the identity.

$\bullet$ Finally if $P=C^*Q$, where $Q$ is an atomic poset, then $a_Q$ extends
to $a'_P\:P\emb 2^{A(P)}$.
We still call $|P|^\bullet\bydef |P|_{a'_P}$ the {\it atomic} geometric realization of $P$.
Note that $a'_{2^\Lambda}\:2^\Lambda\to 2^\Lambda$ is the identity.
\end{definition}

\begin{definition}[Geometric realization of cone precomplex]\label{q_0}
If $P=(\mathcal P,\le)$ is a cone complex, then $\mathcal P$ is countable,
and the cone $\fll p\flr$ of every $p\in P$ is finite.
Then $j_P(P)\incl\Delta^{\mathcal P}$ lies in the weak $\mathcal P$-simplex
$\Delta^{\mathcal P}_w$, that is the set of all nonempty finite subsets of
$\mathcal P$ (see \cite[\S\ref{comb:weak join}]{M1}).
If $C$ is a finite chain of $\Delta^{\mathcal P}$, its convex hull lies in
$[0,1]^{\sup C}\x\{0\}^{\mathcal P\but\sup C}$.
Consequently, $|\Delta^{\mathcal P}_w|^\bullet$ lies in
$q_0\bydef([0,1],0)^{(\mathcal P^+,\infty)}$, where
$\mathcal P^+=\mathcal P\cup\{\infty\}$ is the one-point compactification of
the discrete space $\mathcal P$.
Since $q_0$ is complete, $\overline{|\Delta^{\mathcal P}_w|^\bullet}$ also lies
in $q_0$.
More generally, if $P$ is a cone precomplex, its transitive closure
$\left<P\right>$ is a cone complex, and $\overline{|P|}$ lies in
$\overline{|\left<P\right>|}\incl\overline{|\Delta^{\mathcal P}_w|^\bullet}\incl q_0$.

It is not hard to see that $q_0$ itself is identified with
$\overline{|2^{\mathcal P}_w|^\bullet}$.
On the other hand, note that $\overline{|\Delta^\N_w|^\bullet}\subsetneqq
\overline{|\Delta^\N|^\bullet}$, since $|\Delta^\N|^\bullet$ includes convex hulls of
finite chains of infinite subsets of $\N$.
In fact, since there are uncountably many of infinite subsets of $\N$,
$|\Delta^\N|^\bullet$ is not separable.

If $P$ is a cone complex that is {\it Noetherian}, i.e.\ contains no infinite
chain (which could only be ascending since all cones are finite), then $|P|$
is closed in $[0,1]^{\mathcal P}$ and hence complete.
It follows that $|P|$ is complete also for every Noetherian cone {\it precomplex}
$P$, that is a preposet whose transitive closure is a Noetherian cone complex.
\end{definition}

\begin{example} If $P$ is a poset and $j\:P\to 2^S$ is an injection but not
an embedding, then $|P|_j$ need not be isometric to $|P|$.
Indeed, let $P$ be the subposet of $2^{\{a,b,c\}}$ with elements
$\emptyset$, $\{a\}$, $\{a,b\}$, $\{a,b,c\}$, $\{c\}$.
Let $j\:P\to 2^{\{a,b,c\}}$ re-embed $\{c\}$ onto $\{b\}$ and fix the other elements.
Let $C$ be the chain $\{\emptyset,\{a\},\{a,b\},\{a,b,c\}\}$ and let $D$
be the chain $\{\emptyset,\{c\},\{a,b,c\}\}$ in $P$.
A point $x\in|C|\incl |P|$ has coordinates $(x_a,x_b,x_c)$ for some numbers
$1\ge x_a\ge x_b\ge x_c\ge 0$, and a point $y'\in |D|$ has coordinates
$(y_1,y_2)$ for some numbers $1\ge y_2\ge y_1\ge 0$.
Then the image $y$ of $y'$ in $|P|$ has coordinates $(y_2,y_2,y_1)$, and the image
$y_j$ of $y'$ in $|P|_j$ has coordinates $(y_2,y_1,y_2)$.
Setting $(x_a,x_b,x_c)=(\frac34,\frac12,\frac14)$ and $(y_1,y_2)=(\frac34,\frac12)$,
we obtain $d(x,y)=\frac12$ and $d(x,y_j)=\frac14$.
\end{example}

\begin{example} Here is a simpler example of the same kind.
Let $P$ be the subposet of $2^{\{a,b\}}$ with elements $\emptyset$,
$\{a\}$ and $\{b\}$.
Let $j\:P\to 2^{\{a,b\}}$ re-embed $\{a\}$ onto $\{a,b\}$ and fix the
other elements.
Let $C$ be the chain $\{\emptyset,\{a\}\}$ and let $D$ be the singleton chain
$\{\{b\}\}$ in $P$.
Let $x\in|C|\incl |P|$ have coordinates $(x_a,x_b)=(0,\frac12)$; then its
image $x_j$ in $|P|_j$ has coordinates $(\frac12,\frac12)$.
The point $y\in|D|=\{y\}$ has coordinates $(1,0)$.
Hence $d(x,y)=1$ and $d(x_j,y)=\frac12$.
\end{example}

\begin{theorem}\label{isometry} If $P$ is a poset and $j\:P\to 2^S$ is
an embedding, then $|P|_j$ is isometric to $|P|$.
\end{theorem}

This trivially implies that if $P$ is a preposet and $j\:P\to 2^S$ is
an injection that factors through an embedding of the transitive closure,
then $|P|_j$ is isometric to $|P|$.

\begin{proof}
We first consider the case where $P$ is the totally ordered $n$-element poset
$[n]=(\{1,\dots,n\},\le)$, where $\le$ has the usual meaning.
To avoid confusion, we consider the standard embedding $j_{[n]}$ of $[n]$ in
$2^{\{1,\dots,n\}}$, that is, $j_{[n]}(i)=\{1,\dots,i\}$.
Then each $j_{[n]}(i)\in 2^{\{1,\dots,n\}}$ is identified with the vertex
$(0,\dots,0,1,\dots,1)$ ($n-i$ zeroes, $i$ ones) of the simplex $|[n]|\incl [0,1]^n$.
Similarly each $P_i\bydef j(i)\in 2^S$ is identified with a point of
$\{0,1\}^S\incl [0,1]^S$.
Then $jj_{[n]}^{-1}$ extends uniquely to an affine map $\Phi_j\:|[n]|\to |[n]|_j$.
A point $x=(x_1,\dots,x_n)\in [0,1]^n$ lies in $|[n]|$ if and only if
$1=x_1\ge x_2\ge\dots\ge x_n\ge 0$.
Let us write $\hat x=\Phi_j(x)$.
It is easy to see that $\hat x(s)=x_i$ for each $s\in P_i\but P_{i-1}$,
where $P_0=\emptyset$, and $\hat x(s)=0$ for $s\notin P_n$.

Given another point $y\in |[n]|$, since each $P_i\but P_{i-1}$ is nonempty,
it follows that $d(\hat x,\hat y)=\sup_{s\in S}|\hat x(s)-\hat y(s)|$ equals
$\max_{i\in [n]} |x_i-y_i|=d(x,y)$.
Thus $\Phi_j$ is an isometry, which completes the proof of the case $P=[n]$.

We now resume the proof of the general case, where $P=(\mathcal P,\le)$ is
an arbitrary poset.
Given a finite chain $C$ in $P$, we may represent $C$ as
the image of the poset $[k]$, where $k$ is the cardinality of $C$, under
the (unique) isomorphism $c\:[k]\to C$.
Consider the compositions $c\:[k]\xr{c}P\incl 2^{\mathcal P}$ and
$jc\:[k]\xr{c}P\xr{j}2^S$.
These extend to the isometries $\Phi_c\:|[k]|\to|[k]|_c\incl |P|$ and
$\Phi_{jc}\:|[k]|\to|[k]|_{jc}\incl |P|_j$.
The compositions $|[k]|_c\xr{\Phi_c^{-1}}|[k]|\xr{\Phi_{jc}}|[k]|_{jc}$ agree
with each other for different $c$, and thus combine into a map
$\Phi_j\:|P|\to |P|_j$ that is an isometry on the convex hull of every finite
chain of $P$.

To complete the proof, it suffices to show that all $x,y\in |P|$ satisfy
$d(\Phi_j(x),\Phi_j(y))=d(x,y)$.
This will follow once we prove that $d(\Phi_j(x),\Phi_j(y))$ does not depend
on $j$.
We may assume without loss of generality that $P$ has the least element $\hat 0$
and the greatest element $\hat 1$; for if $P$ has no least (resp.\ greatest)
element, then $S\notin j(P)$ (resp.\ $\emptyset\notin j(P)$), and therefore $j$
extends to an embedding of $CP$ (resp.\ $C^*P$) in $2^S$ defined by
$\hat 1\mapsto S$ (resp.\ $\hat 0\mapsto\emptyset$).
Let $A$ (resp.\ $B$) be some chain in $P$ containing $\hat 0$ and $\hat 1$,
whose convex hull contains $x$ (resp.\ $y$).
We consider the (unique) isomorphisms $a\:[m]\to A$ and $b\:[n]\to B$,
where $m$ is the cardinality of $A$ and $n$ is the cardinality of $B$.
Thus $a(1)=\hat 0=b(1)$ and $a(m)=\hat 1=b(n)$.

Let $\prec$ be the covering relation of the subposet $A\cup B$ of $P$.
(That is, $x,y\in A\cup B$ satisfy $x\prec y$ iff $x<y$ and there exists
no $z\in A\cup B$ such that $x<z<y$.)
Let $(k_1,l_1),\dots,(k_r,l_r)$ be all pairs in $[m]\x [n]$ such that either
$a(k_i)=b(l_i)$ or $a(k_i)\notin B$, $b(l_i)\notin A$ and $a(k_i)\prec b(l_i)$,
where each $k_i\le k_{i+1}$.
Let $Z\incl [r]$ be the set of indices $i$ such that $a(k_i)=b(l_i)$.
It is easy to see\footnotemark\ that each $k_i<k_{i+1}$ and each $l_i<l_{i+1}$.
\footnotetext{Indeed, suppose that $k_i=k_{i+1}$ (and $l_i\ne l_{i+1}$).
The cases (1) $i,i+1\in Z$; (2) $i\in Z$ and $i+1\notin Z$; (3) $i\notin Z$ and
$i+1\in Z$ are ruled out for trivial reasons.
In the remaining case (4) $i,i+1\notin Z$ we have either
$l_i<l_{i+1}$ or $l_i>l_{i+1}$.
Then either $a(k_{i+1})=a(k_i)<b(l_i)<b(l_{i+1})$ or
$a(k_i)=a(k_{i+1})<b(l_{i+1})<b(l_i)$.
Hence either $a(k_{i+1})\nprec b(l_{i+1})$ or $a(k_i)\nprec b(l_i)$, which
is a contradiction.
Thus $k_i<k_{i+1}$.
Next suppose that $l_i\ge l_{i+1}$.
Then $a(k_i)<a(k_{i+1})\le b(l_{i+1})\le b(l_i)$, so $i\notin Z$.
Hence $a(k_i)<a(k_{i+1})<b(l_{i+1})\le b(l_i)$, so $a(k_i)\nprec b(l_i)$, which is
a contradiction.}
Similarly let $(k'_1,l'_1),\dots,(k'_{r'},l'_{r'})$ be all pairs in $[m]\x [n]$
such that either $a(k'_i)=b(l'_i)$ or $a(k'_i)\notin B$, $b(l'_i)\notin A$ and
$a(k'_i)\succ b(l'_i)$; we may assume that each $k'_i<k'_{i+1}$ and each
$l'_i<l'_{i+1}$.
Let $Z'$ be the set of indices $i$ such that $a(k'_i)=b(l'_i)$.
We note that $k_1=l_1=k'_1=l'_1=1$, $k_r=k'_{r'}=m$ and $l_r=l'_{r'}=n$.

It is easy to see\footnotemark\ that $a(k_i+1)\nleq b(l_{i+1}-1)$ for each $i$.
\footnotetext{Indeed, if $a(k_i+1)=b(l_{i+1}-1)$, then $k_i+1=k_j$ and
$l_{i+1}-1=l_j$ for some $j$; hence $i<j<i+1$, which is a contradiction.
Suppose that $a(k_i+1)<b(l_{i+1}-1)$.
Since $<$ is the transitive closure of $\prec$, there exist
$\kappa\ge k_i+1$ and $\lambda\le l_{i+1}-1$ such that
$a(\kappa)\prec b(\lambda)$.
If $a(\kappa)\in B$, then $k_i<\kappa=k_j$ for some $j$ such that
$b(l_j)=a(k_j)=a(\kappa)<b(\lambda)$.
Hence $l_j<\lambda<l_{i+1}$, and therefore $i<j<i+1$, which is a contradiction.
Thus $a(\kappa)\notin B$, and similarly $b(\lambda)\notin A$.
Hence $k_i<\kappa=k_j$ and $l_{i+1}>\lambda=l_j$ for some $j$; hence $i<j<i+1$,
which is a contradiction.}
Let us write $A_i=ja(i)$ and $B_i=jb(i)$.
Since $j$ is an embedding, we obtain that $A_{k_i+1}\not\incl B_{l_{i+1}-1}$.
On the other hand, since $k_i<k_{i+1}$, we have
$A_{k_i+1}\incl A_{k_{i+1}}\incl B_{l_{i+1}}$; and similarly $A_{k_i}\incl B_{l_{i+1}-1}$.
Thus $A_{k_i+1}\but A_{k_i}$ has a nonempty intersection with
$B_{l_{i+1}}\but B_{l_{i+1}-1}$.
In other words, the set $\Sigma$ of all pairs $(\kappa,\lambda)$ such that
$A_\kappa\but A_{\kappa-1}$ has a nonempty intersection with
$B_\lambda\but B_{\lambda-1}$ includes the set $\Delta$ of all
pairs of the form $(k_i+1,l_{i+1})$.
By symmetry, $\Sigma$ also includes the set $\Delta'$ of all pairs of the form
$(k'_{i+1},l'_i+1)$.

We claim that for each $(\kappa,\lambda)\in\Sigma$ there exist
a $(k,l)\in\Delta$ and a $(k',l')\in\Delta'$ such that $k\le\kappa\le k'$
and $l'\le\lambda\le l$.
Indeed, if $a(\kappa)\le b(\lambda-1)$, then $A_\kappa\incl B_{\lambda-1}$; in
particular, $(\kappa,\lambda)\notin\Sigma$.
If $a(\kappa)$ and $b(\lambda-1)$ are incomparable, let $i$ be the maximal number
such that $\kappa>k_i$.
Then $\kappa\le k_{i+1}$, so $\lambda-1<l_{i+1}$.
Finally, if $a(\kappa)>b(\lambda-1)$, let $i$ be the maximal number such that
$\kappa>k_i$.
We claim that still $\lambda-1<l_{i+1}$.
Suppose on the contrary that $\lambda-1\ge l_{i+1}$.
Then $a(\kappa)>b(\lambda-1)\ge b(l_{i+1})\ge a(k_{i+1})$.
On the other hand, $\kappa\le k_{i+1}$ by our choice of $i$, and so
$a(\kappa)\le a(k_{i+1})$, which is a contradiction.
This completes the proof of the assertion on $(k,l)$; and the assertion on
$(k',l')$ is proved similarly.

We have $x=\Phi_a(\alpha)$ and $y=\Phi_b(\beta)$ for some
$\alpha=(\alpha_1,\dots,\alpha_m)\in |A|$ and some
$\beta=(\beta_1,\dots,\beta_n)\in |B|$.
Let us denote $\Phi_j(x)=\Phi_{ja}(\alpha)$ by $\hat\alpha$ and
$\Phi_j(y)=\Phi_{jb}(\beta)$ by $\hat\beta$.
We have $d(\hat\alpha,\hat\beta)=\sup_{s\in S}|\hat\alpha(s)-\hat\beta(s)|$.
Here $\hat\alpha(s)=\alpha_\kappa$ for each $s\in A_\kappa\but A_{\kappa-1}$
(where $A_0=\emptyset$) and $\hat\alpha(s)=0$ for $s\notin A_m$; similarly,
$\hat\beta(s)=\beta_\lambda$ for each $s\in B_\lambda\but B_{\lambda-1}$
(where $B_0=\emptyset$) and $\hat\beta(s)=0$ for $s\notin B_n$.
Hence $d(\hat\alpha,\hat\beta)=\max_{(\kappa,\lambda)\in\Sigma}
|\alpha_\kappa-\beta_\lambda|$.

Next we recall that $1=\alpha_1\ge\alpha_2\ge\dots\ge\alpha_m\ge 0$
and $1=\beta_1\ge\beta_2\ge\dots\ge\beta_n\ge 0$.
In particular, $k\le\kappa\le k'$ implies
$\alpha_k\ge\alpha_\kappa\ge\alpha_{k'}$; and $l'\le\lambda\le l$ implies
$\beta_{l'}\ge\beta_\lambda\ge\beta_l$.
Hence $\alpha_k-\beta_l\ge\alpha_\kappa-\beta_\lambda\ge\alpha_{k'}-\beta_{l'}$,
which implies $|\alpha_\kappa-\beta_\lambda|\le
\max\{|\alpha_k-\beta_l|,\,|\alpha_{k'}-\beta_{l'}|\}$.
Thus
$$d(\hat\alpha,\hat\beta)=\max_{(k,l)\in\Delta\cup\Delta'}
|\alpha_k-\beta_l|.$$
The right hand side does not depend on $j$; therefore so does the left hand
side, that is, $d(\Phi_j(x),\Phi_j(y))$.
\end{proof}

\begin{corollary}\label{subposet realization} If $Q$ is a subpreposet of
a preposet $P$, then $|Q|$ admits a natural isometric embedding in $|P|$.
\end{corollary}

\begin{proof} If $P$ and $Q$ are posets, then by Theorem \ref{isometry},
$|Q|=|Q|_{j_Q}$ is isometric to $|Q|_j\incl |P|_{j_P}=|P|$, where $j$
is the composition $Q\incl P\xr{j_P}2^{\mathcal P}$.

In the general case, the transitive closure $\left<Q\right>$ is a subposet
of $\left<P\right>$, and it is easy to see that the image of the isometric
embedding $|Q|\incl|\left<Q\right>|\to|\left<P\right>|$ lies in $|P|$.
\end{proof}

\begin{corollary}\label{dual-isometric}
If $P$ is a preposet, $|P^*|$ is isometric to $|P|$.
\end{corollary}

\begin{proof}
First assume that $P$ is a poset.
By Theorem \ref{isometry}, there exists an isometry $\Phi_j\:|P|\to |P|_j$,
where $j\:P\to 2^{\mathcal P}$ is defined by $j(p)=\mathcal P\but\cel p\cer$.
The isomorphism $\phi\:2^{\mathcal P}\to (2^{\mathcal P})^*$ defined by
$\phi(S)=\mathcal P\but S$ extends to a self-isometry $\Phi$ of
$[0,1]^{\mathcal P}$, taking $|P|_j$ onto $|P^*|$.
Indeed we have $\phi j(p)=j_{P^*}(p^*)$, where $j_{P^*}\:P^*\to 2^{\mathcal P}$ is
the standard embedding, $j_{P^*}(p^*)=\fll p^*\flr=\cel p\cer^*$.

If $P$ is a preposet, we apply the above construction to its transitive
closure $\left<P\right>$.
Given a finite chain of $P$, viewed as an embedding
$c\:[n]\to\left<P\right>\incl 2^{\mathcal P}$, we have
$|[n]|_c\incl |P|\incl |\left<P\right>|$.
Clearly, the isometry $|\left<P\right>|\to |\left<P\right>|_j\to|\left<P\right>^*|
=|\left<P^*\right>|$ takes $|[n]|_c$ onto $|[n]^*|_{c^*}$.
\end{proof}

\begin{remark}
We recall that $2^{\mathcal P}_w$ is identified with a closed subposet of
$(I^{\mathcal P}_c)^*$ (see \cite[\S\ref{comb:weak join}]{M1}).
The completed geometric realization $\overline{|I^{\mathcal P}_c|}$ therefore
contains an isometric copy of $\overline{|2^{\mathcal P}_w|}$.
The latter is in turn isometric to the completed {\it atomic} geometric
realization $\overline{|2^{\mathcal P}_w|^\bullet}=q_0$.
This isometry extends to an isometry between $\overline{|I^{\mathcal P}_c|}$ and
$Q_0\bydef([-1,1],0)^{(\mathcal P^+,\infty)}$.
\end{remark}

\begin{theorem}\label{3.1} Let $P$ and $Q$ be preposets.

(a) $|P\x Q|$ is uniformly homeomorphic to $|P|\x |Q|$.

(b) $|P+Q|$ is uniformly homeomorphic to $|P|*|Q|$

In particular, $|CP|$ and $|C^*P|$ are uniformly homeomorphic to $C|P|$.

(c) $|P*Q|$ is uniformly homeomorphic to $|P|*|Q|$.
\end{theorem}

\begin{proof}[Proof. (a)] The injections $P\to 2^{\mathcal P}$ and
$Q\to 2^{\mathcal Q}$ as in \cite[Lemma \ref{comb:2.6}]{M1}
yield $|P|\incl [0,1]^{\mathcal P}$ and
$|Q|\incl [0,1]^{\mathcal Q}$, where $P=(\mathcal P,\preceq)$ and
$Q=(\mathcal Q,\le)$.
The injection $P\x Q\to 2^{\mathcal P}\x 2^{\mathcal Q}\simeq
2^{\mathcal P\sqcup\mathcal Q}$ yields
$|P\x Q|\incl [0,1]^{\mathcal P\sqcup\mathcal Q}$.
Meanwhile, $|P|\x|Q|$ lies in $[0,1]^{\mathcal P}\x[0,1]^{\mathcal Q}$, which
may be identified with $[0,1]^{\mathcal P\sqcup\mathcal Q}$.
To see that $|P\x Q|=|P|\x|Q|$ under this identification, it suffices to
consider the case where $P$ and $Q$ are nonempty finite totally ordered sets.
This case (and the more general case where $P$ and $Q$ are finite) follows using
that a chain in $2^{\mathcal P}\x 2^{\mathcal Q}=2^{\mathcal P\sqcup\mathcal Q}$
lies in $P\x Q$ if and only if it projects onto a chain in $P$ and onto a chain
in $Q$.
\end{proof}

\begin{proof}[(b)] Consider the injection
$P+Q\to 2^{\mathcal P\sqcup pt\sqcup\mathcal Q}$ defined by
$\sigma\mapsto\fll\sigma\flr$ if $\sigma\in P$, and by
$\sigma\mapsto\fll\sigma\flr\cup pt$ if $\sigma\in Q$.
This yields $|P+Q|\incl [0,1]^{\mathcal P\sqcup pt\sqcup\mathcal Q}$, so that
$|P|$ lies in $[0,1]^{\mathcal P}\x\{0\}\x\{0\}$ and $|Q|$ in
$\{1\}\x\{1\}\x[0,1]^{\mathcal Q}$.
It is easy to see that $|P+Q|$ is the union of $|P|$, $|Q|$ and all straight
line segments with one endpoint in $|P|$ and another in $|Q|$.
(Beware that these segments alone cover $|P+Q|$ only if both $P$ and $Q$ are
nonempty.)
Thus $|P+Q|$ is the independent rectilinear join of $|P|$ and $|Q|$, as defined
in \cite[\S\ref{metr:join, etc}]{M2}.
Hence by \cite[Theorem \ref{metr:two joins}]{M2}, $|P+Q|$ is uniformly
homeomorphic to $|P|*|Q|$.
\end{proof}

\begin{proof}[(c)] From definition, $P*Q$ is the subpreposet
$C^*P\x Q\cup P\x C^*Q$ of $C^*P\x C^*Q$.
By the proof of part (b), $|C^*P|$ is the rectilinear cone $c|P|$, as defined in
\cite[\S\ref{metr:join, etc}]{M2}.
Then by part (a), $|C^*P\x C^*Q|$ is uniformly homeomorphic
to $c|P|\x c|Q|$; whereas $|P*Q|$ is uniformly homeomorphic to its subspace
$c|P|\x|Q|\cup |P|\x c|Q|$.
Write $|P|=X$ and $|Q|=Y$ for the sake of brevity.
Then, Lemmas \ref{metr:join-amalgam}, \ref{metr:rectilinear cone} and
\ref{metr:join} in \cite{M2} yield uniform homeomorphisms
$$cX\x Y\cup X\x cY\to cX\x Y\underset{X\x Y}\cup X\x cY\to
CX\x Y\underset{X\x Y}\cup X\x CY\to X*Y,$$
where each of the amalgamated unions in the middle is defined as a pushout in
the category of uniform spaces (and so is endowed with the quotient uniformity).
\end{proof}

\section{Geometric realization via quotient}

Given a collection of preposets $P_\alpha=(\mathcal P_\alpha,\le)$, their
{\it disjoint union} $\bigsqcup_\alpha P_\alpha$ is their coproduct in
the category of preposets; more explicitly, it is the preposet
$(\bigsqcup_\alpha\mathcal P_\alpha,\preceq)$, where $p_\alpha\in P_\alpha$
and $p_\beta\in P_\beta$ satisfy $p_\alpha\preceq p_\beta$ iff
$\alpha=\beta$ and $p_\alpha\le p_\beta$.
We note that disjoint union does not commute with geometric realization
unless the index set is finite, because every infinite disjoint union of
non-discrete uniform spaces is easily seen to be non-metrizable.

\begin{theorem} \label{CW} Let $P$ be a poset, and let $P_\sqcup$ be the disjoint
union of all nonempty finite chains of $P$.
Let $\pi\:P_\sqcup\to P$ be determined by the inclusions $C\incl P$, where
$C\in P^\flat$.
Then

(a) $|\pi|\:|P_\sqcup|\to |P|$ is a quotient map (in the category of uniform
spaces);

(b) if $d$ is the standard metric on $|P_\sqcup|$, then
$$d_\infty(x,y)=\inf_{n\in\N}
\inf_{\ \ \substack{x_1,\dots,x_{n-1}\in |P|\\ (x_0\bydef x,\ x_n\bydef y)}\ \ }
\sum_{i=0}^{n-1} d(\pi^{-1}(x_i),\pi^{-1}(x_{i+1}))$$ is a metric on $|P|$.
\end{theorem}

Theorem \ref{CW}(a) implies that $|P|$ is a quotient space of $|P_\sqcup|$
(in the category of uniform spaces).
This is reminiscent of the definition of geometric realization of
semi-simplicial sets, and of the well-known characterization of the topology of
a CW-complex as the topology of a quotient (in the category of topological spaces!)
of the disjoint union of its cells.

Theorem \ref{CW}(b) is reminiscent of the definition of geometric polyhedral
complexes used in metric geometry and in geometric group theory (see \cite{BH},
\cite{BBI}).

\begin{proof} This is based on the technique of quotient maps of finite type
(see \cite[\S\ref{metr:metrizability}]{M2}) and on the proof of
Theorem \ref{isometry} above.

Write $q=|\pi|$, and let $d$ stand for the usual metric on $|P_\sqcup|$ and on
$P$.
Clearly, $q$ is surjective.
Given $x,y\in|P|$, let
$d_n(x,y)=\inf_{x=x_0,\dots,x_n=y}\sum d(q^{-1}(x_i),q^{-1}(x_{i+1}))$ and
$d_\infty(x,y)=\inf_{n\in\N} d_n(x,y)$.
It is easy to see that $d_\infty$ is a pseudo-metric on $|P|$ (while each $d_n$
need not satisfy the triangle axiom) and that the identity maps
$(|P|,d_n)\xr{\id}(|P|,d_\infty)\xr{\id}(|P|,d)$ are uniformly continuous for
each $n$.
If $(|P|,d)\xr{\id}(|P|,d_n)$ is uniformly continuous for some $n$, then
on the one hand, $d_\infty$ is uniformly equivalent to $d$, and on the other hand,
by \cite[Theorem \ref{metr:metrization lemma}(b)]{M2}, $d_\infty$ induces
the quotient uniformity on $|P|$.
Thus it suffices to show that $(|P|,d)\xr{\id}(|P|,d_3)$ is uniformly continuous.

Suppose that $P=(\mathcal P,\le)$, and let $\hat P\bydef C^*CP$ (with additional
elements $\hat 0$ and $\hat 1$).
The standard geometric realization $|P|\subset 2^{\mathcal P}$ lies in
the reduced geometric realization $|\hat P|'\subset 2^{\mathcal P\cup\{\hat 1\}}$
(where $\hat 0$ is identified with $\emptyset\in 2^{\mathcal P\cup\{\hat 1\}}$).

Pick some $x,y\in |P|$, and let $A\incl P$ and $B\incl P$ be any chains whose
convex hulls contain $x$ and $y$ respectively.
We have unique isomorphisms $a\:[m]\to\hat A$ and $b\:[n]\to\hat B$, where
$a(1)=\hat 0=b(1)$ and $a(m)=\hat 1=b(n)$.
Let $(k_1,l_1),\dots,(k_r,l_r)$ and $(k'_1,l'_1),\dots,(k'_{r'},l'_{r'})$, and
$Z$ and $Z'$ be as in the proof of Theorem \ref{isometry}.
We also recall the notation $\alpha_i=x(s)$ for some $s\in a(i)\but a(i-1)$,
and $\beta_j=y(t)$ for some $t\in b(j)\but b(j-1)$, where $2\le i\le m$ and
$2\le j\le n$.
Observe that this does not depend on the choices of $s$ and $t$.

We now define an `intermediary' chain $C\incl P$, viewed as an isomorphism
$c\:[q]\to\hat C$.
The inductive construction starts with $c(1)=\hat 0$, and in the event that
$c(i)=\hat 1$ it terminates with $q=i$.
Suppose that $c(i)=a(k)$ for some $k<m$; if $k\ne k_j$ for any $j$, then let
$c(i+1)=a(k+1)$; if $k=k_j$ for some $j\notin Z$, then let $c(i+1)=b(l_j)$.
Similarly, suppose that $c(i)=b(l)$ for some $l<n$; if $l\ne l'_j$ for any $j$,
then let $c(i+1)=b(k+1)$; if $l=l'_j$ for some $j\notin Z'$, then let
$c(i+1)=a(k'_j)$.
Finally, if $c(i)=a(k)=b(l)<\hat 1$, then we are free to set either $c(i+1)=a(k+1)$
or $c(i+1)=b(l+1)$.

Next we define an $x'\in |A'|_{a'}$ and a $y'\in |B'|_{b'}$, where $A'=A\cap C$
and $B'=B\cap C$, viewed as isomorphisms $a'\:[m']\to\hat A'$ and
$b'\:[n']\to\hat B'$.
Given an $s\in\mathcal P$, we have
$s\in a'(m')=\fll\hat 1\flr=\mathcal P\cup\{\hat 1\}$, and
$s\notin a'(1)=\fll\hat 0\flr=\emptyset$.
Hence $s\in a'(i)\but a'(i-1)$, where $2\le i\le m'$.
Pick some $h_i\in[m]$ so that $a'(i-1)\le a(h_i-1)$ and $a(h_i)\le a'(i)$.
We must be more specific for $i=1$ and $i=m$, and we set $h_2=2$ (which is
the least among all possible choices) and $h_{m'}=m$ (which is the greatest among
all possible choices).
Since $h_i\ge 2$, we may set $x'(s)=\alpha_{h_i}$.
Let $\alpha'_i=x'(s)$ for any $s\in a'(i)\but a'(i-1)$, where $2\le i\le m'$;
clearly this does not depend on the choice of $s$.
Thus $\alpha'_2=\alpha_2$ and $\alpha'_{m'}=m$.
Since $x\in|A|_a$, we have $\alpha_2=1$ and $\alpha_m=0$.
Therefore $\alpha'_2=1$ and $\alpha'_{m'}=0$, whence $x'\in|A'|_{a'}$.
We can similarly define a $y'\in |B'|_{b'}$ and consequently $\beta'_j$ where
$2\le j\le n'$.

Let us estimate $d(x,x')$ from above.
Suppose that $s\in a'(i)\but a'(i-1)$, where $2\le i\le m'$.
If both $a'(i)$ and $a'(i-1)$ belong to $\hat C$, then $x'(s)=x(s)$.
Else we have $a'(i-1)=a(k_j)$ and $a'(i)=a(k'_{j'})$, where $1\le j<r$ and
$1<j'\le r'$.
Moreover, by the construction of $C$ we have $l'_{j'-1}<l_j\le l'_{j'}<l_{j+1}$.
By definition, both $x(s)$ and $x'(s)$ belong to
$[\alpha_{k_j+1},\alpha_{k'_{j'}}]$.
Since $l'_{j'-1}+1\le l_{j+1}$, we have
$\alpha_{k'_{j'}}-\alpha_{k_j+1}\le\alpha_{k'_{j'}}-\beta_{l'_{j'-1}+1}+
\beta_{l_{j+1}}-\alpha_{k_j+1}$.
However $|\alpha_{k'_{j'}}-\beta_{l'_{j'-1}+1}|\le d(x,y)$ and
$|\beta_{l_{j+1}}-\alpha_{k_j+1}|\le d(x,y)$ by the proof of Theorem \ref{isometry}.
Thus $|x(s)-x'(s)|\le 2d(x,y)$.
We have proved that $d(x,x')\le 2d(x,y)$.

We have $d_3(x,y)\le d(x,x')+d(x',y')+d(y',y)$.
By the above, $d(x,x')\le 2d(x,y)$, and similarly $d(y',y)\le 2d(x,y)$.
By the triangle axiom, $d(x',y')\le d(x',x)+d(x,y)+d(y,y')\le 5d(x,y)$.
Hence $d_3(x,y)\le 9d(x,y)$.
Thus $(|P|,d)\to (|P|,d_3)$ is uniformly continuous.
\end{proof}

\begin{corollary}\label{canonical} If $P$ is a preposet, then $|P^\#|$ is uniformly
homeomorphic to $|P|$ by a homeomorphism $h\:|P^\#|\to|P|$.

Moreover, if $P$ is a poset, then $d_\infty(x,y)=2d_\infty(h(x),h(y))$
for all $x,y\in|P^\#|$.
\end{corollary}

\begin{proof}
Consider the map of sets $f\:P^\#\to|P|=|P|_{j_P}$ defined by sending an element
$[\sigma,\sigma]\in P^\#$ into the vertex $|\{\sigma\}|_{j_P}$ of $|P|_{j_P}$,
and an element $[\sigma,\tau]\in P^\#$ with $\sigma<\tau$ into the central point of
the straight line segment connecting the vertices $|\{\sigma\}|_{j_P}$ and
$|\{\tau\}|_{j_P}$.
A finite chain $C$ of $P^\#$ is of the form
$[\sigma_1,\tau_1]\Subset\dots\Subset [\sigma_n,\tau_n]$, where
$\sigma_1\le\dots\le\sigma_n\le\tau_n\le\dots\le\tau_1$.
(We recall that when $P$ is a poset, $\Subset$ is just the usual inclusion.)
By collapsing all the equality signs in the latter string of inequalities
we obtain a sting of strict inequalities, which represents a chain
$\tilde C$ of $P$ (of length $\ge n$).
Then $f(C)\incl|\tilde C|_{j_P}$; hence $f$ extends by linearity to a map
$h\:|P^\#|\to |P|$.

It is not hard to see that $h$ is a bijection (note that this is well-known in
the case where $P$ is a poset, cf.\ \cite{M1}).
Indeed, for every chain $D$ of $P$, $f^{-1}(|D|_{j_P})$ can be identified with
$D^\#$.
So the assertion reduces to the case $P=[n]$, which can be checked directly.

To show that $h$ is a uniform homeomorphism we may assume that $P$ is a
poset by considering the transitive closure.
Then it suffices to prove that $d_\infty(x,y)=2d_\infty(h(x),h(y))$ for all
$x,y\in|P^\#|$.
By the definition of the $d_\infty$ metrics (with respect to the usual metrics
$d$ on $P_\sqcup$ and $(P^\#)_\sqcup$; see the statement of Theorem \ref{CW}(b)),
it suffices to prove this when $x,y\in|D^\#|$ for some chain $D$ of $P$.
So the assertion again reduces to the case $P=[n]$, which can be checked directly.
\end{proof}

\begin{corollary} \label{CW2} Let $P$ be a poset, and let $P_\square$ be
the disjoint union of all intervals of $P$.
Let $\rho\:P_\square\to P$ be determined by the inclusions $Q\incl P$, where
$Q\in P^\#$.
Then $|\rho|\:|P_\square|\to |P|$ is a quotient map (in the category of uniform
spaces).
\end{corollary}

\begin{proof}
Consider the commutative square
$$\begin{CD}
(P_\square)_\sqcup@>>>P_\square\\
@V\pi' VV@V\rho VV\\
P_\sqcup@>\pi>>P.
\end{CD}$$
Here $\pi'$ is trivially a quotient map, and $\pi$ is a quotient map by
Theorem \ref{CW}(a).
Hence $\pi\pi'$ is a quotient map, and therefore so is $\rho$.
\end{proof}

\begin{remark}
Corollary \ref{CW2} is, in a sense, easier than Theorem \ref{CW}, for it can also
be proved as follows.
One first shows that $|r_X|$ (defined in \S\ref{NR} below) is uniformly
continuous without using Corollary \ref{CW2}; this can be done by writing
an explicit formula for $|r_X|$ in coordinates.
Next one observes that $|r_X^\#|\:|P^{\#\#}|\to|P^\#|$ takes any pair of
sufficiently close points onto a pair of points contained in
$|\fll q\flr|$ for the same interval $q\in P^\#$.
It then remains to use \cite[Theorem \ref{metr:metrization lemma}(b)]{M2} in the same way
as it is used in the proof of Theorem \ref{CW}.
\end{remark}

\begin{remark} \label{canonical'}
Corollary \ref{CW2} suffices to show that $h\:|P^\#|\to|P|$,
as defined in the proof of Corollary \ref{canonical}, is a uniform
homeomorphism.
Indeed, writing $P=(S,\le)$, the usual metric on $|P|$ is induced from that
on $|2^S|^\bullet$ via the usual embedding $P\emb 2^S$, and similarly the
usual metric on $|P^\#|$ is induced from that on $|(2^S)^\#|$ (see
\cite[Lemma \ref{comb:2.6}]{M1} and Corollary \ref{subposet realization}).
So it suffices to show that $h\:|(2^S)^\#|\to|2^S|^\bullet$ is a uniform
homeomorphism.
Let $Q=(2^S)^\#$.
By Corollary \ref{CW2}, $(|Q|,d)$ is uniformly homeomorphic to $(|Q|,d_\infty)$,
where $d$ denotes the usual metric on $|Q|$ and on $|Q_\square|$, and $d_\infty$
is as in the statement of Theorem \ref{CW}(b).
On the other hand, it is easy to see that $(|Q|,d_\infty)$ is isometric to
$(|2^S|^\bullet,2d)$, where $d$ is the usual $l_\infty$ metric on the cube
$|2^S|^\bullet$.
Hence $(|Q|,d)$ is uniformly homeomorphic to $(|2^S|^\bullet,d)$.
\end{remark}

\begin{remark} If $P$ is a simple poset (see \cite[\S\ref{comb:simple}]{M1}),
then it is easy to
see that the usual metrics $d$ on $|P_\sqcup|$ and on $|P_\square|$ lead to
the same $d_\infty$ metric on $|P|$.
So Theorem \ref{CW} for such $P$ can be recovered from Corollary \ref{CW2}.
\end{remark}

\begin{definition}[Geometric realization of a monotone map]
Given a monotone map $f\:P\to Q$ between posets, it extends uniquely to
a map $|f|\:|P|\to|Q|$ that is affine on every convex hull of a chain.
In fact, $|f|$ is clearly $1$-Lipschitz on every convex hull of a chain.
On the other hand, $f$ lifts uniquely to a monotone map
$f_\sqcup\:P_\sqcup\to Q_\sqcup$.
Then $|f_\sqcup|$ is $1$-Lipschitz (globally), and in particular,
uniformly continuous.
Since $q\:|Q_\sqcup|\to |Q|$ is uniformly continuous, so is the
composite arrow in the commutative diagram
$$\begin{CD}
|P_\sqcup|@>f_\sqcup>>|Q_\sqcup|\\
@VpVV@VqVV\\
|P|@>f>>|Q|.
\end{CD}$$
By Theorem \ref{CW}(a), $p\:|P_\sqcup|\to|P|$ is a quotient map; in other words,
the uniformity of $|P|$ is final with respect to $p$.
Hence $|f|$ is uniformly continuous.
We call it the {\it geometric realization} of $f$.

It is easy to see that geometric realization of posets and of monotone maps
determines a functor (also called the {\it geometric realization}) from
the category of posets and monotone maps to the category of metrizable uniform
spaces and uniformly continuous maps.
\end{definition}

\begin{theorem}\label{pullback-pushout} The geometric realization functor
preserves pullbacks, as well as those pushouts that remain such upon
barycentric subdivision.
\end{theorem}

\begin{proof} The assertion is equivalent to the preservation of finite
products, finite coproducts (which always remain finite coproducts upon
barycentric subdivision), embeddings, and those quotient maps that remain
quotient maps upon barycentric subdivision.
Finite products were considered in Theorem \ref{3.1}(a) and embeddings in
Corollary \ref{subposet realization}.
The preservation of finite coproducts is obvious.

Finally, let $f\:P\to Q$ be a quotient map of posets such that
$f^\flat\:P^\flat\to Q^\flat$ is also a quotient map.
In particular%
\footnote{In fact, if the simplicial map $f^\flat$ is surjective, then $f^\flat$
and $f$ are quotient maps.}%
, $f^\flat$ is surjective, so every chain in $Q$ is the image of a chain in $P$.
Then $|f_\sqcup|\:|P_\sqcup|\to|Q_\sqcup|$ is a uniformly continuous retraction,
and therefore a quotient map.
By Theorem \ref{CW}(a), also $q\:|Q_\sqcup|\to|Q|$ is a quotient map.
Then the composite arrow in the preceding commutative diagram is a quotient map.
The assertion now follows from and the fact that if a composition $X\to Y\xr{f}Z$
is a quotient map, then so is $f$.
\end{proof}

Theorem \ref{pullback-pushout} yields an alternative proof of Theorem \ref{3.1}(b,c):

\begin{corollary}\label{join2} If $P$ and $Q$ are posets, then $|P*Q|$ and $|P+Q|$
are uniformly homeomorphic to $|P|*|Q|$.
\end{corollary}

\begin{proof} $P*Q$ is the pushout of the diagram
$P\x Q\x I\supset P\x Q\x\partial I\to P\x\{\{0\}\}\sqcup Q\x\{\{1\}\}$,
where $I=\Delta^{\{0,1\}}$;
$P+Q$ is the pushout of the diagram
$P\x Q\x [2]\supset P\x Q\x(\{1\}\sqcup\{2\})\to P\x\{1\}\sqcup Q\x\{2\}$,
where $[2]=(\{1,2\},\le)$;
and $X*Y$ is the pushout of the diagram
$X\x Y\x [0,1]\supset X\x Y\x\{0,1\}\to P\x\{0\}\sqcup Q\x\{1\}$.
Since $|I|$ and $|[2]|$ are uniformly homeomorphic to $[0,1]$, the assertion
follows from Theorem \ref{pullback-pushout}.
\end{proof}

\section{Comparison with traditional geometric realization} \label{comparison}

In this section we compare our ``cubical'' uniform geometric realization $|K|$ with the traditional 
``rectilinear'' geometric realization $||K||$ of the simplicial complex $K$.
Let $V$ denote the vertex set of $K$, and for each $\sigma\in K$ let $\ddot\sigma\subset V$ denote the set 
of vertices of $\sigma$.
Both geometric realizations lie in the vector space $\R[V]$ of formal linear combinations of 
vertices of $K$:
\begin{itemize}
\item $||K||$ is the union of the convex hulls $\left<\ddot\sigma\right>$ of the vertices of
all simplexes $\sigma\in K$;
\item $|K|$ is the union of the convex hulls 
$\left<v_1,v_1+v_2,\dots,v_1+\dots+v_n\right>$ for all ordered collections $(v_1,\dots,v_n)$ of pairwise
distinct vertices $v_i\in\ddot\sigma$ of all simplexes $\sigma\in K$.
\end{itemize}
There are two natural ways to identify $|K|$ with $||K||$ as sets:
\begin{itemize}
\item the ``radial'' bijection $\rho\:|K|\to ||K||$, which moves each point along the linear 1-subspace 
(passing through the origin) that it generates; 
\item the ``radial-affine'' bijection $\alpha\:|K|\to ||K||$, which moves the barycenter of each simplex 
of $K$ in the same way and each simplex of $K^\flat$ by an affine map.
\end{itemize}
The radial bijection $\rho$ is simpler, but it is not piecewise linear.
However, it is good enough in that it sends every convex polytope contained in $|K|$ to a convex polytope 
contained in $||K||$; in particular, it does not destroy affine subdivisions. 
The radial-affine bijection $\alpha$ is an instance of the  ``pseudo-radial projection'' which is used 
to correct the so-called ``standard mistake'' in PL topology.
We will consider both bijections.

\begin{lemma} \label{rho-alpha}
Let $x=x_1v_1+\dots+x_nv_n\in|K|$ and $a=a_1v_1+\dots+a_nv_n\in||K||$, where each $v_i\in V$.

(a) If $\rho(x)=a$, then each $a_i=x_i\big/(x_1+\dots+x_n)$ and each $x_i=a_i/\max(a_1,\dots,a_n)$.

(b) If $\alpha(x)=a$ and $x_1\le\dots\le x_n$, then each 
$a_i=\frac{x_i}{n-i+1}-\sum_{j=1}^{i-1}\frac{x_j}{(n-j+1)(n-j)}$ and each $x_i=a_1+\dots+a_{i-1}+(n-i+1)a_i$.

(b$'$) If $\alpha(x)=a$, then $x_1\le\dots\le x_n$ is equivalent to $a_1\le\dots\le a_n$.
\end{lemma}

\begin{proof}[Proof. (a)]
Since $||K||$ lies in the hyperplane $\{a_1v_1+\dots+a_nv_n\mid a_1+\dots+a_n=1\}$, each 
$a_i=x_i\big/\sum_{j=1}^n x_j$.
Since $|K|$ lies in the corner $\{x_1v_1+\dots+x_nv_n\mid\max(x_1,\dots,x_n)=1\}$, each 
$x_i=a_i/\max_i a_i$.
\end{proof}

\begin{proof}[(b)] Let $w_i=v_i+\dots+v_n$.
Then $x$ lies in the convex hull $\left<w_1,\dots,w_n\right>$; indeed, we have
$x=x_1w_1+(x_2-x_1)w_2+\dots+(x_n-x_{n-1})w_n$ and $x_1+(x_2-x_1)+\dots+(x_n-x_{n-1})=x_n=1$.
Since $\alpha$ is affine on $\left<w_1,\dots,w_n\right>$,  
$\alpha(x)=x_1\alpha(w_1)+(x_2-x_1)\alpha(w_2)+\dots+(x_n-x_{n-1})\alpha(w_n)$, where
$\alpha(w_i)=\frac{w_i}{n-i+1}$.
Hence $\alpha(x)=\sum_{i=1}^n x_i\big(\frac{w_i}{n-i+1}-\frac{w_{i+1}}{n-i}\big)$, where $w_{n+1}=0$.
Therefore $\alpha(x)=\sum_{i=1}^n x_i\big(\frac{v_i}{n-i+1}-\frac{v_{i+1}+\dots+v_n}{(n-i+1)(n-i)}\big)$.
Then $\alpha(x)=\sum_{i=1}^n\big(\frac{x_i}{n-i+1}-\sum_{j=1}^{i-1}\frac{x_j}{(n-j+1)(n-j)}\big)v_i$.
Thus $a_i=\frac{x_i}{n-i+1}-\sum_{j=1}^{i-1}\frac{x_j}{(n-j+1)(n-j)}$.

From the latter equation $a_1=\frac{x_1}n$ and 
$a_{i+1}-a_i=\frac{x_{i+1}}{n-i}-\frac{x_i}{(n-i+1)(n-i)}-\frac{x_i}{n-i+1}=
\frac{x_{i+1}}{n-i}-\frac{x_i}{n-i}$ for $i=1,\dots,n-1$.
Thus we have $x_1=na_1$ and $x_{i+1}-x_i=(n-i)(a_{i+1}-a_i)$ for $i=1,\dots,n-1$.
This yields $x_i=a_1+\dots+a_{i-1}+(n-i+1)a_i$.
\end{proof}

\begin{proof}[(b$'$)] If $x_1\le\dots\le x_n$, then from the proof of (b) we have
$x_{i+1}-x_i=(n-i)(a_{i+1}-a_i)$ for $i=1,\dots,n-1$.
Hence $a_1\le\dots\le a_n$.

Moreover, if $x_1\le\dots\le x_n$ and some $x_i<x_{i+1}$, then by the same argument we get that $a_i<a_{i+1}$.
If $x_1\le\dots\le x_n$ and $x_j<x_k$, then $j<k$ and $x_i<x_{i+1}$, where $j\le i<k$; consequently 
$a_i<a_{i+1}$ and so $a_j<a_k$.
If $x_j<x_k$, then $x_{\pi(1)}\le\dots\le x_{\pi(n)}$ for some permutation $\pi$, and similarly to the above 
we get $a_j<a_k$.

Now if $a_1\le\dots\le a_n$ and it is not true that $x_1\le\dots\le x_n$, then $x_i<x_j$ for some
$i$ and $j$ such that $i>j$.
Then $a_i>a_j$, which contradicts our hypothesis.
\end{proof}

\begin{theorem} \label{Isbell's metric} 
Let $K$ be a simplicial complex with vertex set $V$.
Let $\R[V]$ be endowed with the $l_\infty$ metric.

(a) $\rho^{-1}$ is continuous. 
If $K$ is finite dimensional, then $\rho$ and $\rho^{-1}$ are Lipschitz.

(b) $\alpha$ is uniformly continuous. 
If $K$ is finite dimensional, then $\alpha^{-1}$ is Lipschitz.

(b$'$) $\alpha^{-1}$ is continuous.
\end{theorem}

\begin{proof}[Proof. (a)]
Let $a=\rho(x)$ and $b=\rho(y)$.
We have $x=\sum_{i=1}^nx_iv_i$ and $y=\sum_{i=1}^ny_iv_i$, also $a=\sum_{i=1}^na_iv_i$ and 
$b=\sum_{i=1}^nb_iv_i$ for some $n\in\N$ and some $v_i\in V$.

Let $s_x=x_1+\dots+x_n$ and $s_y=y_1+\dots+y_n$.
Then $s_x,s_y\in[1,n]$ and $|s_x-s_y|\le nd(x,y)$.
By Lemma \ref{rho-alpha}(a) $a_i-b_i=\frac{x_i}{s_x}-\frac{y_i}{s_y}=
\frac{x_i(s_y-s_x)+(x_i-y_i)s_x}{s_xs_y}=\frac{x_i(s_y-s_x)}{s_xs_y}+\frac{x_i-y_i}{s_y}$.
Then we have $|a_i-b_i|\le |s_y-s_x|+|x_i-y_i|$.
Hence $d(a,b)\le (n+1)d(x,y)$.

Let $m_a=\max(a_1,\dots,a_n)$ and $m_b=\max(b_1,\dots,b_n)$.
Then $m_a,m_b\in [\frac1n,1]$.
If $m_a=a_i$, then $m_b\ge b_i\ge a_i-d(a,b)=m_a-d(a,b)$, and similarly $m_a\ge m_b-d(a,b)$.
Hence $|m_a-m_b|\le d(a,b)$.
By Lemma \ref{rho-alpha}(a) $x_i-y_i=\frac{a_i}{m_a}-\frac{b_i}{m_b}=
\frac{a_i(m_b-m_a)+(a_i-b_i)m_a}{m_am_b}=\frac{a_i(m_b-m_a)}{m_am_b}+\frac{a_i-b_i}{m_b}$.
Then $|x_i-y_i|\le\frac{|m_b-m_a|}{1/n^2}+\frac{|a_i-b_i|}{1/n}$.
Hence $d(x,y)\le (n^2+n)d(a,b)$.

We also have $|x_i-y_i|\le\frac{|m_b-m_a|+|a_i-b_i|}{m_a(m_a-|m_a-m_b|)}$.
Hence $d(x,y)\le 2d(a,b)/m_a\big(m_a-d(a,b)\big)$.
In particular, if $d(a,b)\le\frac{m_a}2$, then $d(x,y)\le 4d(a,b)/m_a^2$.
Thus $\rho^{-1}$ is continuous.
\end{proof}

\begin{proof}[(b)]
Let us pick some $v_1,\dots,v_n\in V$.
Let $x,y\in\left<w_1,\dots,w_n\right>$, where $w_i=v_i+\dots+v_n$.
Then $x=\sum_{i=1}^nx_iv_i$ and $y=\sum_{i=1}^ny_iv_i$, where $x_1\le\dots\le x_n$ and $y_1\le\dots\le y_n$.
Let $a=\alpha(x)$ and $b=\alpha(y)$.
We have $a=\sum_{i=1}^na_iv_i$ and $b=\sum_{i=1}^nb_iv_i$.

By Lemma \ref{rho-alpha}(b) each
$a_i-b_i=\frac{x_i-y_i}{n-i+1}-\sum_{j=1}^{i-1}\frac{x_j-y_j}{(n-j+1)(n-j)}$.
Consequently we have $d(a,b)\le\max_i\big(\frac1{n-i+1}+\sum_{j=1}^{i-1}\frac1{(n-j+1)(n-j)}\big)d(x,y)$.
From $\frac1{k(k+1)}=\frac1k-\frac1{k+1}$ we obtain 
$\sum_{j=1}^{i-1}\frac1{(n-j+1)(n-j)}=\frac1{n-i+1}-\frac1n$.
Hence $d(a,b)\le\max_i\big(\frac2{n-i+1}-\frac1n)d(x,y)=(2-\frac1n)d(x,y)$.
So $d(a,b)\le 2d(x,y)$.
Thus $\alpha$ is 2-Lipschitz on $\left<w_1,\dots,w_n\right>$.

Let $D_K$ be the path metric on $|K|$ defined as in Theorem \ref{CW}(b), which extends the $l_\infty$
metric on each $\left<w_1,\dots,w_n\right>$, and let $D'_K$ be the similar path metric on $||K||$, 
extending the $l_\infty$ metric on each $\alpha\big(\!\left<w_1,\dots,w_n\right>\!\big)$.
Then $\alpha$ is 2-Lipschitz from $D_K$ to $D'_K$.
Clearly $\id\:||K||\to||K||$ is 1-Lipschitz from $D'_K$ to the ambient $l_\infty$ metric $d$.
Also, by Theorem \ref{CW}(b) $\id\:|K|\to |K|$ is uniformly continuous from $d$ to $D_K$.
Hence $\alpha$ is uniformly continuous from $d$ to $d$.

By Lemma \ref{rho-alpha}(b) each $x_i-y_i=(a_1-b_1)+\dots+(a_{i-1}-b_{i-1})+(n-i+1)(a_i-b_i)$.
Hence $d(x,y)\le\max_i\big((i-1)d(a,b)+(n-i+1)d(a,b)\big)=nd(a,b)$.
Thus $\alpha^{-1}$ is $n$-Lipschitz on $\alpha\big(\!\left<w_1,\dots,w_n\right>\!\big)$.

Let $k$ be the dimension of $K$.
Let $\hat K$ be the simplicial complex defined as the poset of all sets $\ddot\sigma\cup\ddot\tau$,
where $\sigma$ and $\tau$ run over $K$.
Then the dimension of $\hat K$ is at most $2k+1$, and $K$ is a subcomplex of $\hat K$ such that if 
$x,y\in||K||$, then $x,y\in||\fll\sigma\flr||\subset||\hat K||$ for some $\sigma\in\hat K$.
Since $x$ and $y$ are connected by a straight line segment in $||\fll\sigma\flr||$, we have
$d(x,y)=D'_{\hat K}(x,y)$.
By the above $\alpha^{-1}$ is $(2k+1)$-Lipschitz from $D'_{\hat K}$ to $D_{\hat K}$.
More precisely, we know this for $\alpha_{\hat K}^{-1}$; but $\alpha_{\hat K}|_K=\alpha_K$, so 
$\alpha_K^{-1}$ is also $(2k+1)$-Lipschitz from $D'_{\hat K}$ to $D_{\hat K}$.
Clearly, $\id\:|\hat K|\to|\hat K|$ is 1-Lipschitz from $D_{\hat K}$ to the ambient $l_\infty$ metric $d$.
Thus $\alpha^{-1}$ is $(2k+1)$-Lipschitz from $d$ to $d$.
\end{proof}

\begin{proof}[(b$'$)]
Let $a=\alpha(x)$.
We have $x=\sum_{i=1}^nx_iv_i$ and $a=\sum_{i=1}^na_iv_i$ for some $n\in\N$ and some $v_i\in V$.
We may assume that $a_1\le\dots\le a_n$.
Let $b=\alpha(y)$ be such that $d(a,b)<a_1/2$ and $d(a,b)<(a_{i+1}-a_i)/2$ for each $i=1,\dots,n-1$.
We have $y=\sum_{i=1}^ny_iv_i+\sum_{i=1}^mz_iw_i$ and $b=\sum_{i=1}^nb_iv_i+\sum_{i=1}^mc_iw_i$
for some $m\in\N$ and some $w_i\in V$.
Then $b_1\le\dots\le b_n$ and each $c_i<b_1$.
We may assume that $c_1\le\dots\le c_m$.

By Lemma \ref{rho-alpha}(b,b$'$) $x_i=a_1+\dots+a_{i-1}+(n-i+1)a_i$.
Similarly $z_j=c_1+\dots+c_{j-1}+(m+n-j+1)c_j$ and
$y_i=c_1+\dots+c_m+b_1+\dots+b_{i-1}+\big((m+n)-(m+i)+1\big)b_i$.
Since $c_j\le\dots\le c_m$, we have
$z_j\le c_1+\dots+c_m+nc_j\le (n+1)(c_1+\dots+c_m)\le (n+1)D(a,b)$, where $D$ is the $l_1$ metric.
Also, $|x_i-y_i|\le c_1+\dots+c_m+|a_1-b_1|+\dots+|a_{i-1}-b_{i-1}|+(n-i+1)|a_i-b_i|\le (n-i+1)D(a,b)$.
Hence $d(x,y)\le (n+1)D(a,b)$.
Since $D$ induces the same topology on $||K||$ as $d$ (see \cite{M3}*{\ref{book:4topologies}}),
we get that $\alpha^{-1}$ is continuous.
\end{proof}

\begin{corollary} \label{metric topology} $|K|$ is homeomorphic to $||K||$ via $\alpha$.
\end{corollary}

Thus as topological spaces, uniform polyhedra can be identified with polyhedra.

A different (shorter, but less explicit) proof of Corollary \ref{metric topology} can be extracted from 
\cite{M3}*{Theorem \ref{book:poset-realization} and Remark \ref{book:cubreal}(a)}.

\begin{corollary} \label{Isbell's metric2}
If $K$ is finite dimensional, then $|K|$ and $||K||$ are uniformly homeomorphic via both $\alpha$ and $\rho$.
\end{corollary}

Thus Isbell's finite-dimensional uniform polyhedra can be identified as 
a special case of our uniform polyhedra.

\begin{remark} It is easy to see that $\rho$ is discontinuous if $K$ is not locally finite dimensional.

However, $\rho$ is a homeomorphism between $||K||$ with the $l_\infty$ metric and
$|K|$ with the $l_1$ metric.
This similar to the proof of \ref{Isbell's metric}(a), using additionally that the $l_1$ and $l_\infty$ 
metrics induce the same topology on $||K||$ (see \cite{M3}*{\ref{book:4topologies}}).

Thus the $l_1$ and $l_\infty$ metrics on $|K|$ are related by the homeomorphism $\rho^{-1}\alpha$.
It is easy to see that they do not induce the same topology on $|K|$, unless $K$ is locally 
finite-dimensional.
\end{remark}

\begin{proposition} \label{traditional-qu}
Let $K$ be a simplicial complex with vertex set $V$ and let 
$||K||_\sqcup=\{(x,\sigma)\in ||K||\x K\mid x\in\left<\ddot\sigma\right>\}$ with metric $D$ defined by
$D\big((x,\sigma),(y,\sigma)\big)=d(x,y)$, where $d$ is the $l_\infty$ metric on $\R[V]$, and 
$D\big((x,\sigma),(y,\tau)\big)=1$ if $\sigma\ne\tau$.
Let $\sim$ be the equivalence relation on $||K||_\sqcup$ defined by $(x,\sigma)\sim(y,\tau)$ if $x=y$.
The the bijection $||K||_\sqcup/_\sim\to||K||$, $[(x,\sigma)]\mapsto x$, is a homeomorphism with respect to 
the topology of the quotient uniformity on $||K||_\sqcup/_\sim$ and the metric topology on $||K||$.
\end{proposition}

\begin{proof} Clearly, $\id_{||K||}$ is uniformly continuous from the quotient uniformity to $d$.
Conversely, by Theorem \ref{Isbell's metric}(b$'$) $\alpha^{-1}$ is continuous from $d$ to $d$;
by Theorem \ref{CW}(b) $\id_{|K|}$ is uniformly continuous from $d$ to the quotient uniformity;
and by the proof of Theorem \ref{Isbell's metric}(b) $\alpha$ is uniformly continuous from 
the quotient uniformity to the quotient uniformity.
So $\alpha\alpha^{-1}=\id_{||K||}$ is continuous from $d$ to the quotient uniformity.
\end{proof}

\begin{remark} Let us discuss bases of neighborhoods of points for $|K|$ and $||K||$.

$||K||$ is triangulated by the affine simplicial complex $K^\vartriangle$ whose simplexes 
are the convex hulls $||\sigma||=\left<\ddot\sigma\right>$ for all $\sigma\in K$.
For any $x\in ||K||$ let $\sigma_x$ be the smallest simplex of $K$ such that $x\in||\sigma||$.
Similarly, $|K|$ is cubulated by the affine cubical complex $K^\Box$ whose cubes are the convex hulls
$|q|=\left<\{\sigma^\cubvert\mid\sigma\in q\}\right>$, where $\sigma^\cubvert=\sum_{v\in\ddot\sigma}v$,
for all $q\in K^\#$ (that is, $q$ is an interval of $K$). 
For any $x\in |K|$ let $q_x$ be the smallest cube of $K^\#$ such that $x\in ||q||$.
For a cell $c$ of an affine cell complex $C$, its open star $\ost(c,C)$ is the union of the interiors of all 
cells of $C$ containing $c$.
A base of neighborhoods of any $x\in||K||$ in the metric topology is given by the the sets
$h_x^r\big(\ost(|\sigma_x|,K^\vartriangle)\big)$, where $h_x^r$ denotes the homothety with ratio $r\le 1$ 
centered at $x$ (see \cite{M3}*{Proof of \ref{book:4topologies}}).

Given an $x\in |K|$, we have $q_x=[\rho_x,\tau_x]$ for some $\rho_x,\tau_x\in K$.
Then the interior of $|q_x|$ consists of all $y=\sum_{v\in V} y_vv\in|K|$ such that 
$y_v=1$ for all $v\in\ddot\rho_x$, $0<y_v<1$ for all $v\in\ddot\tau_x\but\ddot\rho_x$ and $y_v=0$ for all 
$v\notin\ddot\tau_x$.
Consequently, $\ost(|q_x|,K^\Box)$ consists of all $y=\sum_{v\in V} y_vv\in|K|$ such that $y_v>0$ for all 
$v\in\ddot\rho_x$, $0<y_v<1$ for all $v\in\ddot\tau_x\but\ddot\rho_x$ and $y_v<1$ for all $v\notin\ddot\tau_x$.
Writing $x=\sum_{v\in V} x_vv$, let $r_0=\min\{x_v\mid v\in\ddot\tau_x\but\ddot\rho_x\}$ and 
$r_1=\min\{1-x_v\mid v\in\ddot\tau_x\but\ddot\rho_x\}$ if $\ddot\tau_x\but\ddot\rho_x\ne\emptyset$, and 
otherwise $r_0=r_1=\frac12$.
Then $r\bydef \min(r_0,r_1)\in (0,\frac12]$, and it is easy to see that $\ost(|q_x|,K^\Box)$ contains
$B_r(x)\cap |K|$, where $B_r(x)$ denotes the open ball of radius $r$ about $x$ in the $l_\infty$ metric.
Also, $\ost(|q_x|,K^\Box)$ is contained in $B_1(x)$.
Since the $l_\infty$ metric is induced by a norm, the image of $\ost(|q_x|,K^\Box)$ under the homothety 
$h_x^r$ with ratio $r\le 1$ centered at $x$ contains $B_{qr}(x)\cap|K|$ and is contained in $B_r(x)$.
Hence the subsets $h_x^r\big(\ost(|q_x|,K^\Box)\big)$ of $|K|$ form a base of neighborhoods of $x$.

Using this description it is not hard to see that the topology of $|K|$ is not induced by the 
product topology of $\R[V]$ if $K$ is not locally finite.
\end{remark}

\section{Second canonical neighborhood} \label{NR}

\begin{lemma}\label{canonical-MC}
Let $K$ be a poset.

(a) $|MC(r_K)|$ is uniformly homeomorphic to $|MC(\id_K)|$ by a homeomorphism that
is the identity on $K$ and extends the homeomorphism $|h(K)|\to|K^\#|\to|K|$
given by \ref{dual-isometric} and \ref{canonical}.

(b) $|MC^*(r_K)|$ is uniformly homeomorphic to $|MC(\id_{h(K)})|$ by a homeomorphism
that is the identity on $|h(K)|$ and extends the homeomorphism $|K|\to|K^\#|\to|h(K)|$
given by \ref{dual-isometric} and \ref{canonical}.
\end{lemma}

Beware that $MC^*(r_K)$ is generally not a poset since $r_K$ is generally not open.
Note that $MC^*(r_K)\simeq (MC(r_K^*))^*$, where $r_K^*$ is the composition
$K^\#\xr{j_K}(K^*)^\#\xr{\#}K^*$.

\begin{proof}[(a)] We define $f\:|MC(r_K)|\to|MC(\id_K)|$ as required on the top
and bottom, and extend it linearly to the convex hull of every chain.
A chain of $MC(\id_K)=K\x [2]$ is of the form $B+A$, where
$A=(\alpha_1<\dots<\alpha_n)$ is a chain in the domain, and
$B=(\beta_1<\dots<\beta_m)$ is a chain in the range, with $\beta_m\le\alpha_1$.
From the similar description of chains of $MC(\id_{h(K)})$ we deduce that
a chain of $MC(r_K)$ is of the form $D+C$, where $C$ is a chain of $h(K)$ of
the form $[\sigma_1,\tau_1]^*<\dots<[\sigma_r,\tau_r]^*$ and $D$ is a chain of $K$
of the form $\rho_1<\dots<\rho_s$, with $\rho_s\le\sigma_1$.
It is easy to see that $f$ sends $|h(B)+A|$ onto $|B+A|$ via the join of
the uniform homeomorphism $|h(B)|\to|B|$ with $\id_{|A|}$.
It follows that $f$ is a uniform homeomorphism.
\end{proof}

\begin{proof}[(b)] We define $f\:|MC^*(r_K)\to |MC(\id_{h(K)})|$ as required on
the top and bottom, and extend it linearly to the convex hull of every chain.
A chain of $MC^*(\id_{h(K)})=h(K)\x [2]$ is of the form $A+B$, where
$A=(\alpha_1<\dots<\alpha_n)$ is a chain in the domain, and
$B=(\beta_1<\dots<\beta_m)$ is a chain in the range, with $\alpha_n\le\beta_1$.
It follows that a chain of $MC(r_K)$ is of the form $C+D$, where $C$ is a chain
of $h(K)$ of the form $[\sigma_1,\tau_1]^*<\dots<[\sigma_r,\tau_r]^*$ and $D$ is
a chain of $K$ of the form $\rho_1<\dots<\rho_s$, with
$[\sigma_r,\tau_r]^*\le[\rho_1,\rho_s]^*$.
It is easy to see that $f$ sends $|C+D|$ onto $|C+h(D)|$ via the join of $\id_{|C|}$
and the uniform homeomorphism $|D|\to|h(D)|$.
It follows that $f$ is a uniform homeomorphism.
\end{proof}

\begin{remark} \label{canonical-MC(c)}
Similarly to the proof of (a), $|MC(\id_{h(K)})|$ is uniformly homeomorphic to
$|MC(r_K)|$ by a homeomorphism that is the identity on $|h(K)|$ and extends
the homeomorphism $|h(K)|\to|K^\#|\to|K|$ given by \ref{dual-isometric} and
\ref{canonical}.
\end{remark}

\begin{theorem}\label{subcomplex-NR}
Let $K$ be a poset and $L$ a closed subposet of $K$.
Then $|L|$ is a uniform neighborhood retract of $|K|$.
\end{theorem}

\begin{proof}
Clearly $|\fll h(L)\flr_{h(K)}|$ is a uniform neighborhood of $|h(L)|$ in $|h(K)|$.
By \ref{dual-isometric} and \ref{canonical}, it corresponds to a uniform
neighborhood of $|L|$ in $|K|$.
On the other hand, let $\partial N(L)=\fll h(L)\flr\but h(L)\incl h(K)$, and let
$N(L)=L\cup MC^*(r_K|_{\partial N(L)})\incl MC^*(r_K)$.
The monotone map $MC^*(r_K)\to MC^*(\id_K)\to K$ is a retraction, and sends $N(L)$
into $L$; hence it restricts to a retraction $N(L)\to L$.
Thus $|L|$ is a uniform retract of $|N(L)|$.

Let $f$ be the composition of the uniform homeomorphism
$|MC^*(r_K)|\to |MC^*(\id_{h(K)})|$ in Lemma \ref{canonical-MC}(b), with
the geometric realization of the projection $\pi\:MC^*(\id_{h(K)})\to h(K)$.
Then $f$ restricts to the identity on $|\partial N(L)|$ and to
the uniform homeomorphism $|L|\to |h(L)|$ on $|L|$.
A chain of $MC^*(r_K|_{\partial N(L)})$ is of the form $C+D$, where $C$ is a chain
of $\partial N(L)\incl h(K)$ of the form
$[\sigma_1,\tau_1]^*<\dots<[\sigma_r,\tau_r]^*$ and $D$ is a chain of $L\incl K$ of
the form $\rho_1<\dots<\rho_s$, with $[\sigma_r,\tau_r]^*\le[\rho_1,\rho_s]^*$.
Then for every chain $E$ of $h(D)$, $C+E$ is a chain of $h(K)$.
Since $\partial N(L)$ is disjoint from $h(L)$, it follows that $f$ sends
$|C+D|$ homeomorphically onto $|C+h(D)|$ via the join of $\id_{|C|}$ and
the uniform homeomorphism $|D|\to|h(D)|$.
It follows that $f$ restricts to a uniform homeomorphism between $|N(L)|$ and
$|\fll h(L)\flr_{h(K)}|$.
Since it sends $|L|$ onto $|h(L)|$, we obtain that $|h(L)|$ is a uniform
retract of $|\fll h(L)\flr_{h(K)}|$.
\end{proof}

\begin{remark}
The proof of Theorem \ref{subcomplex-NR} breaks down if the poset $MC(r_K)$
is used instead of the preposet $MC^*(r_K)$.
Indeed, a chain of $MC(r_K|_{\partial N(L)})$ is of the form $D+C$, where $C$ is
a chain of $\partial N(L)\incl h(K)$ of the form
$[\sigma_1,\tau_1]^*<\dots<[\sigma_r,\tau_r]^*$ and $D$ is a chain of $L\incl K$
of the form $\rho_1<\dots<\rho_s$, with $\rho_s\le\sigma_1$.
Then for a chain $E$ of $h(D)$, $E+C$ is generally not a chain of $h(K)$.

We could still try to follow the argument and see what happens.
So let $N'(L)=L\cup MC(r_K|_{\partial N(L)})\incl MC(r_K)$ and let $f'$ be
the composition of the uniform homeomorphism $|MC(r_K)|\to |MC(\id_{h(K)})|$
in Remark \ref{canonical-MC(c)} with the geometric realization of the projection
$\pi\:MC(\id_{h(K)})\to h(K)$.
Suppose that we have two disjoint chains $D+C$ and $D'+C'$ in $N'(L)$, where
$C$ and $D$ are as above, and $C'$ is a chain of $\partial N(L)\incl h(K)$ of
the form $[\sigma'_1,\tau'_1]^*<\dots<[\sigma'_r,\tau'_r]^*$ and $D$ is a chain
of $L\incl K$ of the form $\rho'_1<\dots<\rho'_s$, where $\rho_s<\rho'_1$,
$\rho'_s<\sigma_1$ and $\tau_1<\sigma'_1$.
Then $f'(|D+C|)$ and $f'(|D'+C'|)$ may intersect; so $f'$ restricted
to $|N'(L)|$ is generally not injective.
\end{remark}

\begin{definition}[Relative canonical subdivisions]
Let $K$ be a poset and let $L$ be a closed subposet of $K$.
In the notation of the proof of Theorem \ref{subcomplex-NR}, let $h(K,L)$
denote $N(L)\cup (h(K)\but h(L))$.
Then by the proof of Theorem \ref{subcomplex-NR}, $|h(K,L)|$ is uniformly
homeomorphic to $|h(K)|$ and hence to $|K|$.

Dually, let $K^\#_L=MC^*(r_K^*)\but(\cel L\cer\cup\cel (K\but L)^\#\cer)$.
This contains $L^\#$ and $K\but L$, and $|K^\#_L|$ is uniformly
homeomorphic to $|K|$ similarly to the above (using part (a) of Lemma
\ref{canonical-MC}).
Note that $(P+Q)^\#\simeq (P^**Q)^\#_{P^**\emptyset\cup\emptyset*Q}$, which
yields an alternative proof that $|P+Q|$ is uniformly homeomorphic to $|P*Q|$.
\end{definition}

\begin{definition}[Homotopy completeness]
We recall from \cite{M2} that a uniform space is called {\it homotopy complete}
if there exists a homotopy $h_t\:\overline{|P|}\to\overline{|P|}$, where
$\overline{|P|}$ is the completion of $|P|$, with $h_0=\id$ and
$h_t(\overline{|P|})\incl |P|$ for $t>0$.
\end{definition}

\begin{lemma}\label{3.9} Let $P$ be a countable poset.
Then $|P|$ is homotopy complete.
\end{lemma}

The proof is an extension of the proof that $q_{00}=|2^\N_w|^\bullet$ is
homotopy complete, which was given in \cite[Corollary \ref{metr:c_00}(b)]{M2}.

\begin{proof}[Proof] $|P|$ is uniformly homeomorphic to $|P^\#|$, and
$P^\#$ is atomic.
Hence we may assume without loss of generality that $P$ is atomic.

Let $R$ be the composition $MC(r_P)\to MC(\id_P)\to P$
(extending the map $r_P$), and let $H\:|MC(r_P)|\to |MC(\id_P)|$ be the uniform
homeomorphism of Lemma \ref{canonical-MC}(a).
Define $h_P\:\overline{|P|}\x I\to\overline{|P|}$ to be the unique extension of
$|R|H^{-1}$ over the completion, where $|[2]|$ is identified with $I=[0,1]$ by
the affine homeomorphism sending $\{2\}$ to $1$.
Further let $h^\bullet_P$ be defined similarly to $h_P$ but using {\it atomic}
geometric realizations throughout, provided that $P$ is either atomic or the dual
cone over an atomic poset.
Then it is easy to check that $h^\bullet_{2^\N_w}$ coincides with the homotopy
$h_t$ constructed in the proof of \cite[Corollary \ref{metr:c_00}(b)]{M2}.
On the other hand, the definition of $|P|$ is based on the embedding of
$P=(\mathcal P,\le)$ in $\Delta^{\mathcal P}_w$, and $h_P$ is the restriction of
$h^\bullet_{\Delta^{\mathcal P}_w}$, which in turn is the restriction of
$h^\bullet_{2^{\mathcal P}_w}$.
Hence $h_P(\overline{|P|}\x (0,1])\incl |2^{\mathcal P}_w|^\bullet
\cap\overline{|P|}=|P|$.
\end{proof}

\begin{remark} The proof of Lemma \ref{3.9} also shows, generalizing 
\cite[Remark \ref{metr:CW}]{M2}, that if $P$ is a countable poset, 
then $|P|$ is non-uniformly homotopy equivalent to the direct limit 
$|P|_{CW}$ of $|Q|$'s over all finite subposets $Q$ of $P$.
Since each $|Q|$ is compact, the underlying topology of $|P|_{CW}$
is the direct limit topology (see \cite[\S\ref{metr:dirlimits}]{M2}).
Hence the underlying topological space of $|P|_{CW}$ is a CW complex;
its topology can also be described as the quotient topology with 
respect to the surjection $|P_\sqcup|\to|P|$, where $P_\sqcup$ is 
the disjoint union of all chains of $P$.
\end{remark}

\begin{lemma}\label{cube} $|2^\N_w|$ is a uniform AR.
\end{lemma}

\begin{proof}
By Theorem \ref{isometry}, $|2^\N_w|$ is isometric to
$|2^\N_w|^\bullet=q_{00}$, which is known to be a uniform AR
\cite[Corollary \ref{metr:c_00}(b)]{M2}.
\end{proof}

\begin{remark} Since $2^\N_w\simeq C^*(\Delta^\N_w)$, by Theorem \ref{3.1}(b)
(or alternatively by Corollary \ref{join2}), $|2^\N_w|$ is uniformly
homeomorphic to a cone, and therefore is uniformly contractible.
Thus asserting that it is a uniform AR is equivalent (see
\cite[Lemma \ref{metr:uniform AR}]{M2}) to asserting that it is
a uniform ANR.
\end{remark}

\begin{theorem}\label{simplicial ANR}
If $P$ is a simplicial complex, then $|P|$ is a uniform ANR.
\end{theorem}

The finite-dimensional case is due to Isbell \cite[1.9]{I1}, \cite[VI.15]{I3}.

It is easy to see that the geometric realizations of simplicial complexes are
cubohedra in the sense of \cite{M2}, and so Theorem \ref{simplicial ANR} is
a special case of \cite[Corollary \ref{metr:cubohedron}]{M2}.
This results in a proof of Theorem \ref{simplicial ANR} based on the uniform
version of Hanner's criterion of ANR'ness
\cite[Theorem \ref{metr:Hanner}(b)]{M2}, which is in turn inherently
rooted in an infinite process.

We now give an alternative proof of Theorem \ref{simplicial ANR}, which is
arguably more technical, but on the other hand more ``combinatorial'' in
that it certainly does not involve any infinite process.

\begin{proof} By \cite[Theorem \ref{comb:2.10}]{M1}, $P$ is isomorphic to
a subcomplex of the simplex $\Delta^\N_w$.
Hence $P^\#$ is isomorphic to a subcomplex of $(\Delta^\N_w)^\#$, which
in turn is a subcomplex of $(C^*\Delta^\N_w)^\#=(2^\N_w)^\#$.
By Theorem \ref{subcomplex-NR}, we have that $|P^\#|$ is a uniform
neighborhood retract of $|(2^\N_w)^\#|$.
Hence by Corollary \ref{canonical}, $|P|$ is a uniform neighborhood retract of
$|2^\N_w|$.
Since $|2^\N_w|$ is a uniform ANR by Lemma \ref{cube}, so is $|P|$.
\end{proof}

\begin{example}\label{counterexample} Consider the poset $[n]=(\{1,\dots,n\},\le)$.
Let $C_n$ be the amalgamated union of $C^\#$ for all proper subchains
$C\subsetneq [n]$.
The canonical map $j_n\:C_n\to[n]^\#$ is an injection, but not an embedding
for $n>1$.
Consider the map $j\bydef \sqcup_{n\in\N}j_n$, injecting
$P\bydef \bigsqcup_{n\in\N} C_n$ into $Q\bydef \bigsqcup_{n\in\N} [n]^\#$.

Then the generalized geometric realization $|P|_j$ is not a uniform ANR.
Indeed, it follows from Corollary \ref{canonical} that each $|C_n|_{j_n}$
is uniformly homeomorphic to $X_n\bydef \bigcup_{C\subsetneq [n]}|C|\incl |[n]|$.
Now $|[n]|=\{(x_0,\dots,x_n)\mid 0=x_0\le\dots\le x_n=1\}$,
and $X_n$ consists of all $(x_0,\dots,x_n)\in |[n]|$ such that $x_i=x_{i+1}$
for some $i$.
But each $(x_0,\dots,x_n)\in |[n]|$ satisfies $x_{i+1}-x_i\le\frac1n$ for
some $i$ (by the pigeonhole principle).
Hence the $\frac1{2n}$-neighborhood of $X_n$ in $|[n]|$ is the entire $|[n]|$.
Consequently, for each $\eps>0$ the $\eps$-neighborhood of $|P|_j$ in $|Q|$
contains $|\bigsqcup_{n\in\N\but[m]}[n]^\#|$ for some $m$, and so does not retract
uniformly or even continuously onto $|P|_j$.
\end{example}

The following is a direct consequence of \cite[Theorem \ref{comb:tmc}]{M1}

\begin{theorem}\label{tmc2}
Let $f\:P\to Q$ be a monotone map between conditionally complete posets.
Then $TMC(f)$ is a conditionally complete poset, and $|TMC(f)|$ uniformly
strongly deformation retracts onto $|LMC(f)|$.
\end{theorem}

\begin{proposition}\label{tmc-retraction}
Let $f\:P\to Q$ be a monotone embedding between posets.
Then $|\left<MC(f)\right>|$ uniformly deformation retracts onto $|MC(f)|$.
\end{proposition}

\begin{proof} Since $f$ is an embedding, we may identify
$\left<MC(f)\right>$ with a subposet of $Q\x[2]$.
Then $h(\left<MC(f)\right>)$ gets identified with a subposet of
$h(Q\x [2])\simeq h(Q\x[2]^*)$.
It is easy to see that the monotone map $r_{Q\x[2]^*}\:h(Q\x[2]^*)\to Q\x [2]^*$
(see \cite[\S\ref{comb:handles onto cores}]{M1}) sends $h(\left<MC(f)\right>)$
onto $MC^*(f)$.
Using this, similarly to the proof of Theorem \ref{subcomplex-NR}
one constructs a uniform retraction of $|h(\left<MC(f)\right>)|$ onto
$|h(MC^*(f))|$, and by using Lemma \ref{canonical-MC} or similarly to
the proof of Lemma \ref{3.9} one constructs a uniform homotopy from this
retraction to the identity.
\end{proof}

\begin{definition}[Huge mapping cylinder]
Let $f\:P\to Q$ be a monotone map between countable posets.
Let $j_P\:P\emb 2^\N$ be the usual embedding $p\mapsto\fll p\flr$, where
the underlying set of $P$ is identified with a subset of $\N$.
Let $F$ be the composition $P\x 2^\N\to P\xr{f\x j_P} Q\x 2^\N$ of
the projection and the joint map.
Finally let $HMC(f)$ be the transitive closure $\left<MC(F)\right>$.
Note that $HMC(f)$ contains canonical copies of $P=P\x\{\emptyset\}$ and
$Q=Q\x\{\emptyset\}$.
\end{definition}

\begin{corollary}\label{hmc}
Let $f\:P\to Q$ be a monotone map between countable posets.
Then $|HMC(f)|$ is uniformly homotopy equivalent to $|MC(f)|$ relative to
$|P\sqcup Q|$.
If additionally $P$ and $Q$ are conditionally complete posets and
$f$ preserves infima, then $HMC(f)$ is a conditionally complete poset.
\end{corollary}

\section{Uniform local contractibility}

\begin{theorem}\label{LCU} If $P$ is a countable poset, then $|P|$ is uniformly
locally contractible.
\end{theorem}

\begin{proof}
Given $x,y\in |P|$ with $d(x,y)<1$, the proof of Theorem \ref{CW} above
produces $x',y'\in|P|$ such that each of the pairs $\{x,x'\}$, $\{x',y'\}$,
$\{y',y\}$ lies in the convex hull of a single chain of $P$, and
$d(x,x')$ and $d(y,y')$ are bounded above by $2d(x,y)$.
We shall modify this pair of discontinuous maps $(x,y)\mapsto x'$,
$(x,y)\mapsto y'$ into a pair of uniformly continuous maps
$\phi$, $\psi$ from the uniform neighborhood $\{(x,y)\mid d(x,y)<\delta\}$ of
the diagonal in $|P|\x |P|$ into $|P|$ such that $d(x,\phi(x,y))$ and
$d(y,\psi(x,y))$ are bounded above by $\frac\eps2$.
Given $\delta$-close maps $f,g\:X\to|P|$, we then define a homotopy $h_t\:X\to|P|$
by $h_0=f$, $h_1=g$, $h_{1/3}(x)=\phi(f(x),g(x))$ and
$h_{2/3}(x)=\psi(f(x),g(x))$ and by linear extension to the remaining values
of $t$.
Then $h_{1/3}$ and $h_{2/3}$ are uniformly continuous as compositions of
uniformly continuous maps, and are $\frac\eps2$-close to $h_0$ and $h_1$,
respectively.
Since each of $h_1$, $h_{2/3}$ and $h_{1/3}$ is $(\frac\eps2+\delta)$-close to
$h_0$, we infer that $h_t$ is a uniformly continuous
$(\frac\eps2+\delta)$-homotopy.

It remains to construct $\phi$ and $\psi$.
Pick some $x,y\in |P|$ with $d(x,y)<\delta$, and let $A\incl P$ and $B\incl P$ be
some finite chains whose convex hulls contain $x$ and $y$ respectively.
Arguing as in the proof of Theorem \ref{CW}, we may enlarge $P$ to $\hat P=C^*CP$
and consider the unique isomorphisms $a\:[m]\to\hat A\incl\hat P$ and
$b\:[n]\to\hat B\incl\hat P$.
Thus $a(1)=\hat 0=b(1)$ and $a(m)=\hat 1=b(n)$.
Let $(k_1,l_1),\dots,(k_r,l_r)$ and $(k'_1,l'_1),\dots,(k'_{r'},l'_{r'})$, and
$Z$ and $Z'$, and $\alpha_i$, $\beta_i$ be as in the proof of Theorem
\ref{isometry}.
We recall that $\alpha_i=x(s)$ for any $s\in a(i)\but a(i-1)$, $1<i\le m$, and
$\beta_i=y(s)$ for any $s\in b(i)\but b(i-1)$, $1<j\le n$.

The basic problem with the original construction of $x'$, $y'$ in the proof of
Theorem \ref{CW} is that they depend on the choice of $A$, $B$.
But they should not if $\phi$ and $\psi$ are to be continuous; indeed, if $A$, $B$
are taken to be the {\it smallest} chains whose convex hulls contain $x$, $y$
respectively, then a pair $(\tilde x,\tilde y)$ arbitrarily close to $(x,y)$ can
give rise to a different pair of chains $(\tilde A,\tilde B)$.

Let $\delta$ be such that $\delta\le\frac\eps6$ and $N\bydef \frac 1{4\delta}\in\Z$.
Let $u_i,u'_i\in[m]$ and $v_i,v'_i\in[n]$ be the maximal numbers such that
$\alpha_{u_i}\ge 1-4i\delta$, $\alpha_{u'_i}\ge 1-(4i+1)\delta$,
$\beta_{v_i}\ge 1-(4i+2)\delta$ and $\beta_{v'_i}\ge 1-(4i+3)\delta$.
Thus $\alpha_{u_0}=1$ and $u_0\ge 2$, whereas $\alpha_{u_N}=\alpha_{u_N-1}=0$
and $u_N=m$.
It is easy to see\footnotemark\ that for each $\kappa\in [m]$ there exists
a $\lambda\in[n]$ such that $a(\kappa)\le b(\lambda)$ and
$\beta_\lambda\ge\alpha_\kappa-\delta$.
\footnotetext{Indeed, let $i$ be the minimal number satisfying $k_i\ge\kappa$,
and let $\lambda=l_i$.
If $i=1$ then $\beta_\lambda=\alpha_\kappa=1$.
Else $\kappa>k_{i-1}$, hence
$\beta_\lambda=\beta_{l_i}\ge\alpha_{k_{i-1}+1}-d(x,y)\ge\alpha_\kappa-\delta$.
(The first inequality was established in the proof of Theorem \ref{isometry}.)}
Similarly, for each $\lambda\in[n]$ there exists a $\kappa\in[m]$ such that
$b(\lambda)\le a(\kappa)$ and $\alpha_\kappa\ge\beta_\lambda-\delta$.
Hence each $a(u'_i)\le b(v_i)$ and each $b(v'_i)\le a(u_{i+1})$.
Thus we get an `intermediary' chain $C$ consisting of:
\begin{alignat*}{10}
\hat 0&\le &a(u_1) &\le &a(u_1+1) &\le &\dots &\le & a(u'_1)\\
&\le &b(v_1) &\le &b(v_1+1) &\le &\mspace{2mu}\dots &\le & b(v'_1)\\
&\le\ &a(u_2) &\le\ &a(u_2+1) &\le\ &\dots &\le\ & a(u'_2)\\
&\le &\cdots&\cdots&\cdots\cdots\cdots&\cdots&\cdots&\cdots&\cdots\ \ \;\\
\end{alignat*}
(It should be noted that if $u_i\in Z$, then the inequalities
$u_i\le u_i'\le\dots\le u_{i+k}\le u_{i+k}'$ may all happen to be equalities for
an arbitrarily large $k$.
This is the only way that it can happen, for it is easy to see\footnotemark\ that
if $a(\kappa)\notin B$, then $\alpha_{\kappa+1}\ge\alpha_\kappa-2\delta$.)
\footnotetext{Let $i$ be the minimal number satisfying $k_i\ge\kappa$.
By the hypothesis $i\ne 1$.
Then $\kappa>k_{i-1}$, hence
$\alpha_\kappa\le\alpha_{k_{i-1}+1}\le\beta_{l_i}+d(x,y)$.
Next let $j$ be the minimal number satisfying $l'_j\ge l_i$.
Then $j\ne 1$ due to $i\ne 1$.
Hence $l_i>l'_{j-1}$, so $\beta_{l_i}\le\beta_{l'_{j-1}+1}\le\alpha_{k'_j}+d(x,y)$.
Thus $\alpha_\kappa\le\alpha_{k'_j}+2\delta$.
Now if $k'_j>k_i$, then $k'_j>\kappa$ due to $k_i\ge\kappa$.
If $k'_j\le k_i$, then $a(k_i)\le b(l_i)\le b(l'_j)\le a(k'_j)\le a(k_i)$,
implying $k'_j=k_i$ and $i\in Z$.
The latter implies $k_i\ne\kappa$ in view of the hypothesis.
Then $k_i>\kappa$, and so $k'_j>\kappa$ once again.
Thus we obtain that $\alpha_{k'_j}\le\alpha_{\kappa+1}$.}

If we use this chain $C$ to construct $x'$ and $y'$ as in the proof of
Theorem \ref{CW}, the result will no longer depend on the choice of $A$ and $B$.
However, the definition of $C$ now involves the maximum function, which is
discontinuous; so an arbitrarily small change in the coordinates of $x$ can
lead to a significant (even though bounded above by $\delta$) change in the
coordinates of $x'$.

Thus we need a new construction of $x'$ and $y'$ that would compensate for
the discontinuity of the maximum function.
We set $x'(s)=1-4i\delta$ for all $s\in a(u_i)\but a(u'_{i-1})$, and (not entirely
symmetrically) $y'(s)=1-4i\delta$ for all $s\in b(v_i)\but b(v'_{i-1})$.
We shall define $x'(s)$ and $y'(s)$ for the remaining values of $s$ by
distributing the total jump value $4\delta$ (between e.g.\ $x'(s)$ and $x'(t)$ for
$s\in a(u_i)\but a(u'_{i-1})$ and $t\in a(u_{i+1})\but a(u'_i)$, provided
that such $s$ and $t$ exist) over all the jumps so as to best
approximate the (continuous) uniform distribution.
Thus the jump value over $a(j)\but a(j-1)$ must be proportional to the
step length $\alpha_{j-1}-\alpha_j$ for each $j\in [u_i+1,u'_i+1]$.
The total horizontal length of the stairs is $\delta$ (from $1-4i\delta$ to
$1-(4i+1)\delta$, for instance).
Therefore we set $x'(s)=(1-4i\delta)-4((1-4i\delta)-\alpha_j)$ for all
$s\in a(j)\but a(j-1)$, for each $j\in [u_i+1,u'_i]$.
Similarly (but not entirely symmetrically)
$y'(s)=(1-4i\delta)-4((1-(4i+2)\delta)-\beta_j)$ for all
$s\in b(j)\but b(j-1)$, for each $j\in [v_i+1,v'_i]$.
We define $\phi(x,y)=x'$ and $\psi(x,y)=y'$.
We also define $\alpha'_j=x'(s)$ for all $s\in a(j)\but a(j-1)$, where $2\le j\le m$,
and $\beta'_j=y'(s)$ for all $s\in b(j)\but b(j-1)$, where $2\le j\le n$
(beware that this notation is not entirely analogous to that in the proof of
Theorem \ref{CW}).
Then $\alpha'_2=\alpha'_{u_0}=1$ and $\alpha'_{m-1}=\alpha'_{u_N}=0$,
where $u_0\ge 2$ and $u_N=m$, so $x'\in|A|_{a}$.
Due to the non-symmetric definition of $y'$, also $\beta'_2=\beta_{v_0}=1$ and
$\beta'_{n-1}=\beta'_{v_{N-1}}=0$, where $v_0\ge 2$ and $v_{N-1}\le n$, so
$y'\in |B|_b$.

It is easy to check that $x'$ and $y'$ do not depend on the choice
of $A$ and $B$.
When $s\in a(u_i)\but a(u'_{i-1})$ we have
$x(s)\in[\alpha_{u_i},\alpha_{u'_{i-1}+1}]\subset[1-4i\delta,1-(4i-3)\delta]$
whereas $x'(s)=1-4i\delta$.
When $s\in a(u'_i)\but a(u_i)$ we have $x'(s)-D=4[x(s)-D]$, where $D=1-4i\delta$,
so $x'(s)-x(s)=[x'(s)-D]-[x(s)-D]=3[x(s)-D]\in [0,3\delta]$.
In both cases $|x'(s)-x(s)|\le 3\delta\le\eps/2$ as desired.
Similarly (but not entirely symmetrically) $|y'(s)-y(s)|\le 3\delta\le\eps/2$.

It remains to verify that $\phi$ and $\psi$ are uniformly continuous, that is,
for each $\zeta>0$ there exists an $\eta>0$ such that $d(x,\tilde x)<\eta$ and
$d(y,\tilde y)<\eta$ imply $d(x',\tilde x')<\zeta$ and $d(y',\tilde y')<\zeta$.
By the proof of Theorem \ref{CW} for each $\theta>0$ there exists an $\eta>0$
(namely, $\eta=\theta/5$) such that given $x,\tilde x\in|P|$ with
$d(x,\tilde x)\le\eta$, there exist $x^*,\tilde x^*\in |P|$ such that each of
the pairs $\{x,x^*\}$,
$\{x^*,\tilde x^*\}$, $\{\tilde x^*,\tilde x\}$ has diameter at most $\theta$
and lies in the convex hull of some chain of $P$.
Given $y,\tilde y\in|P|$ with $d(y,\tilde y)\le\eta$, we similarly get
$y^*,\tilde y^*$.
Therefore it suffices to consider the case where the pairs $\{x,\tilde x\}$ and
$\{y,\tilde y\}$ lie in the convex hulls of some chains $A$ and $B$, respectively.
Since $\phi$ and $\psi$ are well-defined, we may assume that
$x',\tilde x',y',\tilde y'$ are all defined using these $A$ and $B$.
In this case, we set $\eta=\min(\zeta/4,\delta/2)$.

Thus suppose that $d(x,\tilde x)<\eta$.
In other words, $|\alpha_j(s)-\tilde\alpha_j(s)|<\eta$ for all $s\in\mathcal P$.
Fix some $j$; by symmetry we may assume that $\alpha_j(s)\ge\tilde\alpha_j(s)$.
Since $\eta<\delta$, one of the following four cases has to occur for some $i$:
\begin{roster}
\item $1-4i\delta>\alpha_j(s)\ge\tilde\alpha_j(s)\ge 1-(4i+1)\delta$;
\item $1-(4i-3)\delta>\alpha_j(s)\ge\tilde\alpha_j(s)\ge 1-4i\delta$;
\item
$1-(4i-3)\delta>\alpha_j(s)\ge 1-4i\delta>\tilde\alpha_j(s)\ge 1-(4i+1)\delta$;
\item
$1-4i\delta>\alpha_j(s)\ge 1-(4i+1)\delta>\tilde\alpha_j(s)\ge 1-(4i+4)\delta$.
\end{roster}

In the case (i), we have
$\alpha'_j(s)-\tilde\alpha'_j(s)=4(\alpha_j(s)-\tilde\alpha_j(s))$.
In the case (ii), both $\alpha'_j(s)$ and $\tilde\alpha'_j(s)$ equal $1-4i\delta$.
In the case (iii), $\alpha'_j(s)=1-4i\delta$, whereas
$\tilde\alpha'_j(s)\in [1-4i\delta,1-4i\delta-4\eta]$.
Similarly, in the case (iv), $\tilde\alpha'_j(s)=1-(4i+4)\delta$, whereas
$\alpha'_j(s)\in [1-(4i+4)\delta+4\eta,1-(4i+4)\delta]$.

In all cases, $\alpha'_j(s)-\tilde\alpha'_j(s)\in [0,4\eta]$.
This shows that $d(x',\tilde x')\le 4\eta$.
Thus $\phi$ is uniformly continuous; the uniform continuity of $\psi$ is verified
similarly.
\end{proof}

\begin{example}\label{counterexample2}
Given a preposet $P=(\mathcal P,\le)$, we define the ``co-deleted prejoin''
$P\boxplus P^*$ to be the preposet $(\mathcal P\sqcup\mathcal P^*,\preceq)$,
where $\mathcal P^*=\{p^*\mid p\in P\}$ is a just fancy notation for a copy
of $\mathcal P$, and the relation is defined by
\begin{itemize}
\item $p\preceq q$ iff $p\le q$;
\item $p^*\preceq q^*$ iff $p\ge q$;
\item $p^*\preceq q$ never holds;
\item $p\preceq q^*$ iff either $p\le q$ or $p\ge q$
\end{itemize}
for all $p,q\in\mathcal P$.
Note that $P\boxplus P^*$ need not be a poset even if $P$ is.

Let us define $j\:P\to(P\boxplus P^*)^\#$ by $j(p)=[p,p^*]$.
Obviously, $j$ is a monotone embedding, i.e.\ $p\le q$ if and only if
$j(p)\preceq j(q)$.
We claim that $|j|$ is a homotopy equivalence.
Indeed, $|j|$ is homotopic to the composition $|P|\xr{|i|}|P|\xr{h}|P^\#|$,
where $h$ is the uniform homeomorphism.
On the other hand, $i^\flat$ is split by the simplicial map
$r\:(P\boxplus P^*)^\flat\to P^\flat$, defined on vertices by $p,p^*\mapsto p$.
Given a chain $\sigma=(p_1<\dots<p_n)\in P^\flat$ we have
$r^{-1}(\fll\sigma\flr)=\fll\sigma\boxplus\bar\sigma\flr$, where
$\sigma\boxplus\bar\sigma$ denotes the chain
$(p_1<\dots<p_n<p_n^*<\dots<p_1^*)\in (P\boxplus P^*)^\flat$.
Since $r$ is simplicial, it follows that
$|r|\:|(P\boxplus P^*)^\flat|\to|P^\flat|$ has contractible point-inverses,
and therefore (or by Quillen's fiber lemma) is a homotopy equivalence.
If $k$ is a homotopy inverse to $|r|$, then $k=k|ri^\flat|\simeq|i^\flat|$,
so $|i^\flat|$ is also a homotopy equivalence.

Let $K_0$ be the preposet of the four sets $0$, $\{0,1\}$, $\{0,2\}$ and
$\{\{0,1\},\{0,2\}\}$ ordered by $\in$.
Thus $|K_0|$ is homeomorphic to $S^1$.
Let $K_{n+1}=K_n\boxplus K_n^*$.
Finally let $K=K_0\sqcup K_1\sqcup\dots$.

We claim that $|K|$ is not uniformly locally contractible (and in particular is
not a uniform ANR).
Indeed, by the above we have an embedding $f_n\:K_0\to K_n^{\#n}$ such that
$|f_n|$ is a homotopy equivalence.
In order to use the $d_\infty$ metric, which has been shown to work only
for posets, we consider the transitive closure.
Let $f'_n$ be the composition
$K_0\xr{f_n}K_n^{\#n}\subset\left<K_n^{\#n}\right>\subset\left<K_n\right>^{\#n}$.
Since $f'_n$ is monotone, the image of $|f'_n|$ has diameter $1$ in the $d_\infty$
metric on $|\left<K_n\right>^{\#n}|$, hence by Corollary \ref{canonical},
the image of the composition
$|K_0|\xr{|f_n'|}|\left<K_n\right>^{\#n}|\xr{h_n}|\left<K_n\right>|$ has
diameter $2^{-n}$ in the $d_\infty$ metric on $|\left<K_n\right>|$.
Since $\id\:(|\left<K_n\right>|,d_\infty)\to(|\left<K_n\right>|,d)$ is
$1$-Lipschitz, the image of the composition
$|K_0|\xr{|f_n|}|K_n^{\#n}|\xr{h_n}|K_n|$ has diameter at most $2^{-n}$
with respect to the usual metric on $|K_n|$.
However this composition is not null-homotopic since it is a homotopy
equivalence.
\end{example}

We note that the preposet $K$ in Example \ref{counterexample2} satisfies the
following property (P): For each $\eps>0$ there exists an essential map
$S^1\to |K|$ with image of diameter $<\eps$.
On the other hand, since $|K^\flat|$ is a uniform ANR, $|K|$ is a
non-uniform ANR, and in particular satisfies the non-uniform homotopy
extension property.
It follows that every metrizable uniform space that is uniformly homotopy
equivalent to $|K|$ satisfies (P) as well.
In particular, we get the following

\begin{theorem}\label{CCP-homotopy}
There exists a countable preposet whose geometric realization is not uniformly
homotopy equivalent to a uniform ANR, nor even to a uniformly locally contractible
metrizable uniform space.
\end{theorem}

\section{Combinatorics of covers}\label{coverings}

In this subsection we shall need basic operations and relations on covers
as introduced in \cite[\S\ref{metr:cover notation}]{M2}, as well as
the following additional notation.

\begin{definition}[Nerve]
We recall that the {\it nerve} of a cover $C\subset 2^S$ of a set $S$ is
the simplicial poset $N(C)\subset 2^C$, where a subset $B\subset C$ is a simplex
of $N(C)$ iff $\bigcap B$ (the intersection of all elements of $B$) is
nonempty.
The notion of nerve was introduced by Alexandroff \cite{Al}; the poset view
was employed by Segal \cite{Se}.
We note that for a cover $C$,

\medskip
$\bullet$ $C$ countable and point-finite iff $N(C)$ is a simplicial complex;

\smallskip
$\bullet$ $C$ is countable and Noetherian iff $N(C)$ is a Noetherian simplicial
complex;

\smallskip
$\bullet$ $C$ is countable and star-finite iff $N(C)$ is a locally finite
simplicial complex.
\end{definition}

\begin{definition}[Simplex determined by subset]
Given a nonempty $T\subset S$ that is contained in at least one element of $C$, let
$\Delta_C(T)$ denote the element $\{U\in C\mid T\subset U\}$ of $N(C)$.
Given an $s\in S$, we write $\Delta_C(s)=\Delta_C(\{s\})$.
Note that every element of $N(C)$ belongs to some simplex of $N(C)$ of
the form $\fll\Delta_C(s)\flr$ for some $s\in S$.
\end{definition}

\begin{definition}[Cover by open stars]
If $P$ is a poset, by the {\it open star} $\ost(p,P)$ of a $p\in P$ we mean
$|\st(p,P)|\but|\st(p,P)\but\cel p\cer|$.
(We recall that $\st(p,P)=\bigfl\cel p\cer\bigfr$.)

If $P$ is atomic with atom set $\Lambda$, we have the cover
$\{\ost(\lambda,P)\mid\lambda\in\Lambda\}$ of $|P|^\bullet$ by
{\it open stars of vertices.}
It is easy to see that the open star of a vertex $v\in\Lambda$ in $|P|^\bullet$
is precisely the set of points of $|P|^\bullet\subset [0,1]^\Lambda$ whose $v$th
coordinate is nonzero.
On the other hand, the set of points of $|P^\bullet|$ whose $v$th coordinate
equals $1$ is precisely the dual cone $|\cel v\cer|$ of $v$.
Thus we get a cover of $|P|^\bullet$ by {\it dual cones of vertices}, and it
follows that the cover of $|P|^\bullet$ by the open stars of vertices is uniform
with Lebesgue number $1-\eps$ for each $\eps>0$.
\end{definition}

\begin{definition}[Strict shrinking]
If $C$ is a cover of a uniform space $X$, a {\it strict shrinking} of $C$ is a cover
$C'$ of $X$ such that there exist a uniform cover $D$ of $X$ and a bijection
$C\to C'$, denoted $U\mapsto U'$, such that $\st(U',\,D)\subset U$ for each $U\in C$.
\end{definition}

The following is a strengthened statement of \cite[IV.19]{I3}:

\begin{lemma} Let $X$ be a uniform space, $C$ a uniform cover of $X$, and $D$
a uniform barycentric refinement of $C$.
Then $C_D\bydef \{U_D\mid U\in C\}$, where $U_D=X\but\st(X\but U,D)$, is a strict
shrinking of $C$.
\end{lemma}

Beware that $C_D$ is not intended to be a uniform cover here.

\begin{proof}
Obviously, $\st(U_D,D)\subset U$ for each $U\in C$.
On the other hand, since $D$ is a barycentric refinement of $C$,
for every $x\in X$, $\st(x,D)$ lies in some $U\in C$, whence $x\in U_D$.
Thus $C_D$ is a cover of $X$.
\end{proof}

\begin{example}
Let $P$ be a simplicial complex or more generally an atomic poset.
Let $C$ be the cover of $|P|$ by open stars of vertices, and let $C^\#$ denote
the image of the cover of $|P^\#|$ by open stars of vertices under
the homeomorphism $|P^\#|\cong |P|$.
Then $C^\#$ barycentrically refines $C$, and $C_{C^\#}$ is the cover of $|P|$ by dual cones
of vertices.
\end{example}

\begin{lemma}\label{nerve map} Let $X$ be a metrizable uniform space and $C$
a countable point-finite uniform open cover of $X$.

(a) There exists a uniformly continuous map $\phi_C\:X\to |N(C)|^\bullet$ that sends
each $x\in X$ into the interior of $|\fll\Delta_C(x)\flr|^\bullet$.

(b) If $D$ is a uniform barycentric refinement of $C$, then there exists a uniformly
continuous map $\phi_{C,D}\:X\to |N(C)|^\bullet$ that sends each $x\in X$
into the intersection of $|\cel\Delta_{C_D}(x)\cer|^\bullet$ and the interior of
$|\fll\Delta_C(x)\flr|^\bullet$.
\end{lemma}

Note that the conclusion of (a) implies that each $U\in C$ is the preimage of
the open star of the vertex $\{U\}$ of $N(C)$.
The conclusion of (c) implies additionally that each $U_D\in C_D$ lies in
the preimage of the dual cone of the vertex $\{U\}$ of $N(C)$.

\begin{proof}[Proof] Note that (b) implies (a).

To prove (b), let $\lambda$ be a Lebesgue number of $D$.
Then each $d(U_D,\,X\but U)\ge\lambda$.
For each $U\in C$ define $f_U\:X\to [0,1]$ by
$f_U(x)=\min(\lambda^{-1}d(x,\,X\but U),\,1)$.
Let us consider $\phi=\prod f_U\:X\to l_\infty^c$.
Then $\{U\in C\mid f_U(x)>0\}=\Delta_C(x)$, and since $C_D$ is
a cover, $\{U\in C\mid f_U(x)=1\}\ne\emptyset$.
Hence $\phi(X)\subset |N(C)|^\bullet$.
Next, $f_U^{-1}(0)=X\but U$, which implies the assertion on $C$, and
$f_U(U_D)=\{1\}$, which implies the assertion on $C_D$.
Finally, $|f_U(x)-f_U(y)|\le\lambda^{-1} d(x,y)$ for each $U\in C$, so
$\phi$ is uniformly continuous.
\end{proof}

\begin{definition}[Subordinated map]
Let $C$ be a cover of a set $S$, and $f\:S\to|P|^\bullet$ a map, where $P$
is an atomic poset.
We say that $f$ is {\it subordinated} to $C$ if $C$ is refined by
$f^{-1}(D)$, where $D$ is the cover of $|P|^\bullet$ by the open stars of
vertices.
A homotopy $h_t\:S\to|P|^\bullet$ is said to be subordinated to $C$ if
it is through maps subordinated to $C$.
\end{definition}

\begin{lemma} \label{nerve map-1}
Let $X$ be a metrizable uniform space.

(a) If $C$ a point-finite countable uniform cover of $X$, and $D$, $E$ are
uniform barycentric refinements of $C$, then $\phi_{C,D}$ and $\phi_{C,E}$ are uniformly
homotopic by a homotopy subordinated to $C$.

(b) Let $P$ a simplicial complex and $f\:X\to |P|^\bullet$ a uniformly continuous map.
Let $E$ be the cover of $|P|^\bullet$ by open stars of vertices, and let
$C=f^{-1}(E)$ and $D=f^{-1}(C^\#)$.
Then $\phi_{C,D}\:X\to |N(C)|^\bullet\subset |N(E)|^\bullet$ is uniformly homotopic
to $f\:X\to|P|^\bullet=|N(E)|^\bullet$ by a homotopy subordinated to $C$.

(c) Let $C$ be a point-finite countable uniform cover of $X$ and $D$ a uniform
barycentric refinement of $C$.
Let $E_D$ be a subset of $C_D$ that still covers $X$, and $E$ the corresponding subset
of $C$.
Then $\phi_{C,D}\:X\to |N(C)|^\bullet$ is uniformly homotopic to
$\phi_{E,D}\:X\to |N(E)|^\bullet\subset |N(C)|^\bullet$ by a homotopy
subordinated to $C$.
\end{lemma}

\begin{proof}[Proof. (a)]
By considering $\{U_D\cup U_E\mid U\in C\}$ we may assume that each element of
$C_D$ is contained in the corresponding element of $C_E$.
Then $\Delta_{C_D}(x)\subset\Delta_{C_E}(x)$ for each $x\in X$.
Hence $[\Delta_{C_E}(x),\Delta_C(x)]\subset
[\Delta_{C_D}(x),\Delta_C(x)]$.
Therefore $\phi_{C,D}(x)$ and $\phi_{C,E}(x)$ both belong to
$|[\Delta_{C_D}(x),\Delta_C(x)]|^\bullet$.
Thus the linear homotopy between $\phi_{C,D}$ and $\phi_{C,E}$ in $q_{00}$
has values in $|N(C)|^\bullet$.
Since both $\phi_{C,D}$ and $\phi_{C,E}$ are subordinated to $C$, so is the
homotopy.
\end{proof}

\begin{proof}[(b)]
By Lemma \ref{nerve map}, $\phi_{C,D}(x)\in
|[\Delta_{C_D}(x),\Delta_C(x)]|^\bullet$ for each $x\in X$.
Clearly, $f(x)\in|[\Delta_{C_D}(x),\Delta_C(x)]|^\bullet$ for each $x\in X$.
Thus the linear homotopy between $f$ and $\phi_{C,C_D}$ in $q_{00}$
has values in $|N(C)|^\bullet$.
Since both $f$ and $\phi_{C,C_D}$ are subordinated to $C$, so is the
homotopy.
\end{proof}

\begin{proof}[(c)]
Given an $x\in X$, we have $\Delta_{E_D}(x)\subset\Delta_{C_D}(x)$ and
$\Delta_E(x)\subset\Delta_C(x)$.
Hence $\Delta_{C_D}(x),\Delta_E(x)\in [\Delta_{E_D}(x),\Delta_C(x)]$.
Therefore $\phi_{C,D}(x)$ and $\phi_{E,D}(x)$ both belong to
$|[\Delta_{E_D}(x),\Delta_C(x)]|^\bullet$.
Thus the linear homotopy between $\phi_{C,D}$ and $\phi_{E,D}$ in $q_{00}$
has values in $|N(C)|^\bullet$.
Since both $\phi_{C,D}$ and $\phi_{E,D}$ are subordinated to $C$, so is the
homotopy.
\end{proof}

\begin{definition}[Intersection poset and Euler diagram]
Given a cover $C$ of a set $S$, the {\it intersection poset} $IP(C)$ is
the subposet of $2^C$ consisting of all nonempty $B\subset C$ such that $\bigcap B$
is not contained in any element of $C\but B$.
The terminology ``intersection poset'' derives from Lemma \ref{IP}(a) below,
which however characterizes $IP(C)$ only up to isomorphism, and not as a subposet
of $\Delta^C$.

The {\it Euler diagram} $ED(C)$ is the subposet of $2^C$ consisting of
all $B\subset C$ such that $\bigcap B$ is not contained in $\bigcup (C\but B)$.
This is a formalization of the intuitive notion of an ``Euler diagram'' from courses of ``abstract 
mathematics'' (already known to Leibniz, and further popularized by Venn, see \cite{Ba}), for it
can be argued that $ED(C)$ contains all the combinatorial information on containment
of points of $X$ in elements of $C$ (see Lemma \ref{IP}(b) below) --- and nothing
else (see Lemma \ref{VD}(b) below).

Clearly, $ED(C)\subset IP(C)\subset N(C)$ and $\fll ED(C)\flr=\fll IP(C)\flr=N(C)$.
Note that if $C$ is countable and point-finite, then $IP(C)$ and $ED(C)$
are cone complexes.
Clearly $\Delta_C(x)$ belongs to $ED(C)$ for every $x\in S$; in contrast,
$\Delta_C(T)$ belongs to $IP(C)$ for every $T\subset S$ that is contained in at
least one element of $C$.
\end{definition}

\begin{lemma}\label{IP} Let $C$ be a cover of a set $S$.

(a) $IP(C)$ is isomorphic to the poset consisting of arbitrary nonempty
intersections of elements of $C$, ordered by reverse inclusion.
In particular, $IP(C)$ is conditionally complete.

(b) $ED(C)$ is isomorphic to the poset consisting of those intersections
of elements of $C$ that are of the form $\bigcap\Delta_C(s)$ for some $s\in S$,
ordered by reverse inclusion.

(c) There is a canonical retraction $N(C)\to IP(C)$.
\end{lemma}

\begin{proof}[Proof. (a,b)]
It will be convenient to work in a slightly greater generality.
The definitions of $N(C)$, $IP(C)$ and $ED(C)$ generalize straightforwardly for any
collection $\phi\:C\to 2^S$ of subsets of $S$ (possibly with repeated subsets
$\phi(U)=\phi(U')$ and with $\bigcup\phi(C)$ not necessarily covering the whole
of $S$).
It is easy to see that

\medskip
$\bullet$ $N(\phi)=\Phi^{-1}(\Delta^S)$, where $\Phi\:\Delta^C\to 2^S$ is defined
by $D\mapsto\bigcap\phi(D)$;

\smallskip
$\bullet$ $IP(\phi)=\Delta_\phi(\Delta^S)\but\{\emptyset\}$, where
$\Delta_\phi\:\Delta^S\to 2^C$
is defined by $T\mapsto\phi^{-1}(\cel T\cer)$;

\smallskip
$\bullet$ $ED(\phi)=\Delta_\phi(A(\Delta^S))\but\{\emptyset\}$, where
the subset $A(\Delta^S)=\{\{s\}\mid s\in S\}$ of $\Delta^S$ should not
be confused with the element $S$ of $\Delta^S$.

\medskip
\noindent
We note that the maps $\Phi$ and $\Delta_\phi$ are anti-monotone, and restrict
to mutually inverse bijections between $IP(\phi)$ and
$\Phi(\Delta^C)\but\{\emptyset\}$.
In particular, $IP(\phi)$ is isomorphic to
$(\Phi(\Delta^C)\but\{\emptyset\})^*$, which implies the first assertion of (a).
Similarly, $ED(\phi)$ is isomorphic to
$\Phi(\Delta_\phi(A(\Delta^S))\but\{\emptyset\})^*$, which yields (b).
\end{proof}

\begin{remark} It follows from the proof that $IP(\phi)\simeq IP(\Phi)$ and
$ED(\phi)\simeq ED(\Phi)$.
\end{remark}

\begin{proof}[Proof of (c)]
Given a $B\in N(C)$, i.e., a $B\subset C$ such that $\bigcap B\ne\emptyset$,
let $B'\subset C$ consist of all elements of $C\but B$ that contain $\bigcap B$.
Then $\bigcap B=\bigcap (B\cup B')$, and $B\cup B'\in IP(C)$.
The map $N(C)\to IP(C)$, $B\mapsto B\cup B'$ is monotone: if $B_1\subset B_2$,
then $\bigcap B_2\subset\bigcap B_1$, whence $B_1'\but B_2\subset B_2'$; thus
$B_1\cup B_1'\subset B_2\cup B_1'\subset B_2\cup B_2'$.
Clearly, if $B\in IP(C)$, then $B'=\emptyset$, so this map is a retraction. 
\end{proof}

\begin{remark} (a) Every point-finite open cover $C$ of a topological space $X$ clearly yields 
a continuous map $X\to ED(C)$, $x\mapsto\Delta_C(x)$, where the underlying set of 
the poset $ED(C)$ is endowed with the Alexandroff topology.

(b) If $C$, $D$ are covers of a set $S$ and $C$ {\it partitions} $D$, i.e., every element of $D$ is 
a union of elements of $C$, then there is a canonical monotone
map $f_{CD}\:ED(C)\to ED(D)$, $\Delta_C(x)\mapsto\Delta_D(x)$.
If further $D$ partitions a cover $E$ of $S$, then clearly $f_{DE}\circ f_{CD}(B)=f_{CE}(B)$.

(c) If $C$, $D$ are covers of a set $S$ and $C$ refines $D$, then there is a canonical monotone
map $f_{CD}\:IP(C)\to IP(D)$, $B\mapsto\Delta_D(\bigcap B)$.
If further $D$ refines a cover $E$ of $S$, then clearly $f_{DE}\circ f_{CD}(B)\subset f_{CE}(B)$.
\end{remark}

\begin{lemma}\label{VD}
Let $P$ be a poset embedded in some $\Delta^\Lambda$.
Let $C$ be the cover $\{\cel\lambda\cer\cap P\mid\lambda\in\Lambda\}$ of
the underlying set of $P$ by the dual cones of vertices.

(a) $ED(C)\simeq P$.

(b) $IP(C)=ED(C)$ if and only if for every $R\subset A(\Delta^\Lambda)$, the set
of all upper bounds of $R$ in $P$ either is empty or is the dual cone in $P$
of a single element of $P$.
\end{lemma}

\begin{proof}[Proof. (a)] By Lemma \ref{IP}(b), $ED(C)$ is isomorphic to
the poset consisting of the dual cones
$\cel p\cer^{\Delta^\Lambda}\cap P=\cel p\cer^P$ of all elements $p\in P$,
ordered by reverse inclusion.
The latter is obviously isomorphic to $P$.
\end{proof}

\begin{proof}[(b)] By Lemma \ref{IP}(a), $IP(C)$ is isomorphic to the poset of
all nonempty intersections of the form $\bigcap\sigma$, where $\sigma\subset C$,
ordered by reverse inclusion.
We have $\bigcap\sigma=P\cap\bigcap_{\lambda\in R_\sigma}\cel\lambda\cer$, where
$R_\sigma=\{\{\lambda\}\mid\lambda\in\Lambda,\,(\cel\lambda\cer\cap P)\in\sigma\}$.
(The subset $R_\sigma$ of $\Delta^\Lambda$ should not be confused with
the element $\bigcup R_\sigma$ of $\Delta^\Lambda$.)
Thus $IP(C)$ is in bijection with the set of all nonempty intersections of
the form $P\cap\cel R\cer$, where $R\subset\Lambda$.
By the proof of Lemma \ref{IP}, the same bijection sends $ED(C)$ onto
the set of the dual cones $\cel p\cer^P$ of all elements $p\in P$, and the
assertion follows.
\end{proof}

\begin{corollary} Let $P$ be an atomic poset, and let $C$ be the cover of
$P$ by the dual cones of its atoms.
Then $IP(C)=ED(C)$ if and only if $P$ is conditionally complete.
\end{corollary}

It is not hard to see that the same assertion is true of the cover of $|P|$ by
the geometric realizations of the dual cones of the vertices of $P$, and of
the cover of $|P|$ by the open stars of these vertices.

\begin{proof} Let us embed $P$ in $\Delta^{A(P)}$ as in
\cite[Lemma \ref{comb:2.9}]{M1}.
Lemma \ref{VD}(b) then says that $IP(C)=ED(C)$ if and only if every $R\subset A(P)$
that has an upper bound in $P$ has a least upper bound in $P$.
This proves the ``if'' assertion, and the ``only if'' assertion now follows from
\cite[Lemma \ref{comb:atomic CCP}]{M1}.
Alternatively, the ``only if'' assertion follows from Lemma \ref{VD}(a)
and the second assertion of Lemma \ref{IP}(a).
\end{proof}

\begin{definition}[Canonical bonding map]\label{canonical bonding}
Let $C$ and $D$ be covers of a set $S$, and suppose that $C$ barycentrically refines $D$.
We define a map $\phi^C_D\:N(C)\to N(D)^\#$ by sending each $\sigma\in N(C)$
into $[\Delta_D(\bigcup\sigma),\Delta_D(\bigcap\sigma)]\in N(D)^\#$.
Here $\Delta_D(\bigcup\sigma)$ is non-empty by the barycentric refinement hypothesis;
every vertex of $\Delta_D(\bigcup\sigma)$ is obviously a vertex of
$\Delta_D(\bigcap\sigma)$; and $\Delta_D(\bigcap\sigma)\in N(D)$ since
$\bigcap\sigma\ne\emptyset$.

Given a $\tau\le\sigma$, we have $\bigcup\tau\subset\bigcup\sigma$, whence
$\Delta_D(\bigcup\tau)\ge\Delta_D(\bigcup\sigma)$; and
$\bigcap\tau\supset\bigcap\sigma$, whence
$\Delta_D(\bigcap\tau)\le\Delta_D(\bigcap\sigma)$.
Thus $\phi^C_D$ is monotone.

Finally, recall that $IP(D)$ contains every element of $N(D)$ of the form
$\Delta_D(T)$, where $T\subset S$.
Hence $[\Delta_D(\bigcup\sigma),\Delta_D(\bigcap\sigma)]$ belongs to
the isomorphic copy of $IP(D)^\#$ in $N(D)^\#$.
Thus we may write $\phi^C_D\:N(C)\to IP(D)^\#$.
\end{definition}

\begin{remark} Let us discuss some motivation/geometry behind the definition of
$\phi^C_D$.

If $V\in\sigma$ then $\{V\}\le\sigma$, so by the above $\Delta_D(V)$ belongs to
$[\Delta_D(\bigcup\sigma),\Delta_D(\bigcap\sigma)]$.
Hence the map defined on vertices by $\{V\}\mapsto [\Delta_D(V),\Delta_D(V)]$ extends
uniquely to a monotone map $N(C)\to N(D)^\#$ sending each simplex $\fll\sigma\flr$
into the cube $\fll[\Delta_D(\bigcup\sigma),\Delta_D(\bigcap\sigma)]\flr$.

If $C$ star-refines $D$, then $\Delta_D(\st(V,C))\in N(C)$, and it is
not hard to see that $\phi^C_D$ sends each $\st(\{V\},N(C))$ into the canonical
subdivision of the dual cone of $\Delta_D(\st(V,C))$ in the subcomplex
$\bigcup\{\fll\Delta_D(x)\flr\mid x\in V\}$ of $N(D)$.
\end{remark}

\begin{remark}\label{canonical-barycentric}
Let us note that $\phi^C_D\:N(C)\to N(D)^\#$ induces a simplicial map between the barycentric
subdivisions $N(C)^\flat\to \big(N(D)^\#\big)^\flat$. 
In fact, here $\big(N(D)^\#\big)^\flat$ is a subdivision of $N(D)^\flat$.

In general, let us show that $(P^\#)^\flat$ is a subdivision of $P^\flat$ for every preposet $P$.
Every simplex of $(P^\#)^\flat$ is a chain $[p_1,q_1]\subsetneqq\dots\subsetneqq[p_n,q_n]$ of intervals of $P$.
Such a chain of intervals determines a chain $p_n\subset\dots\subset p_1\subset q_1\subset\dots\subset q_n$
of $P$ (where the inclusions are not necessarily proper).
Moreover, a subchain $[p_{i_1},q_{i_1}]\subsetneqq\dots\subsetneqq[p_{i_k},q_{i_k}]$ of the chain of intervals, 
where $1\le i_1<\dots<i_k\le n$, determines a subchain 
$p_{i_k}\subset\dots\subset p_{i_1}\subset q_{i_1}\subset\dots\subset q_{i_n}$ of the chain of $P$.
Thus we get a monotone map $f\:(P^\#)^\flat\to P^\flat$.
Every $n$-simplex $\fll\sigma\flr$ of $P^\flat$ is the isomorphic image of the simplex $[n+1]^\flat$ under 
the map induced by a monotone injection $[n+1]\to P$.
Clearly, $f^{-1}(\fll\sigma\flr)$ is the isomorphic image of $K_n\bydef ([n+1]^\#)^\flat$ and 
$f^{-1}(\partial\fll\sigma\flr)$ is the isomorphic image of 
$L_n\bydef \bigcup_{i=1}^{n+1}\big(([n+1]\but\{i\})^\#\big)^\flat$.
But it is easy to see that $|K_n|$ is a ball with boundary $|L_n|$.
\end{remark}

\begin{proposition}
Let $X$ be a metrizable uniform space, $P$ and $Q$ simplicial complexes, and
$f\:X\to |P|$ and $g\:X\to|Q|$ uniformly continuous maps.
Let $C_P$ and $C_Q$ be the covers of $|P|$ and $|Q|$ by the open stars of vertices,
and let $D\bydef g^{-1}(C_Q)$.
If $C\bydef f^{-1}(C_P)$ refines $E\bydef g^{-1}(C_Q^\#)$, then the composition
$X\xr{f}|P|\xr{|\phi^C_D|}|Q^\#|\cong|Q|$ is uniformly homotopic to $g$ by
a homotopy subordinated to $D$.
\end{proposition}

Note that $C$ barycentrically refines $D$, since $E$ does; thus $\phi^C_D$ is defined.

\begin{proof}
Given an $x\in X$, we have $g(x)\in|[\Delta_{D_E}(x),\Delta_D(x)]$.
On the other hand, $f(x)\in|\fll\Delta_C(x)\flr|$.
Hence $\phi^C_D(f(x))\in
|[\Delta_D(\bigcup\Delta_C(x)),\Delta_D(\bigcap\Delta_C(x))]|$.
Since $x\in\bigcap\Delta_C(x)$, we have $\Delta_D(\bigcap\Delta_C(x))\subset\Delta_D(x)$.
On the other hand, since $C$ refines $E$, and $\st(U_E,E)\subset U$ for each $U\in D$,
we have that $x\in U_E$ implies $\bigcup\Delta_C(x)\subset U$.
Hence $\Delta_{D_E}(x)\subset\Delta_D(\bigcup\Delta_C(x))$.
Thus both $g(x)$ and $\phi^C_D(f(x))$ lie in $|[\Delta_{D_E}(x),\Delta_D(x)]$.
Then the linear homotopy between $g$ and $\phi^C_D f$ in $q_{00}$
has values in $|N(D)|\subset|Q|$.
Since both $g$ and $\phi^C_D$ are subordinated to $D$, so is the homotopy.
\end{proof}

\begin{proposition} Let $C$ and $D$ be covers of a set $S$.
If $C$ star-refines $D$, there exists a simplicial map
$N(C)\xr{g} N(D)$ such that $N(C)^\#\xr{g^\#}N(D)^\#$ is homotopic to
the composition $N(C)^\#\xr{(\phi^C_D)^\#}N(D)^{\#\#}\xr{r^\#}N(D)^\#$, where
$r$ stands for either $r_{N(D)}$ or $r^*_{N(D)}$, by a monotone homotopy
$N(C)^\#\x I\to N(D)^\#$ sending each block of the form
$\fll\sigma\flr^\#\x I$ into the star of
$[\Delta_D(\bigcup\sigma),\Delta_D(\bigcap\sigma)]$ in $N(D)^\#$.
\end{proposition}

\begin{proof}
Given a vertex $\{V\}$ of $N(C)$, the hypothesis furnishes a vertex
$\{V'\}$ of $N(D)$ such that $\st(V,C)\subset V'$.
For each $\sigma\in N(C)$ and each $V\in\sigma$ we have $\bigcup\sigma\subset V'$,
and consequently $\{V'\}\in\fll\Delta_D(\bigcup\sigma)\flr$.
Hence $V\mapsto V'$ extends to a simplicial map $g\:N(C)\to N(D)$ sending each
$\sigma$ onto some
$\sigma'\subset\Delta_D(\bigcup\sigma)\subset\Delta_D(\bigcap\sigma)$.
Then $g^\#$ sends each $[\sigma,\tau]$ onto $[\sigma',\tau']$.
The composition
$N(C)^\#\xr{(\phi^C_D)^\#}N(D)^{\#\#}\xr{r^\#}N(D)^\#$ sends it onto
$[\Delta_D(\bigcup\sigma),\Delta_D(\bigcup\tau)]$ when $r=r_{N(D)}$
and onto $[\Delta_D(\bigcap\sigma),\Delta_D(\bigcap\tau)]$ when $r=r^*_{N(D)}$.
The required homotopy is defined by sending each
$([\sigma,\tau],\{0,1\})$ onto $[\sigma',\Delta_D(\bigcup\sigma)]$ in the first
case and onto $[\sigma',\Delta_D(\bigcap\sigma)]$ in the second case.
\end{proof}

\begin{example}\label{bad bonding map}
Here is a simple example of a canonical bonding map that
is neither closed nor infima-preserving.

Let $C=\{U_1,\,U_2,\,U_3\}$ be the cover of $\Delta^2$ by the open stars of
vertices.
Let $D=\{U_1,\ U_2\cup U_3,\ U_1\cup U_2\cup U_3\}$.
Then $C$ barycentrically refines $D$.
The canonical bonding map $\phi^C_D\:N(C)\to N(D)^\#$ is the composition
of the simplicial surjection $\Delta^2\to\Delta^1$ (see \cite[Example
\ref{comb:non-CCP MC}]{M1}), the diagonal embedding $\Delta^1\to\Delta^1\x\Delta^1$
(see \cite[Example \ref{comb:diagonal}]{M1}), and a subcomplex inclusion
$\Delta^1\x\Delta^1\subset(\Delta^2)^\#$.
\end{example}

\section{Approximation of maps}

\begin{theorem} \label{Hahn} Let $Q$ be a countable conditionally complete poset.
Then $|Q|$ satisfies the Hahn property.
\end{theorem}

\begin{proof}
Given an $\eps>0$, let $C'$ be the cover of $|Q^{\#n}|$ by the open stars of
vertices, where $2^{-n+1}<\eps$ and $n\ge 1$ so that $Q^{\#n}$ is atomic.
Since $Q$ is conditionally complete, so is $Q^{\#n}$, and therefore $IP(C')=ED(C')\simeq Q^{\#n}$.
Let $h_n\:|Q^{\#n}|\to|Q|$ be the uniform homeomorphism given by
Corollary \ref{canonical}, and let $C=h_n(C')$.
Let $\delta$ be the Lebesgue number of $C$ with respect to the
$d_\infty$ metric on $|Q|$.

Given a metric space $X$ and a $(\gamma,\delta)$-continuous map
$f\:X\to |Q|$ for some $\gamma>0$, let $E$ be the cover of $X$ by
$\frac\gamma4$-balls.
Then $E$ barycentrically refines $D\bydef f^{-1}(C)$.
Let $\Phi$ denote the composition
$$X\xr{\phi_E}|N(E)|\xr{|\phi^E_D|} |IP(D)^\#|\subset |IP(C)^\#|\xr{h}|IP(C)|
\cong |Q^{\#n}|\xr{h_n}|Q|,$$
where $\phi_E$ is the uniformly continuous map given by Lemma \ref{nerve map} and
$\phi^E_D$ is the canonical bonding map.
Given an $x\in X$, by Lemma \ref{nerve map} $\phi_E(x)\in|\fll\Delta_E(x)\flr|$.
By the definition of $\phi^E_D$ we have
$\phi^E_D(\Delta_E(x))=[\Delta_D(\bigcup\Delta_E(x)),\Delta_D(\bigcap\Delta_E(x))]
\subset\fll\Delta_D(\bigcap\Delta_E(x))\flr\subset\fll\Delta_D(x)\flr$.
The latter is identified with $\fll\Delta_C(f(x))\flr$, where $\Delta_C(f(x))$
is an element of $IP(C)\simeq Q^{\#n}$, and it follows that
$\Phi(x)\in h_n(|\fll\Delta_C(f(x))\flr|)$.
Now $|\fll\Delta_C(f(x))\flr|$ has diameter $\le 2$ with respect to the
$d_\infty$ metric on $|Q^{\#n}|$, so $h_n(|\fll\Delta_C(f(x))\flr|)$ has
diameter $\le 2^{-n+1}$ with respect to the $d_\infty$ metric on $|Q|$.
Since this set contains both $\Phi(x)$ and $f(x)$, we infer that
$\Phi$ is $\eps$-close to $f$ with respect to the $d_\infty$ metric on $|Q|$.
\end{proof}

We define a {\it uniform polyhedron} to be the geometric realization of
a (countable) conditionally complete poset.
Theorem \ref{LCU}, Theorem \ref{Hahn}, and \cite[Theorem
\ref{metr:LCU+Hahn}]{M2} have the following

\begin{corollary}\label{CCP-ANR} Uniform polyhedra are uniform ANRs.
\end{corollary}

\begin{lemma}\label{equiuniform approximation} For each $\eps>0$ there exist an $n$
and a $\delta>0$ such that for each $\gamma>0$ there exists an $M$ such that for
each $m\ge M$ the following holds.
Let $P$ be a preposet and $Q$ a conditionally complete poset, and $f\:|P|\to |Q|$ be a
$(\gamma,\delta)$-continuous map.
Then there exists a monotone map $g\:P^{\#m}\to Q^{\#n}$ such that the composition
$|P|\xr{h_m^{-1}}|P^{\#m}|\xr{g}|Q^{\#n}|\xr{h_n}|Q|$ is $\eps$-close to $f$.
\end{lemma}

\begin{proof} Let $2^{-n+1}<\eps$, $n\ge 1$, let $\delta<2^{-n-1}$,
and let $2^{-M+1}<\gamma/4$, $M\ge 1$.

Let $C'$ be the cover of $|Q^{\#n}|$ by the open stars of vertices (using that
$Q^{\#n}$ is atomic due to $n\ge 1$).
Since $C'$ has Lebesgue number $\frac12$ with respect to the usual metric
$d$ on $|Q^{\#n}|$, is also has Lebesgue number $\frac12$ with respect to
the $d_\infty$ metric, due to $d(x,y)\le d_\infty(x,y)$.
Then $C\bydef h_n(C')$ has Lebesgue number $2^{-n-1}$ (and therefore also Lebesgue
number $\delta$) with respect to the $d_\infty$ metric on $|Q|$.

Let $E'$ be the cover of $|P^{\#m}|$ by the open stars of vertices (using that
$P^{\#m}$ is atomic due to $m\ge M\ge 1$).
Then $E'$ refines the cover of $|P^{\#m}|$ by balls of radius $2$ about every vertex
of $P^{\#m}$ with respect to the $d_\infty$ metric on $|P^{\#m}|$.
Hence $E\bydef h_m(E')$ refines the cover of $|P|$ by balls of radius $2^{-m+1}$
(and therefore also that by balls of radius $\gamma/4$) about
all points of $|P|$ with respect to the $d_\infty$ metric on $|P|$.
We note that the composition
$\phi\:|P|\xr{h_m^{-1}}|P^{\#m}|\cong|ED(E)|\subset|N(E)|$
satisfies $\phi(x)\in|\fll\Delta_E(x)\flr|$.

The assertion now follows by the proof of Theorem \ref{Hahn}.
\end{proof}

From the preceding lemma we infer the following canonical version of
Brouwer's simplicial approximation theorem:

\begin{theorem}\label{monotone approximation} For each $\eps>0$ there exists an $n$
such that the following holds.
Let $f\:|P|\to |Q|$ be a uniformly continuous map, where $P$ is a preposet
and $Q$ is a conditionally complete poset.
Then there exists an $M$ such that for each $m\ge M$ there exists
a monotone map $g\:P^{\#m}\to Q^{\#n}$ such that
the composition $|P|\xr{h_m^{-1}}|P^{\#m}|\xr{g}|Q^{\#n}|\xr{h_n}|Q|$ is
$\eps$-close to $f$.
\end{theorem}

\begin{example}\label{counterexample3} Let $P_1=[2]$ and let $P_{i+1}=P_i+[2]$.
Finally let $P=\bigsqcup_{n\in\N} P_{2^n+1}$.
We claim that $|P|$ does not satisfy the Hahn property (and in particular is not
a uniform ANR).

Indeed, let $Q_n=(P_{2^n+1})^{\#n}$, and let $C_n$ be the cover of $|Q_n|$
by the stars of atoms of $Q_n$.
Then $|ED(C_n)|\cong|Q_n|\cong|P_{2^n+1}|$ is a $2^n$-sphere; but we
shall now show that $|N(C_n)|$ is contractible.

If $K$ is a poset, then $K^{\#n}$ is isomorphic to the poset consisting of
non-decreasing sequences $a=(a_1\le\dots\le a_{2^n})$ of elements of $K$, where
$a\ge b$ iff $a_i\le b_i$ for all odd $i$ and $b_i\le a_i$ for all even $i$.
Such a sequence $a$ represents an atom of $K^{\#n}$ iff $a_i=a_{i+1}$ for
all odd $i$; and a coatom of $K^{\#n}$ iff $a_i=a_{i+1}$ for all even $i<2^n$,
$a_1$ is an atom of $K$ and $a_{2^n}$ is a coatom of $K$.
Thus the atoms of $K^{\#n}$ can be identified with non-decreasing sequences
$a=(a_2\le a_4\le\dots\le a_{2^n-2}\le a_{2^n})$ of elements of $K$, and
the coatoms of $K^{\#n}$ with non-decreasing sequences
$s=(s_1\le s_3\le\dots\le s_{2^n-1}\le s_{2^n+1})$ of elements of $K$, where $s_1$
is an atom and $s_{2^n+1}$ a coatom of $K$; in this notation, $a\le s$ iff
$s_1\le a_2\le s_3\le a_4\le\dots\le a_{2^n}\le s_{2^n+1}$.
If $C$ is the cover of $|K^{\#n}|$ by the open stars of vertices of $K^{\#n}$,
then the vertices of $N(C)$ correspond to the atoms of $K^{\#n}$, and a set of
vertices of $N(C)$ determines a simplex of $N(C)$ iff the corresponding atoms
of $K^{\#n}$ all belong to the cone of some coatom of $K^{\#n}$.

Consider the projection $\pi\:P_i=[2]+\dots+[2]\to[1]+\dots+[1]\simeq [i]$.
By the above, $N(C_n)\subset 2^A$, where $A$ is the set of all non-decreasing
sequences $a=(a_2\le a_4\le\dots\le a_{2^n})$ of elements of $P_{2^n+1}$, and
$N(C_n)$ consists of all $S\subset A$ such that there exists a non-decreasing
sequence $s=(s_1\le s_3\le\dots\le s_{2^n+1})$ of elements of $P_{2^n+1}$, where
$\pi(s_1)=1$ and $\pi(s_{2^n+1})=2^n+1$, and
$s_1\le a_2\le s_3\le a_4\le\dots\le a_{2^n}\le s_{2^n+1}$.

Let $L_k=\{a\in A\mid \pi(a_{2i})\le 2i\text{ for all }i\le k\}$ and
$R_k=\{a\in A\mid \pi(a_{2i})\ge 2i\text{ for all }i\ge 2^{n-1}+1-k\}$.
Thus $A=L_0\supset L_1\supset\dots\supset L_{2^{n-1}}$, and
$A=R_0\supset R_1\supset\dots\supset R_{2^{n-1}}$.
Note that $L_{2^{n-1}}\cap R_{2^{n-1}}=\{a\in A\mid\pi(a_{2i})=2i\text{ for all }i\}$,
which lies in a single simplex, as witnessed by any sequence $s$ with
$\pi(s)=(1\le 3\le\dots\le 2^n+1)$.
For $i=0,\dots,2^{n-1}$ let $N_i$ be the full subcomplex of $N(C_n)$ spanned
by $L_i$, and for $i=2^{n-1}+1,\dots,2^n$ let $N_i$ be the full subcomplex of
$N(C_n)$ spanned by $L_{2^{n-1}}\cap R_{i-2^{n-1}}$.
Thus $N_0=N(C_n)$; on the other hand, since $N_{2^n}$ is a full simplex,
$|N_{2^n}|$ is contractible.
We shall now construct a deformation retraction of $|N_i|$ onto $|N_{i+1}|$
for each $i=0,\dots,2^n-1$.

We first define a retraction $r_k\:L_{k-1}\to L_k$ by $r_k(a)=a$ if $a\in L_k$,
and else by $r_k(a)=b$, where $b_{2i}=a_{2i}$ for $i\ne k$ and
$\pi(b_{2k})=2k$.
(This leaves two possibilities for $b_{2k}$, among which we choose arbitrarily.)
Let $S\in N_{k-1}$ be witnessed by a sequence $s=(s_1\le s_3\le\dots\le s_{2^n+1})$
of elements of $P_{2^n+1}$, where $\pi(s_1)=1$ and $\pi(s_{2^n+1})=2^n+1$.
If $\pi(s_{2k+1})\le 2k$, then all elements of $S$ belongs to $L_k$, and so
$r_k|_S$ is the identity.
Else $\pi(s_{2k+1})\ge 2k+1$, and since all elements of $S$ belong to $L_{k-1}$,
we may assume that $\pi(s_{2k-1})\le 2k-1$ by modifying $s$ if necessary.
Then $r_k(S)$ is a simplex of $N_k$, and furthermore $S\cup r_k(S)$ is a simplex of
$N_{k-1}$, as witnessed by the same sequence $s$.
Thus $r_k$ extends to a simplicial retraction $R_k\:N_{k-1}\to N_k$, and furthermore
we get a simplicial map $H_k\:N_{k-1}*N_{k-1}\to N_{k-1}$ that restricts to
the identity on the first factor and to $R_k$ on the second factor.
It follows that $|R_k|$ is a deformation retraction.

We next similarly define a retraction $r'_k\:R_{k-1}\to R_k$ by $r'_k(a)=a$ if
$a\in R_k$, and else by $r'_k(a)=b$, where $b_{2i}=a_{2i}$ for $i\ne 2^{n-1}+1-k$
and $\pi(b_{2^n+2-2k})=2^n+2-2k$.
(This leaves two possibilities for $b_{2^n+2-2k}$, among which we choose
arbitrarily.)
Then $r'_k(L_{2^{n-1}})\subset L_{2^{n-1}}$, and then the preceding argument
goes through.
This completes the proof that $|N(C_n)|$ deformation retracts onto a contractible
subspace, and therefore is itself contractible.

Suppose that $|P|$ satisfies the Hahn property.
(The following somewhat technical argument can be somewhat simplified by using
Lemma \ref{6.1} below, modulo the proof of that lemma.)
Since $|P|$ is uniformly locally contractible, there exits an $\eps>0$
such that every two $\eps$-close maps into $|P|$ are homotopic with respect to
the usual metric on $|P|$.
Let $D_n$ be the cover of $|P^{\#n}|$ by the open stars of vertices of $P^{\#n}$.
Since $P\simeq ED(D_0)$ is a subposet $N(D_0)$, we have an isometric
embedding $|P|\subset |N(D_0)|$ with respect to the usual metrics.
We recall that the $d_\infty$ metric on $N(D_0)$ is uniformly equivalent
to the usual metric $d$.
Let $d'$ denote the $d_\infty$ metric on $|N(D_0)|$ restricted over $|P|$.
Let $\delta>0$ be such that $\delta$-close points in $|P|$ with respect to $d'$
are $\eps$-close with respect to $d$.
By our assumption there exists a $\gamma>0$ such that for every $\beta>0$,
every $(\beta,\gamma)$-continuous map into $(|P|,d')$ is $\delta$-close
with respect to $d'$ to a continuous map.
Now the $\frac\gamma3$-neighborhood $U$ of $|P|$ in $|N(D_0)|$ with respect to
the $d_\infty$ metric on $|N(D_0)|$ admits a discontinuous,
$(\gamma/3,\gamma)$-continuous retraction onto $(|P|,d')$.
Thus we obtain a continuous map $f\:U\to(|P|,d')$ such that $f|_{|P|}$ is
$\delta$-close to the identity.
Then by the above $f|_{|P|}$ is homotopic to the identity.

Let $n$ be such that $2^{-n}\le\gamma/3$.
Let $R_n=\fll ED(D_0)^{\#n}\flr_{N(D_0)^{\#n}}$.
Then $|R_n|$ lies in the $1$-neighborhood of $|ED(D_0)^{\#n}|$ with respect to
the $d_\infty$ metric on $|N(D_0)^{\#n}|$.
Hence the image of $|R_n|$ under the homeomorphism $|N(D_0)^{\#n}|\cong|N(D_0)|$
lies in the $2^{-n}$-neighborhood of $|P|$ with respect to the
$d_\infty$ metric on $|N(D_0)|$.
Thus we obtain a continuous map $g\:|R_n|\to|P|$ whose restriction over $P^{\#n}$
is homotopic to the homeomorphism $|P^{\#n}|\cong|P|$.
Next, each $D_{i+1}$ barycentrically refines $D_i$, so we have the canonical bonding map
$\phi^{D_{i+1}}_{D_i}\:N(D_{i+1})\to N(D_i)^\#$.
It is easy to see that the composition $P^{\#(n+1)}\simeq
ED(D_{i+1})\xr{\phi} ED(D_i)^{\#n}\simeq P^{\#n}$, where $\phi$ is the
restriction of $\phi^{D_{i+1}}_{D_i}$, is the identity map.
By iterating we obtain a monotone map $\phi_n\:N(D_n)\to N(D_0)^{\#n}$
that extends the composition $ED(D_n)\simeq P^{\#n}\simeq ED(D_0)^{\#n}$.
Since $\phi_n$ is monotone and $N(D_n)=\fll ED(D_n)\flr$, the image of
$\phi_n$ lies in $R_n$.
Thus we obtain a continuous map $h\:|N(D_n)|\to|P|$ whose restriction to
$|ED(D_n)|$ is homotopic to the homeomorphism $|P^{\#n}|\cong|P|$.
In particular, we obtain a continuous map $k\:|N(C_n)|\to|P_{2^n+1}|$
whose restriction to $|ED(C_n)|$ is homotopic to the homeomorphism
$|Q_n|\cong|P_{2^n+1}|$.
This yields a continuous retraction of the $(2^n+1)$-ball onto the
boundary $2^n$-sphere, which is a contradiction.
\end{example}

We note that the poset $P$ in Example \ref{counterexample3} satisfies the
following property (Q): There exist essential maps $e_n\:S^{2^n}\emb|P|$
such that for each $\eps>0$ there exists an $n$, a $\delta>0$ and
a discontinuous, $(\delta,\eps)$-continuous extension of $e_n$ over
$B^{2^n+1}$.
On the other hand, since $|P^\flat|$ is a uniform ANR, $|P|$ is a
non-uniform ANR, and in particular satisfies the non-uniform homotopy
extension property.
It follows that every metrizable uniform space that is uniformly homotopy
equivalent to $|P|$ satisfies (Q) as well.
In particular, using that $|P|$ is uniformly locally contractible, we get
the following

\begin{theorem}\label{Hahn-homotopy}
There exists a countable poset whose geometric realization is not uniformly
homotopy equivalent to a uniform ANR, nor even to a metrizable uniform space
satisfying the Hahn property.
\end{theorem}

The remainder of this subsection is not used elsewhere in this paper, and could be
of interest primarily to the reader who is looking for a class of posets larger
than conditionally complete posets whose geometric realizations are uniform ANRs.

\begin{definition}[Hereditarity]
We call a cover $D$ of a metric space $M$ {\it hereditarily uniform},
if there exists a $\lambda>0$ such that each $E\subset D$ is a uniform cover
of $\bigcup E$ with Lebesgue number $\lambda$.
Any such $\lambda$ is a {\it hereditary Lebesgue number} of
the hereditarily uniform cover.

We say that a cover $C$ of a set $S$ {\it hereditarily barycentrically refines}
a cover $D$ of $S$ if for every $E\subset D$, the cover $C\cap (\bigcup E)$
of the subset $\bigcup E\subset X$ barycentrically refines the cover $E$ of $\bigcup E$.

If $D$ is a cover of $M$ by sets of diameters $<\eps$, and $C$ is a hereditarily
uniform cover of $M$ with a hereditary Lebesgue number $2\eps$, then clearly $D$
hereditarily barycentrically refines $C$.
\end{definition}

\begin{lemma}\label{5.3} Let $C$ and $D$ be covers of a set $S$.
If $C$ hereditarily barycentrically refines $D$, then $\phi^C_D$ sends $N(C)$
into $ED(D)^\#$.
\end{lemma}

\begin{proof} The hypothesis implies that for every $x\in C$, every $T\subset
\st(x,C)$ such that $x\in T$ satisfies the following property $(*)$: if $T$ lies
in $\bigcup E$ for some $E\subset D$, then it lies in some element of $E$.
In particular, $(*)$ is satisfied by any $T$ of the form $\bigcap\sigma$
or $\bigcup\sigma$, where $\sigma\in N(C)$.
On the other hand, every element of $N(D)$ of the form $\Delta_D(T)$ where
$T$ satisfies $(*)$ clearly belongs to $ED(D)$.
\end{proof}

\begin{definition}[Construction of hereditary uniform covers]
For a finite-dimensional atomic poset $P$ with atom set $\Lambda$, it is easy to
construct a hereditarily uniform cover of $|P|$, namely the cover by the sets
$U_\lambda=\bigcup_{\sigma\ge\lambda}|H_\sigma|$ composed of the barycentric
handles $H_\sigma=\fll\left<\sigma\right>^*\flr\subset (P^\flat)^*$, where
$\sigma\in P$.
The hereditarity is due to the fact that $|H_\sigma|$ and $|H_\tau|$ are uniformly
disjoint when $\sigma$ and $\tau$ are incomparable and $P$ is
finite-dimensional.
This argument does not apply to canonical handles
$h_\sigma=\fll[\sigma,\sigma]^*\flr\subset (P^\#)^*$ because $h_\sigma\cap h_\tau$
can be nonempty when $\sigma$ and $\tau$ are incomparable.

Clearly, the preimage of a hereditarily uniform cover under a uniformly
continuous map of metrizable uniform spaces is hereditarily uniform.
Hence by using Lemma \ref{nerve map}, we infer that every uniform cover of
a uniformly finitistic metrizable uniform space admits a hereditarily
uniform refinement (in fact, one of a finite multiplicity).
We conjecture that the hypothesis of uniform finitisticity cannot
be dropped here.
\end{definition}

\begin{definition}[Weak hereditarity]
We call a cover $D$ of a metric space $M$ {\it weakly hereditarily uniform},
if there exists a $\lambda>0$ such that for every $F\subset D$
satisfying $\bigcap F\subset\bigcup (D\but F)$, the cover $(D\but F)\cap(\bigcap F)$
of the subset $\bigcap F\subset X$ is a uniform cover of $\bigcap F$ with
Lebesgue number $\lambda$.
Any such $\lambda$ is called a {\it weak hereditary Lebesgue number} of
the weakly hereditarily uniform cover.
A hereditarily uniform cover is weakly hereditarily uniform by considering
$E=D\but F$; and a weakly hereditarily uniform cover is uniform by considering
$F=\emptyset$ (in which case $\bigcap F=M$).

We say that a cover $C$ of a set $S$ {\it weakly hereditarily barycentrically refines}
a cover $D$ of $S$ if for every $F\subset D$ satisfying
$\bigcap F\subset\bigcup (D\but F)$, the cover $C\cap(\bigcap F)$ of the subset
$\bigcap F\subset X$ barycentrically refines the cover $(D\but F)\cap(\bigcap F)$ of
$\bigcap F$.
Similarly to the above,

\begin{center}
{\small hereditary bar.\ refinement \imp\ weak hereditary bar.\ refinement \imp\ bar.\ refinement.}
\end{center}

It is easy to see that if $D$ is a cover of $M$ by sets of diameters $<\eps$,
and $C$ is a weakly hereditarily uniform cover of $M$ with a weak hereditary Lebesgue
number $2\eps$, then $D$ weakly hereditarily barycentrically refines $C$.

Beware that the preimage of a weakly hereditarily uniform cover under a uniformly
continuous map $f$ of metrizable uniform spaces need not be weakly hereditarily
uniform, because $\bigcap F\not\subset\bigcup (D\but F)$ does not imply
$f^{-1}(\bigcap F)\not\subset f^{-1}(\bigcup (D\but F))$.
\end{definition}

The proof of Lemma \ref{5.3} works to establish

\begin{lemma}\label{5.3'} Let $C$ and $D$ be covers of a set $S$.
If $C$ weakly hereditarily barycentrically refines $D$, then $\phi^C_D$ sends $N(C)$ into
$ED(D)^\#$.
\end{lemma}

\section{Approximation of spaces}

\begin{theorem}\label{inverse limit} Every separable metrizable complete uniform
space $X$ is the limit of a convergent inverse sequence of geometric realizations of
simplicial complexes $P_i$ and uniformly continuous maps.

If $X$ is star-finite or uniformly finitistic, or both, then each $P_i$
can be chosen to be locally compact, or finite dimensional, or both (respectively).
\end{theorem}

Theorem \ref{inverse limit} follows from \cite[Theorem \ref{metr:A.7}]{M2}
along with the following

\begin{lemma}\label{6.1} Let $\{C_n\}$ be a basis of the uniformity of
a metrizable complete uniform space $X$, where each uniform cover $C_n$ is
countable and point-finite.
Let $N_i=|N(C_i)|$, and let $p_i\:N_{i+1}\to N_i$ be the composition
$|N(C_{i+1})|\xr{|\phi^{C_{i+1}}_{C_i}|}|N(C_i)^\#|\xr{h}|N(C_i)|$, where
$\phi^{C_{i+1}}_{C_i}$ is the canonical bonding map, and $h$ is the
uniform homeomorphism.
Then the inverse sequence $\dots\xr{p_1}N_1\xr{p_0}N_0$ is convergent,
and its limit $L$ is uniformly homeomorphic to $X$.
\end{lemma}

Though our notion of geometric realization and our bonding maps are quite
different from those used by Isbell, our proof of Lemma \ref{6.1} is
modelled on Isbell's proof of what amounts to its uniformly finitistic case 
\cite[Lemma V.33]{I3}.%
\footnote{There is a minor error in the proof of step (2) in
\cite[Lemma V.33]{I3}, as the Cauchy filter base considered there might
consist entirely of the empty sets.
This can be remedied as shown in the last paragraph of our proof.}

\begin{proof} Let $s_i(x)$ denote the simplex $\fll\Delta_{C_i}(x)\flr$.
Then each $f_i\bydef \phi^{C_{i+1}}_{C_i}$ sends $s_{i+1}(x)$ into $s_i(x)^\#$ for each
$x\in X$.
Hence each $p_i$ sends $|s_{i+1}(x)|$ into $|s_i(x)|$.
Since $f_i$ is monotone, $|f_i|$ is $1$-Lipschitz with respect to the $d_\infty$
metrics, and therefore $p_i$ is $\frac12$-Lipschitz with respect to the $d_\infty$
metrics.
Since the diameter of $|s_i(x)|$ is bounded above by $2$ in the $d_\infty$ metric,
the diameter of $p^{i+n}_i(|s_{i+n}(x)|)$ is bounded above by $2^{1-n}$.
Since each $|s_i(x)|$ is compact, their inverse limit (with the restrictions of
$p_i$ as the bonding maps) is nonempty, and by the above it has zero diameter.
Thus it is a single point $\lambda(x)\in L$.

Each $N_i$ is the union of the $|s_i(x)|$ over all $x\in X$, and every $|s_i(x)|$
contains $p^\infty_i(\lambda(x))$.
Hence every point of $p^{i+n}_i(N_{i+n})$ is $2^{1-n}$-close to a point of
$p^\infty_i(L)$.
Thus the inverse sequence is convergent (see
\cite[Lemma \ref{metr:A.12}(c)]{M2}).

To see that $\lambda\:X\to L$ is uniformly continuous, it suffices to show
that every its coordinate $\lambda_i\:X\xr{\lambda}L\xr{p^\infty_i}N_i$
is uniformly continuous.
Indeed, for each $x\in X$ and each $j\ge i$ we have
$\lambda_j(x)\in |s_j(x)|=|\fll\Delta_{C_j}(x)\flr|$.
For each $V\in C_j$, every $x\in V$ satisfies $V\in\Delta_{C_j}(x)$.
Hence $\lambda_j(V)\subset |\st(\{V\},N(C_j))|$.
The diameter of $|\st(\{V\},N(C_j))|$ is bounded above by $4$, hence its image
under $p^j_i$ has diameter at most $2^{2-(j-i)}$.
Thus for each $\eps>0$ there exists a $j\ge i$ such that
$\lambda_i(C_j)=p^j_i\lambda_j(C_j)$ refines the cover of $N_i$ by $\eps$-balls.
Since $\{C_j\}$ is a fundamental sequence of covers of $X$, we infer that
$\lambda_i$ is uniformly continuous.

Next, given $x,y\in X$ at a distance $\eps>0$, there exists an $n=n(\eps)$
such that any two elements of $C_n$ containing $x$ and $y$ are disjoint.
Then $\lambda_n$ sends $x$ and $y$ into disjoint closed simplices
$|s_n(x)|$ and $|s_n(y)|$ of $N_n$.
It follows that $\lambda$ is injective and, using the uniform continuity of
each $p^\infty_n$, that $\lambda^{-1}$ is uniformly continuous.

Finally, if $(q_i)\in L$ is a thread of $q_i\in N_i$, and
$\sigma_n$ is the minimal simplex of $N(C_n)$ such that $q_n\in|\sigma_n|$, then
$f_n(\sigma_{n+1})\subset\sigma_n^\#$, moreover, $\sigma_n$ is the minimal
simplex of $N(C_n)$ satisfying the latter property.
Hence $\sigma_n=\Delta_{C_n}(\bigcap\sigma_{n+1})$; in particular,
$\bigcap\sigma_{n+1}\subset\bigcap\sigma_n$.
Let $S_n$ be the closure of $\bigcap\sigma_n$.
Then $S_n$ lies in the closure of an element of $C_n$.
Since $\{C_i\}$ is a basis of the uniformity of $X$, for each $\eps>0$
there exists an $n$ such that every element of $C_n$ is of diameter
at most $\eps$.
It follows that the inverse sequence $\dots\subset S_1\subset S_0$ is Cauchy
(see \cite{M2}*{Lemma \ref{metr:A.12}(d)}).
Since $X$ is complete, so are the $S_i$'s, hence $\dots\subset S_1\subset S_0$ is
convergent (see \cite{M2}*{Lemma \ref{metr:A.12}(b)}) and therefore
$\bigcap S_i$ is nonempty (see \cite{M2}*{Lemma \ref{metr:A.12}(f)}).
Since the diameters of $S_i$ tend to zero, $\bigcap S_i$ must be a single
point $q$.
Now $q$ lies in the closure of $\bigcap\sigma_n$, and each
$q'\in\bigcap\sigma_n$ satisfies $\sigma_n\subset s_n(q')$ and
$\lambda_n(q')\in |s_n(q')|$.
Since $\lambda_n$ is continuous, $\lambda_n(q)$ lies in the closed subset
$|\fll\cel\sigma_n\cer\flr|$ of $|N(C_n)|$.
Hence $s_n(q)\subset\fll\cel\sigma_n\cer\flr$, or equivalently
$\sigma_n\subset\fll\cel s_n(q)\cer\flr$.
Since $\lambda(q)$ is also the inverse limit of the simplicial
neighborhoods $|\fll\cel s_n(q)\cer\flr|$, we conclude that
$\lambda(q)=(q_i)$.
Thus $\lambda$ is surjective.
\end{proof}

\begin{remark} Given simplicial complexes $K_i$ and monotone maps
$f_i\:K_{i+1}\to K_i^\#$, it is easy to see (by analyzing the definition
of a countable product of uniform spaces, see \cite{M2}) that a base of
uniformity of $\invlim|K_i|$ is given by the preimages of the covers of
$|K_i|$'s by the open stars of vertices.
It follows, for instance, that $\id\:\invlim|K_i|\to\invlim|K_i^{\#n}|$
is $2^n$-Lipschitz.
The best that can be said of the inverse is that it is $1$-Lipschitz.
\end{remark}

Similarly to \cite[proof of 7.2(ii)]{I2}, one deduces from Theorem
\ref{inverse limit} the following

\begin{corollary} \label{inverse limit2}
If $\{C_\alpha\}$ is a basis of uniform covers of a separable complete uniform
space $X$, where each $C_\alpha$ is countable and point-finite, then $X$ is
a (non-sequential) inverse limit of inverse limits of uniformly continuous maps
between $|N(C_\alpha)|$.
\end{corollary}

\begin{theorem} \label{inverse limit-general}
Every separable complete uniform space is the limit of an inverse spectrum of
uniform polyhedra and uniformly continuous maps.
\end{theorem}

The following argument is an elaboration on \cite[7.2(i)]{I2}.
It shows that the uniform polyhedra can be chosen to be geometric realizations
of cubical complexes, and more specifically cubohedra in the sense of \cite{M2}.

\begin{proof}
It is well-known that every uniform space $X$ embeds in an (uncountable) product
of complete metric spaces $M_i$ \cite[II.14, II.15]{I3}.
If $X$ is separable, we may assume that so is each $M_i$, by considering
the closures of the images of $X$ in the $M_i$'s.
Then by Theorem \ref{inverse limit}, each $M_i$ in turn embeds in a
product of uniform polyhedra.
Thus $X$ can be identified with a subspace of a product of uniform polyhedra
$|K_j|$.

Since $\prod |K_j|$ is the inverse limit of finite subproducts, and $X$ is closed
in $\prod |K_j|$, it is the inverse limit of its images in the finite
subproducts, cf.\ \cite[IV.34]{I3}.
Each finite subproduct is a uniform polyhedron $|K_{j_1}\x\dots\x K_{j_r}|$.
Let $P_{n;\,j_1\dots j_r}$ be the minimal subcomplex of
$(K_{j_1}\x\dots\x K_{j_r})^{\#n}$ such that $|P_{n;\,j_1\dots j_r}|$ contains
the image of $X$.
Then it is shown similarly to \cite[proof of Theorem 10.1]{ES} that $X$
is the inverse limit of all the $|P_{n;\,j_1\dots j_r}|$.
\end{proof}

\begin{theorem}\label{domination} A separable metric space is
a uniform ANR if and only if it is uniformly $\eps$-homotopy dominated by
the geometric realization of a simplicial complex, for each $\eps>0$.
\end{theorem}

The ``if'' direction follows from \cite[Corollary \ref{metr:Hanner}]{M2}.
The following proof of the ``only if'' direction is based on
Theorem \ref{inverse limit}; the reader who feels that this is an overkill
can get a more elementary argument by combining the first sentence of this
proof with the proof of Theorem \ref{domination-general} below.

\begin{proof}[Proof, ``only if'']
By \cite[Theorem \ref{metr:uniform ANR}]{M2}, the given
uniform ANR is uniformly $\eps$-homotopy equivalent to its completion $X$,
which is still a uniform ANR, for each $\eps>0$.
By Theorem \ref{inverse limit}, $X$ is the limit of a convergent inverse sequence
of geometric realizations $P_i$ of simplicial complexes, and uniformly continuous
bonding maps $p_i$.
By \cite[Corollary \ref{metr:telescope retraction}]{M2} there exists a $k$
and a uniformly continuous retraction $r_{[k,\infty]}\:P_{[k,\infty]}\to X$.
For each $l\ge k$ let $r_l$ and $r_{[l,\infty]}$ denote the restrictions
of $r_{[k,\infty]}$ over $P_l$ and over $P_{[l,\infty]}$, respectively.
Let $p^\infty_{[k,\infty]}\:X\x I\to P_{[k,\infty]}$ be obtained by combining
the maps $p^\infty_i\:X\to P_i$.
Then for each $l\ge k$, the composition
$r_{[l,\infty]}p^\infty_{[l,\infty]}\:X\x I\to X$ is a uniformly continuous
homotopy between $X\xr{p^\infty_l}P_k\xr{r_l}X$ and $\id_X$.
Moreover, for each $\eps>0$ there exists a $k$ such that
$r_{[l,\infty]}p^\infty_{[l,\infty]}$ is an $\eps$-homotopy.
Thus $X$ is uniformly $\eps$-homotopy dominated by $P_l$.
\end{proof}

\begin{theorem}\label{domination-general} Every separable ANRU $X$ is uniformly
$C$-homotopy dominated by the geometric realization of a simplicial complex,
for each uniform cover $C$ of $X$.
\end{theorem}

The following proof elaborates on \cite[proof of 7.3]{I2}.

\begin{proof}
Since $X$ is an ANRU it is uniformly locally contractible, i.e.\ for every
uniform cover $C$ of $X$ there exists a uniform cover $C_1$ of $X$ such that
for every uniform space $Y$, every two $C_1$-close maps $Y\to X$ are uniformly
$C$-homotopic \cite[proof of 4.2]{I2}.
Let $C_2$ be a uniform barycentric refinement of $C_1$.
Since $X$ satisfies the Hahn property \cite[4.1]{I2}, there exists a uniform
cover $C_3$ of $X$ such that for any uniform space $Y$, any map
$\phi\:Y\to X$ such that $\phi^{-1}(C_3)$ is uniform is $C_2$-close to
a uniformly continuous map.
Let $C_4$ be a uniform star-refinement of $C_3$.

By \cite[Theorem I.14]{I3} there exists a uniformly continuous map $f$ from $X$ onto
a metric space $M$ such that $C_4$ is refined by $f^{-1}(D)$, where $D$ is
the cover of $M$ by all sets of diameter $<1$.
Since $X$ is separable and $f$ is surjective, $M$ is separable.
Then $D$ is refined by a point-finite countable uniform cover $E$ (see
\cite[Theorem \ref{metr:A.7}]{M2}).
Without loss of generality $E$ is open.
Then by Lemma \ref{nerve map} there exists a uniformly continuous map
$g\:M\to |N(E)|$ such that $E=f^{-1}(F)$, where $F$ is the cover of
$|N(E)|$ by the open stars of vertices.
Let us define a discontinuous map $\phi\:|N(E)|\to X$ by sending an $x\in|N(E)|$
to any point in $f^{-1}(g^{-1}(U))$, where $U$ is any element of $F$
containing $x$.
Note that $g^{-1}(U)$ is nonempty, being an element of $E$; hence
$f^{-1}(g^{-1}(U))$ is nonempty, due to the surjectivity of $f$.
Since $f^{-1}(g^{-1}(U))$ is contained in some $V_U\in C_4$,
we get that $\phi(F)$ refines $C_3$ and the composition
$X\xr{f}M\xr{g}|N(E)|\xr{\phi}X$ is $C_4$-close to the identity.
Then there exists a uniformly continuous map $\psi\:|N(E)|\to X$ that is
$C_2$-close to $\phi$.
Hence $\psi$ is $C_1$-close to the identity, and therefore $C$-homotopic to it.
\end{proof}

The following improves on \cite[Theorems \ref{metr:approximate polyhedron}
and \ref{metr:approximate cubohedron}]{M2}.

\begin{theorem}
A separable uniform space $X$ satisfies the Hahn property if and only if for each
uniform cover $C$ of $X$ there exists a simplicial complex $K$ and uniformly
continuous maps $X\xr{f}|K|\xr{g}X$ whose composition is $C$-close to $\id_X$.
\end{theorem}

\begin{proof}
The ``only if'' assertion can be proved similarly to Theorem \ref{domination-general}
(or, in the metrizable case, to Theorem \ref{domination}).
The ``if'' assertion is easy (cf.\ \cite{M2}).
\end{proof}

\begin{theorem}\label{Mather trick}
If $X$ is a uniform ANR, then $X\x\R$ is uniformly homotopy equivalent to
a uniform polyhedron.
\end{theorem}

\begin{proof}
By Theorem \ref{domination} we are given uniformly continuous maps $d\:|K|\to X$ and
$u\:X\to |K|$, where $K$ is a simplicial complex, such that the composition
$X\xr{u}|K|\xr{d}X$ uniformly homotopic to the identity by an homotopy
$h\:X\x I\to X$.
We now perform a uniform version of Mather's trick (see \cite{FR}): $X\x\R$ is
uniformly homotopy equivalent to the double mapping telescope of
$$\dots\to X\xr{du}X\xr{du}X\to\dots,$$
which is in turn uniformly homotopy equivalent to the double mapping telescope of
$$\dots\to X\xr{d}|K|\xr{u}X\xr{d}|K|\xr{u}X\to\dots,$$
which is in turn uniformly homotopy equivalent to the double mapping telescope of
$$\dots\to |K|\xr{ud}|K|\xr{ud}|K|\to\dots.$$
Since $|K|$ is a uniform ANR, by \cite[Lemma \ref{metr:LCU}]{M2} it is
uniformly locally contractible.
Then by Theorem \ref{monotone approximation}, $ud$ is uniformly homotopic to
the composition $|K|\xr{h_m^{-1}}|K^{\#m}|\xr{|f|}|K^{\#n}|\xr{h_n}|K|$ for some
monotone map $f\:K^{\#m}\to K^{\#n}$, where $m\ge n$ for the sake of definiteness.
On the other hand, we have the monotone map $\#(m-n)\:K^{\#m}\to K^{\#n}$, whose
geometric realization is uniformly homotopic to the uniform homeomorphism
$h_m\:|K^{\#m}|\to|K^{\#n}|$ (see Lemma \ref{canonical-MC}).
Thus $X\x\R$ is uniformly homotopy equivalent to the geometric realization of
the double mapping telescope of
$$\dots\xl{\#(m-n)}K^{\#n}\xr{f}K^{\#m}\xl{\#(m-n)}K^{\#n}\xr{f}\dots.$$
By Theorem \ref{tmc2}, the latter is uniformly homotopy equivalent to
the geometric realization of the thickened double mapping telescope, which
is a uniform polyhedron (using \cite[Lemma \ref{comb:CCP amalgam}]{M1}).
\end{proof}

\begin{remark}
Similar arguments (with double mapping telescopes not of individual nerves
but of their mapping telescopes) also show that if $X$ is a complete uniform ANR,
then $X\x\R$ is the limit of a convergent inverse sequence
$\dots\xr{q_1}Q_1\xr{q_0}Q_0$ of geometric realizations of countable preposets
and uniformly continuous maps such that each $q_i$ is a uniform homotopy equivalence.
\end{remark}

\begin{theorem}\label{ANR-limit} Every complete uniform ANR is the limit of
a convergent inverse sequence $\dots\xr{q_1}Q_1\xr{q_0}Q_0$ of uniform polyhedra
and uniformly continuous maps such that each $q_i$ is (non-uniformly) a homotopy
equivalence.
\end{theorem}

\begin{proof}
Let $X$ be the given complete uniform ANR.
By Theorem \ref{inverse limit}, $X$ is the limit of a convergent inverse
sequence $\dots\xr{p_1}P_1\xr{p_0}P_0$ of uniform polyhedra and
uniformly continuous maps.
Suppose that we have constructed a finite sequence $n_0,n_1,n_2,\dots,n_k$
and a finite chain of uniform polyhedra and uniformly continuous maps
$P_{n_k}\subset Q_k\xr{q_{n-1}}\dots\xr{q_1}
P_{n_1}\subset Q_1\xr{q_0} P_{n_0}\subset Q_0$
such that the composition $X\xr{p^\infty_{n_i}}P_{n_i}\subset Q_i$ is
a homotopy equivalence for each $i\le k$.

Since $P_{n_k}$ is uniformly locally contractible, there exists an $\eps>0$
such that every two $\eps$-close uniformly continuous maps $Y\to P_{n_k}$
are uniformly homotopic.
Let $\delta$ be such that $p^\infty_i$ sends $\delta$-close points into
$\eps/2$-close points.
Since the inverse sequence is convergent and $X$ satisfies the Hahn property
(see \cite[Lemma \ref{metr:Hahn}(a)]{M2}), there exists an $m\ge n$ and
a map $r\:P_m\to X$ such that the composition $X\xr{p^\infty_m}P_m\xr{r}X$ is
$\delta$-close to the identity.
Then the composition $X\xr{p^\infty_m}P_m\xr{r}X\xr{p^\infty_{n_k}}P_{n_k}$
is $\eps/2$-close to $p^\infty_{n_k}$.
Since the inverse sequence is convergent, there exists an $l\ge m$ such
that the composition $P^l\xr{p^l_m}P_m\xr{r}X\xr{p^\infty_{n_k}}P_{n_k}$
is $\eps$-close to $p^l_{n_k}$.
Let $d$ be the composition $P_l\xr{p^l_m}P_m\xr{r}X$, let $u=p^\infty_l$,
and let $f$ be the composition $P_l\xr{d}X\xr{u}P_l$.
Then the composition $P_l\xr{f}P_l\xr{p^l_{n_k}}P_{n_k}$ is uniformly homotopic
to $p^l_{n_k}$.
This yields a uniformly continuous map $F\:MC(f)\to P_{n_k}$ that restricts to
$p^l_{n_k}$ on each of the two copies of $P_l$ in $MC(f)$.
Let $n_{k+1}=l$.
Applying to the maps $X\xr{u}P_l\xr{d}X$ the construction in the proof of
Theorem \ref{Mather trick}, we obtain the double mapping telescope $Q_l$
homotopy equivalent to $X$, via the composition $X\xr{u}P_l\subset Q_{k+1}$.
The partial map $Q_{k+1}\supset P_l\xr{p^l_{n_k}}P_{n_k}$ now extends (using $F$)
to a total uniformly continuous map
$q_{k+1}\:Q_{k+1}\to P_{n_k}$, and we are done with the inductive step.

The assertion now follows using that inverse limit is unchanged upon passage to
an infinite subsequence.
\end{proof}

\subsection*{Acknowledgments}

I would like to thank P. M. Akhmetiev, N. Brodsky, A. N. Dranishnikov, J. Dydak,
S. Illman, E. V. Shchepin and M. Skopenkov for useful discussions.

\subsection*{Disclaimer}

I oppose all wars, including those wars that are initiated by governments at the time when 
they directly or indirectly support my research. The latter type of wars include all wars 
waged by the Russian state in the last 25 years (in Chechnya, Georgia, Syria and Ukraine) 
as well as the USA-led invasions of Afghanistan and Iraq.

\end{document}